\documentclass[11pt]{article}

\usepackage[a4paper, total={6.7in, 10.7in}]{geometry}
\usepackage[utf8]{inputenc}
\usepackage{xcolor, comment, enumerate, caption, subcaption, tcolorbox, multicol, multirow, array}
\usepackage{algorithm, algpseudocode}
\usepackage{amssymb, amsmath, amsthm, bbm}
\usepackage[sort,nocompress]{cite}
\usepackage{tikz}
\usetikzlibrary{matrix}

\usepackage{hyperref}
\hypersetup{
    colorlinks=true,
    allcolors=blue
    }

\usepackage[capitalise]{cleveref}
\crefname{equation}{}{}
\crefname{figure}{Figure}{Figures}
\creflabelformat{equation}{\textup{(#2#1#3)}}
\crefname{assumption}{Assumption}{Assumptions}
\crefname{condition}{Condition}{Conditions}
\crefname{property}{Property}{Properties}

\usepackage{enumitem}

\usepackage{pifont}
\newcommand{\cmark}{\ding{51}}%
\newcommand{\xmark}{\ding{55}}%
\crefname{equation}{}{}
\crefname{figure}{Figure}{Figures}
\creflabelformat{equation}{\textup{(#2#1#3)}}
\crefname{assumption}{Assumption}{Assumptions}
\crefname{condition}{Condition}{Conditions}
\crefname{property}{Property}{Properties}
\crefname{fact}{Fact}{Facts}

\newtheorem{fact}{Fact}
\newtheorem{theorem}{Theorem}
\newtheorem{definition}{Definition}

\newtheorem{lemma}{Lemma}
\newtheorem{assumption}{Assumption}

\newtheorem*{corollary*}{Corollary}

\newtheorem{remark}{Remark}
\newtheorem{property}{Property}
\usepackage{amsthm}
\usepackage[framemethod=TikZ]{mdframed}

\newcommand{\f}[1]{\mathbf{#1}}
\newcommand{\ft}[1]{\tilde{\mathbf{#1}}}

\newcommand{\lr}[1]{\left\langle #1 \right\rangle}
\newcommand{\krylov}[3]{\mathcal{K}_{#3}(\f{#1}, \f{#2})}
\newcommand{\Rdd}{\mathbb{R}^{d \times d}}
\newcommand{\Rd}{\mathbb{R}^{d}}

\newcommand{\zero}{\f{0}}
\newcommand{\argmin}[1]{\underset{#1}{\text{argmin}}}
\newcommand{\Range}{\text{Range}}
\newcommand{\Null}{\text{Null}}
\newcommand{\Span}{\text{Span}}

\newcommand*\vnorm[1]{\left\| #1\right\|}

\definecolor{uqp}{RGB}{152,24,147}

\makeatletter
\newcommand*{\transpose}{%
	{\mathpalette\@transpose{}}%
}
\newcommand*{\@transpose}[2]{%
	\raisebox{\depth}{$\m@th#1\intercal$}%
}
\makeatother

\newcommand*{\T}{{\transpose}}

\usepackage{xspace}

\renewcommand\th{\textsuperscript{th}\xspace}

\newcommand{\hf}{\frac{1}{2}}

\newcommand{\real}{\mathbb{R}}

\newcommand{\red}[1]{{\leavevmode\color{red}#1}}
\newcommand{\blue}[1]{{\leavevmode\color{blue}#1}}
\definecolor{forestgreen}{rgb}{0.13, 0.55, 0.13}

\title{Conjugate Direction Methods Under Inconsistent Systems}
\author{Alexander Lim\footnote{School of Mathematics and Physics, University of Queensland, Australia. Email : alexander.lim@uq.edu.au} \and Yang Liu \footnote{Mathematical Institute, University of Oxford, UK. Email: yang.liu@maths.ox.ac.uk} \and Fred Roosta\footnote{School of Mathematics and Physics, University of Queensland, Australia. Email : fred.roosta@uq.edu.au}}
\date{\today}

\begin{document}
	\maketitle
	\begin{abstract}
		Since the development of the conjugate gradient (CG) method in 1952 by Hestenes and Stiefel, CG, has become an indispensable tool in computational mathematics for solving positive definite linear systems. On the other hand, the conjugate residual (CR) method, closely related CG and introduced by Stiefel in 1955 for the same settings, remains relatively less known outside the numerical linear algebra community. 
		Since their inception, these methods -- henceforth collectively referred to as conjugate direction methods -- have been extended beyond positive definite to indefinite, albeit consistent, settings. Going one step further, in this paper, we investigate theoretical and empirical properties of these methods under inconsistent systems. Among other things, we show that small modifications to the original algorithms allow for the pseudo-inverse solution. Furthermore, we show that CR is essentially equivalent to the minimum residual method, proposed by Paige and Saunders in 1975, in such contexts. 
		Lastly, we conduct a series of numerical experiments to shed lights on their numerical stability (or lack thereof) and their performance for inconsistent systems. Surprisingly, we will demonstrate that, unlike CR and contrary to popular belief, CG can exhibit significant numerical instability, bordering on catastrophe in some instances.
	\end{abstract}

    \section{Introduction}
    Consider the following linear system,
\begin{equation}\label{eq:linear_system}
	\f{Ax} = \f{b},
\end{equation}
where $\f{A} \in \Rdd$ is a symmetric matrix and $\f{b} \in \Rd$. In 1952, Hestenes and Stiefel \cite{hestenes1952methods} developed the celebrated conjugate gradient (CG) method to solve \cref{eq:linear_system} when $\f{A}$ is positive definite (PD). Since then, CG has become a cornerstone and essential tool in many areas of applied mathematics, statistics, machine learning, and numerical optimization. For example, in solving for numerical solutions of partial differential equations, the discretization of a differential operator often leads to a positive definite matrix, making CG the go-to method. In statistical modeling and machine learning, CG is frequently the primary solver in methods involving kernel regression. In numerical optimization, the classical Newton-CG method \cite{nocedal1999numerical} has been the canonical example of truncated Newton-type methods for large-scale problems. 
A closely related technique to CG is the conjugate residual (CR) method, which was developed by Stiefel in 1955 \cite{stiefel1955relaxationsmethoden} for the same PD settings\footnote{In fact, without explicitly mentioning the name, the structure of the method is developed in the original paper of Hestenes and Stiefel \cite{hestenes1952methods}.}. However, despite its profound connections to CG, the CR method remains relatively less known outside the numerical linear algebra community. For example, to our knowledge, the only attempt at employing CR within Newton-type optimization methods is Dahito et al \cite{dahito2019conjugate}. 
	
Both these methods, henceforth collectively referred to as conjugate direction (CD) methods\footnote{Precisely speaking, the conjugate direction approach serves as a framework for deriving various methods, including CG and CR. However, in this paper, when we refer to CD, we specifically mean these two particular instances.}, are examples of Krylov subspace algorithms for solving \cref{eq:linear_system} when $ \f{A} $ is PD. 
However, advancements in science and engineering have led to a growing need for iterative solvers that can effectively address a broader range of challenges beyond PD systems. This includes scenarios where the matrix $ \f{A} $ is symmetric, but singular and/or indefinite, as well as inconsistent settings, namely $\f{b} \not\in \Range(\f{A})$, e.g., \cite{yoon2016solving,royer2020newton,roosta2022newton,mardal2011preconditioning}.
As a result, the popularity of CD, especially CG, to solve problems under these conditions continues to extend. 
There are studies \cite{fletcher2006conjugate,luenberger1969hyperbolic,luenberger1970conjugate,fasano2005planar,fasano2005planar_app,fasano2007lanczos,greenbaum1997iterative,sogabe2022krylov} providing comprehensive treatment of CD beyond positive definite systems. However, all these results operate under the assumption that the system is consistent, i.e. $ \f{b} \in \Range(\f{A}) $. 
	
In contrast, the minimum residual\footnote{We would like to caution the reader that Algorithm 5.3 in the book \cite{saad2003iterative} is called Minimal Residual Iteration which is entirely unrelated to MINRES.}
(MINRES) method by Paige and Saunders \cite{paige1975solution} extends naturally to all systems involving symmetric matrices, without relying on a consistency assumption and has shown great success in numerical optimization in recent years, e.g., \cite{roosta2022newton,lim2023complexity,liu2022newton,liu2022minres}. In comparison to CG or CR, however, MINRES is a significantly more complex algorithm and entails more intricate implementations. Therefore, understanding the properties and the behavior of the CD methods in inconsistent settings can be of great interest.
Before delving any deeper, it is essential to note that our objective in this paper is not to devise novel variants of CG or CR that can extend to inconsistent systems. Instead, our emphasis is on exploring the theoretical and empirical behavior of the original methods in such situations.

\subsubsection*{Contributions} For symmetric, but inconsistent, systems, the contributions of this paper can be informally summarized as follows:  
\begin{enumerate}[label = (\roman*)]
	\item We first show that the classical CG is fundamentally unable to recover any solution of the normal equation (\cref{thm:cg:Ar}). Subsequently, we provide a modification (\cref{alg:cg_pis}) that ensures convergence to the pseudo-inverse solution (\cref{thm:cg:pseudo}).
	\item We provide a variant of CR (\cref{alg:cr_pis}) that converges to the pseudo-inverse solution (\cref{thm:cr:pseudo}). We also establish that, in essence, CR and MINRES are equivalent (\cref{thm:mr=cr}).
	\item We numerically demonstrate that, beyond typical finite precision errors, CG--unlike CR--can exhibit borderline catastrophic instability in inconsistent settings. These instabilities can persist even for PD matrices with a few small eigenvalues. 
\end{enumerate}
	
\cref{table:cgcr} highlights some of the similarities and differences between vanilla CG and CR at termination.

\begin{table}[htbp]
	\centering
	\begin{tabular}{|c|wc{8em}|wc{8em}|wc{8em}|wc{8em}|}
		\hline
		& \multicolumn{2}{c|}{Consistent: $\f{b} \in \Range(\f{A})$} & \multicolumn{2}{c|}{Inconsistent: $\f{b} \notin \Range(\f{A})$}\\
		\cline{2-5}
		& CG & CR & CG & CR\\
		\hline
    	\multicolumn{1}{|c|}{\multirow{2}{*}{$\f{x} = \f{x}^\star$}} & \multicolumn{2}{c|}{} & \multirow{2}{*}{\red{\xmark}} & \blue{\cmark} \\
        & \multicolumn{2}{c|}{} & & (\cref{thm:cr:Ar})\\
        [3pt]
        \cline{1-1}\cline{4-5}
		\multicolumn{1}{|c|}{\multirow{4}{*}{$\f{x} = \f{x}^{+}$}} & \multicolumn{2}{c|}{} & \multirow{2}{*}{\red{\xmark}} & \multirow{2}{*}{\red{\xmark}}\\
        & \multicolumn{2}{c|}{} & &\\
        \cline{4-5}
        & \multicolumn{2}{c|}{\blue{\cmark}} & \blue{\cmark} & \blue{\cmark}\\
        & \multicolumn{2}{c|}{(\cref{fact:convergence_consistent})} & (w.\ \cref{alg:cg_pis}) & (w.\ \cref{alg:cr_pis})\\
        [3pt]
		\cline{1-1}\cline{4-5}
		\multicolumn{1}{|c|}{\multirow{2}{*}{$\f{r} = \zero$}} & \multicolumn{2}{c|}{} & \multirow{2}{*}{\red{\xmark}} & \multirow{2}{*}{\red{\xmark}} \\
        & \multicolumn{2}{c|}{} & & \\
		[3pt]
		\cline{1-1}\cline{4-5}
		\multicolumn{1}{|c|}{\multirow{2}{*}{$\f{Ar}  = \zero$}} & \multicolumn{2}{c|}{} & \red{\xmark} & \blue{\cmark}\\ 
        & \multicolumn{2}{c|}{} & (\cref{thm:cg:Ar}) & (\cref{thm:cr:Ar}) \\
        [3pt]
		\hline
		\multicolumn{1}{|c|}{\multirow{2}{*}{$\f{p}  = \zero$}} & \multicolumn{2}{c|}{} & \red{\xmark} & \red{\xmark} \\ 
        & \multicolumn{2}{c|}{\blue{\cmark}} & (\cref{lem:cg:Ap}) & (\cref{lem:cr:Ap}) \\
        [3pt]
		\cline{1-1}\cline{4-5}
		\multicolumn{1}{|c|}{\multirow{2}{*}{$\f{Ap}  = \zero$}} & \multicolumn{2}{c|}{(\cref{fact:pg=0_consistent})} & \blue{\cmark} & \blue{\cmark} \\ 
        & \multicolumn{2}{c|}{} & (\cref{lem:cg:Ap}) & (\cref{lem:cr:Ap})\\
        [3pt]
		\hline
	\end{tabular}
	\caption{Similarities and differences between CG (\cref{alg:cg}) and CR (\cref{alg:cr}) at termination. Here, $\f{x}$ is the final iterate, $\f{r} = \f{b} - \f{Ax}$ is the residual at $\f{x}$, $\f{p}$ is the update direction within the CD methods, $\f{x}^{+}$ and $\f{x}^\star$ denote the pseudo-inverse solution and a normal solution with $\f{y} \neq \f{0}$ from \cref{def:solution}, respectively. All quantities are at the termination of the respective algorithm. \label{table:cgcr}}
\end{table}

The rest of this paper is organized as follows. We end this section by introducing the notation and definitions that are used throughout this work. We first review the CD methods for consistent systems in \cref{sec:review}. Subsequently, we shift our focus to the central theme of our paper and present theoretical results for inconsistent systems in \cref{sec:inconsistent_system}. Finally, in \cref{sec:num_exp}, we provide numerical experiments to verify our theoretical results as well as to shed light on the numerical performance of CD methods in various situations.
	
	\subsubsection*{Notations and Definitions} We denote scalars, vectors, and matrices as lower-case, bold lower-case, and bold upper-case letters, e.g., $\alpha$, $\f{b}$ and $\f{A}$, respectively. 
    For two vectors $ \f{a} \in \Rd $ and $ \f{b}\in \Rd $, their Euclidean inner product is denoted as $\lr{\f{a}, \f{b}} = \f{a}^{\top}\f{b} $. 
    Similarly, for a symmetric matrix $\f{A} \in \Rdd$, the product $\f{a}^{\top}\f{Ab}$ (not necessarily an inner-product) is denoted by $\lr{\f{a},\f{b}}_{\f{A}}$. 
	The norm $\vnorm{\f{a}} = \sqrt{\lr{\f{a},\f{a}}}$ is the Euclidean norm, and  for a PD matrix $\f{A}$, we define the energy norm $\|\f{a}\|_{\f{A}} = \sqrt{\lr{\f{a},\f{A}\f{a}}}$. We write $\f{A} \succeq \f{0}$ and $\f{A} \succ \f{0}$ to indicate $\f{A}$ is positive semi-definite (PSD) and PD, respectively. 
	The Moore-Penrose pseudo-inverse of $\f{A}$ is denoted by $ \f{A}^\dagger $.	

	For a symmetric matrix $ \f{A} $ and $\f{b} \not\in \Range(\f{A})$, we denote the orthogonal decomposition of $ \f{b} $ with respect to $ \f{A} $ as
	\begin{align}
		\label{eq:orth_decomp}
		\f{b} \triangleq \ft{b} + \f{b}^\perp, \quad \text{where} \quad  \ft{b} \in \Range(\f{A}), \quad \text{and} \quad \f{b}^\perp \in \Null(\f{A}).
	\end{align} 
	So, for example, for consistent systems, we have $\f{b}^\perp = \f{0}$. By CG$(\f{A},\f{b})$, CR$(\f{A},\f{b})$ and MINRES$(\f{A},\f{b})$ we mean using CG, CR and MINRES to solve for a least-squares problem or a linear system with $\f{A}$ and $\f{b}$, respectively. 
	Furthermore, under inconsistent systems, we will often compare between two settings, say, CG$(\f{A}, \f{b})$ and its corresponding consistent case CG$(\f{A}, \ft{b})$, where $ \ft{b} $ is as in the orthogonal decomposition of $ \f{b} $ above. We denote $\alpha_{k-1}$, $\f{r}_{k-1}$, etc, to be the terms generated for solving the former system and use $\tilde{\alpha}_{k-1}$, $\ft{r}_{k-1}$ etc, to denote the terms generated for solving the latter system. The residual vector at iteration $k$ is denoted by $\f{r}_k = \f{b} - \f{Ax}_k$. 
	
	\begin{definition}[Normal Solutions and the Pseudo-inverse Solution]
		\label{def:solution}
		Any solution to the normal equation $\f{A}^{\top}\f{A}\f{x} = \f{A}^{\top}\f{b}$, namely $\f{x}^{\star} \triangleq \f{A}^\dagger\f{b} + (\f{I} - \f{A}^\dagger\f{A})\f{y}$ for some $\f{y} \in \mathbb{R}^d$, is referred to as a \textit{normal solution}. In particular, the \textit{pseudo-inverse solution} is given by $\f{x}^{+} \triangleq \f{A}^\dagger\f{b}$, i.e., when $\f{y} = \f{0}$.
	\end{definition}
    The Krylov subspace of degree $k \geq 1$, generated by $ \f{A} $ and $ \f{v} $, is denoted as 
    \begin{align*}
        \mathcal{K}_k(\f{A}, \f{v}) = \Span\{\f{v}, \f{Av}, \cdots, \f{A}^{k-2}\f{v}, \f{A}^{k-1}\f{v}\}.
    \end{align*}
    Recall that the grade of $\f{v}$ with respect to $\f{A}$ is defined as follows \cite{bjorck2015numerical}.
	\begin{definition}[Grade of $\f{v}$ with respect to $\f{A}$]\label{def:grades}
		The grade of $\f{v}$ with respect to $\f{A}$ is the positive integer $g \in \mathbb{N}$ such that
		\begin{align*}
			\dim(\krylov{A}{v}{k}) = \begin{cases}
				k, & k \leq g,\\
				g, & k > g.
			\end{cases}
		\end{align*}
	\end{definition}
	
	\begin{remark}[Initialization]
		\label{rem:ini}
		In this paper, we assume $\f{x}_0 = \f{0}$. Thus, the initial residual and the Krylov subspace become $\f{r}_0 = \f{b} - \f{Ax}_0 = \f{b}$ and $\krylov{A}{b}{k} = \Span(\f{b}, \cdots, \f{A}^{k-1}\f{b})$ respectively. Since extensions to arbitrary initialization are straightforward, we exclude them for the sake of brevity.
	\end{remark}
	
	\begin{remark}[Symmetric Matrix and Non-zero Right-hand side]
		Throughout the paper, for simplicity, we consider systems involving real symmetric matrices. We note that it is straightforward to adjust all statements for Hermitian matrices. Furthermore, we always assume $ \f{b} \neq \zero $ to avoid triviality.
	\end{remark}
	
	For indefinite matrices, we may encounter a non-zero vector $\f{x} \notin \Null(\f{A})$ for which $\lr{\f{x}, \f{Ax}} = 0$, i.e., $ \f{x} \perp \f{Ax}$. Such a vector is henceforth referred to as a zero-curvature direction.
	\begin{definition}[Zero-curvature Direction]
		\label{def:zero_curv}
		For a given indefinite matrix $ \f{A} $, a non-zero vector $\f{x} \notin \Null(\f{A})$ is a zero-curvature direction if $\lr{\f{x}, \f{Ax}} = 0$. 
	\end{definition}
 	\section{Consistent Systems: Review}\label{sec:review}
	Conjugate direction methods have been extensively studied in the literature for consistent settings where $ \f{b} \in \Range(\f{A}) $. In this section, we offer a concise overview of these methods in this context and present some results that, while potentially well-known, may not be readily available in a unified source in the literature.
	
	For solving a linear system \cref{eq:linear_system} with a positive definite matrix $ \f{A} \succ \zero $, the $ k\th $ iteration of CG (\cref{alg:cg}) can be characterized as the minimizer of the quadratic optimization problem
	\begin{align}\label{eq:cg_quadratic}
		\f{x}_k = \argmin{\f{x}\in\krylov{A}{b}{k}} \; \hf\lr{\f{x},\f{A}\f{x}} - \lr{\f{x}, \f{b}}.
	\end{align}
	Equivalently, one can derive $ \f{x}_k $ by enforcing the Petrov-Galerkin condition \cite{bjorck2015numerical}, as the vector in the search space $ \krylov{A}{b}{k} $ whose residual $\f{r}_{k}$ is orthogonal to a constraint space, which in the case of CG is taken to be the same as $\krylov{A}{b}{k}$. When $ k = g $, i.e., the grade of $ \f{b} $ with respect to $ \f{A} $, the solution that minimizes \cref{eq:cg_quadratic} is also the solution to \cref{eq:linear_system}.
	
	In the same setting, CR (\cref{alg:cr}) can be driven from CG by replacing the Euclidean inner-product $ \lr{.,.} $ with the one involving $ \f{A} \succ \zero $ as $ \lr{.,.}_{\f{A}} $. In this light, the $k\th$ iteration of CR amounts to the solution of a different quadratic optimization problem as
	\begin{align}\label{eq:cr_quadratic}
		\f{x}_k = \argmin{\f{x}\in\krylov{A}{b}{k}} \; \hf\lr{\f{x},\f{A}\f{x}}_{\f{A}} - \lr{\f{x}, \f{b}}_{\f{A}}.
	\end{align}
	Consequently, the Petrov-Galerkin condition here involves the same search space as that of CG, namely $ \krylov{A}{b}{k} $, but the constraint space is now $\f{A} \krylov{A}{b}{k}$. Furthermore, it is easy to see that \cref{eq:cr_quadratic} can be formulated as 
	\begin{align}\label{eq:minres_least_squares}
		\f{x}_k = \argmin{\f{x}\in\krylov{A}{b}{k}} \; \hf\|\f{b} - \f{Ax}\|^{2},
	\end{align}
	highlighting the well-known equivalence between CR and MINRES for when $ \f{A} \succ \zero $ \cite{fong2012cg}. We remind that while both CG and CR methods involve one matrix-vector product in each iteration, CR requires one more vector of storage and one more vector update compared with CG. Indeed, in \cref{alg:cr}, after computing $\f{Ar}_k$, one can calculate $ \f{Ap}_k $ as $\f{Ap}_k = \f{Ar}_k + \beta_k\f{Ap}_{k-1}$; see \cite{bjorck2015numerical} for more details.\par
    
    \noindent\begin{minipage}{0.49\textwidth}
    \begin{algorithm}[H]
    \caption{Conjugate Gradient}\label{alg:cg}
	\begin{algorithmic}[1]
		\Require $\f{A}$ and $\f{b}$
		\State $k = 1$, $\f{x}_{0} = \zero$, $\f{r}_0 = \f{p}_0 = \f{b}$
		\While{Not Terminated}
		\State $\alpha_{k-1} = \|\f{r}_{k-1}\|^2/\lr{\f{p}_{k-1}, \f{Ap}_{k-1}}$
		\State $\f{x}_{k} = \f{x}_{k-1} + \alpha_{k-1}\f{p}_{k-1}$
		\State $\f{r}_{k} = \f{r}_{k-1} - \alpha_{k-1}\f{Ap}_{k-1}$
		\State $\beta_{k-1} = \|\f{r}_{k}\|^2/\|\f{r}_{k-1}\|^2$
		\State $\f{p}_{k} = \f{r}_{k} + \beta_{k-1}\f{p}_{k-1}$
		\State $k = k + 1$
		\EndWhile
		\State \Return $\f{x}_k$
	\end{algorithmic}
    \end{algorithm}
    \end{minipage}
    \hfill
    \begin{minipage}{0.49\textwidth}
    \begin{algorithm}[H]
    \caption{Conjugate Residual}\label{alg:cr}
	\begin{algorithmic}[1]
		\Require $\f{A}$ and $\f{b}$
		\State Let $k = 1$, $\f{x}_{0} = \zero$, $\f{r}_0 = \f{p}_0 = \f{b}$
		\While{Not Terminated}
		\State $\alpha_{k-1} = \|\f{r}_{k-1}\|^2_{\f{A}}/\lr{\f{p}_{k-1}, \f{Ap}_{k-1}}_{\f{A}}$
		\State $\f{x}_{k} = \f{x}_{k-1} + \alpha_{k-1}\f{p}_{k-1}$
		\State $\f{r}_{k} = \f{r}_{k-1} - \alpha_{k-1}\f{Ap}_{k-1}$
		\State $\beta_{k-1} = \|\f{r}_{k}\|^2_{\f{A}}/\|\f{r}_{k-1}\|^2_{\f{A}}$
		\State $\f{p}_{k} = \f{r}_{k} + \beta_{k-1}\f{p}_{k-1}$
		\State $k = k + 1$
		\EndWhile
		\State \Return $\f{x}_k$
	\end{algorithmic}
    \end{algorithm}
    \end{minipage}\smallskip
  
    Going beyond PD settings, there have been several efforts in extending the CD methods. However, in all these attempts, the underlying assumption is that the linear system is consistent. The extension to consistent systems involving positive semi-definite (PSD) matrices is relatively straightforward and has been explored in the literature \cite{greenbaum1997iterative,hayami2018convergence,hayami2001behaviour}. In the work of Kaasschieter \cite{kaasschieter1988preconditioned}, the behavior of preconditioned CG under positive semi-definite, but consistent, systems is studied. 
	The main challenge in extending the CD methods to indefinite settings is the risk of encountering division by zero. In CG, for example, this occurs when $\lr{\f{p}_{k-1}, \f{Ap}_{k-1}} = 0$ at which point $\alpha_{k-1}$ is no longer well-defined and CG breaks down. This can happen when either $\f{A p}_{k-1} = \zero$ or $\f{Ap}_{k-1} \perp \f{p}_{k-1}$. We will show later that the former case constitutes a breakdown, as it signals the final iteration and CG can be safely terminated (\cref{lem:cg:Ap}). The latter case, however, is considered an ``unlucky breakdown''. For CG applied to indefinite, but consistent, settings, this situation is handled by employing the concept of a hyperbolic pair \cite{fletcher2006conjugate,luenberger1969hyperbolic}. Likewise, there are works \cite{fasano2005planar,fasano2005planar_app,fasano2007lanczos}, which developed planar CG, offering an alternative approach for handling unlucky breakdowns within CG for indefinite, but consistent, settings. 
	In indefinite settings, similar breakdowns can occur in CR when $\lr{\f{r}_{k-1}, \f{Ar}_{k-1}} = 0$ at which point $\beta_{k-1}$ is not well-defined and CR breaks down. This can happen when either $ \f{Ar}_{k-1} = \zero $ or $\f{Ar}_{k-1} \perp \f{r}_{k-1}$. Again, we will later show that the former case is a breakdown that is theoretically guaranteed by \cref{thm:cr:Ar}, while the latter scenario is again characterized as an unlucky breakdown.
	
	As mentioned earlier, our goal in this paper is not to design new variants of CG or CR that can extend to indefinite systems. Instead, our emphasis is on exploring the convergent behavior of the original methods in such situations. As a result, to exclude cases where the algorithm encounters an unlucky breakdown in indefinite settings, we make an umbrella assumption that a zero-curvature, as in \cref{def:zero_curv}, is not encountered in the algorithm prior to termination.
	\begin{assumption}[No Unlucky Breakdown]\label{assumpt:unlucky}
		If $\f{A}$ is indefinite, the CD methods do not encounter a zero-curvature direction as in \cref{def:zero_curv}.
	\end{assumption}
    Without \cref{assumpt:unlucky}, the convergence behavior of CG and CR for inconsistent systems can generally be studied based on whether or not an unlucky breakdown occurs during the iterations. In the case of an unlucky breakdown, both CG and CR fail to converge by definition, which is an uninteresting scenario. The more compelling question is to explore their convergence behavior in the absence of an unlucky breakdown throughout the iterations. Therefore, to align with our goal of investigating the convergence of CG and CR for inconsistent indefinite systems, we exclude the former case and focus on the latter, justifying the introduction of \cref{assumpt:unlucky}.
    Furthermore, we note that \cref{assumpt:unlucky} is only required for indefinite settings. This is because it turns out that for positive semi-definite matrices, the only breakdown is a lucky one, which only happens at the very last iteration, and signals that the algorithm has indeed reached the grade; see \cref{lem:cg:Ap,thm:cr:Ar}. 
    In addition, in most practical applications, unlucky breakdowns are virtually never encountered, which is consistent with the observations in \cite{sogabe2022krylov}. Therefore, we argue that \cref{assumpt:unlucky} is reasonable in our context.

    We note in passing that, in some applications, encountering a zero-curvature direction while solving \cref{eq:linear_system} can be algorithmically beneficial. For instance, in Newton-type optimization methods, where $\f{A}$ and $\f{b}$ typically represent the Hessian and the negative gradient of a function, respectively, a zero-curvature direction can be leveraged to construct effective descent directions (see, e.g., \cite{liu2022newton,lim2023complexity,royer2018complexity}).
	
	\begin{remark}
		One might suspect that with $ \f{x}_{0} $ if, at the first iteration, $\langle \f{b}, \f{Ab} \rangle \neq 0$, subsequent iterations would also avoid an unlucky breakdown. However, as noted in \cite{fong2011minimum}, this is not necessarily the case. The following example,
		\[\f{A} = \begin{bmatrix}
			1 & 0 & 0\\ 
			0 & -1 & 0\\
			0 & 0 & 2
		\end{bmatrix} \quad \quad \f{b} = \begin{bmatrix}
			1 \\ -1/\sqrt{11}\\ 1,
		\end{bmatrix}\]
		shows that, even though $\lr{\f{b},\f{Ab}} \neq 0$, at the second iteration of CG, we have $\lr{\f{p}_1, \f{A}\f{p}_1} = 0$ . 
	\end{remark}

    To end this section, we briefly outline some important, and well-known, facts about the behavior of CD methods under consistent systems. The proofs are either omitted or are deferred to \cref{sec:appendix:results} for completeness. 
    \begin{fact}[Pseudo-inverse Solution Under Consistent Systems]\label{fact:convergence_consistent}
		Let $\f{b} \in \Range(\f{A})$. Under \cref{assumpt:unlucky}, \cref{alg:cg,alg:cr} converge to the pseudo-inverse solution $ \f{x}_{g} = \f{A}^{\dagger}\f{b} $ and $\f{r}_{g} = \zero$.
	\end{fact}
    \begin{fact}\label{fact:pg=0_consistent}
        Let $\f{b} \in \Range(\f{A})$. Under \cref{assumpt:unlucky}, for \cref{alg:cg,alg:cr}, we have $\f{p}_k = \f{0}$ if and only if $k = g$.
    \end{fact}
    For consistent systems, under \cref{assumpt:unlucky}, both the residual vector (\cref{fact:convergence_consistent}), and the update direction (\cref{fact:pg=0_consistent}) are zero at termination for the CD methods. \cref{fact:convergence_consistent} has been shown in the works of Hayami \cite{hayami2018convergence,hayami2001behaviour} for semi-definite matrices. It is a straightforward extension to the indefinite matrices with \cref{assumpt:unlucky} in place, and the proof of \cref{fact:convergence_consistent} is a direct consequence of \cref{lem:solution_in_krylov}, which can be found in \cref{sec:appendix:results}. 
    \cref{fact:pg=0_consistent} indicates that, under consistent systems, not only does $\f{r}_g = \f{0}$ at termination, but $\f{p}_g = \f{0}$ as well. While this fact may seem trivial, it is worth noting as it sharply contrasts with the behavior of CG under inconsistent systems, where $\f{Ar}$ is never zero (\cref{thm:cg:Ar}), but $\f{Ap}$ is zero at termination (\cref{lem:cg:Ap}).\label{sec:intro}

    \section{Inconsistent Systems}
    Under consistent systems, and in the absence of round-off errors, the orthogonality and $\f{A}$-conjugacy are the distinctive features of the CD methods. These properties play a key role in the development of the characteristically clean and concise algorithmic procedure of these methods. However, as long as the algorithm has not encountered an unlucky breakdown, these properties continue to hold even in inconsistent settings. The proof of \cref{prop:properties} follows the standard line of reasoning in the consistent case, and hence is deferred to \cref{sec:appendix:unify}; see \cref{lem:properties_CD,lem:expanding_subspace}.
	
	\begin{property}
		\label{prop:properties}
		Consider any real symmetric matrix $\f{A}$, let $\{\f{p}_0, \cdots, \f{p}_{k-1}\}$ and $\{\f{r}_0, \cdots, \f{r}_{k-1}\}$ be the update directions and residuals generated by \cref{alg:cg,alg:cr} up to iteration $k - 1$. 
		\begin{enumerate}[label = (\alph*)]
			\item \label{property:cgcr} For both \cref{alg:cg,alg:cr}, we have $\dim(\krylov{A}{b}{k}) = k$ and,
			\begin{align*}
				\Span(\f{r}_0, \cdots, \f{r}_{k-1}) &= \Span(\f{p}_0, \cdots, \f{p}_{k-1}) = \krylov{A}{b}{k}
			\end{align*}
			\item \label{property:cg}
			In CG, \cref{alg:cg}, for all $0 \leq i < j \leq k-1$, we have
			\begin{align*}
				\lr{\f{r}_i, \f{r}_j} = 0, \quad \lr{\f{p}_i, \f{Ap}_j} = 0 \quad \text{and} \quad \lr{\f{r}_j, \f{p}_i} = 0;
			\end{align*}
           while for all $0 \leq j \leq i \leq k-1$, we have $\lr{\f{r}_j, \f{p}_i} = \vnorm{\f{r}_i}^2$.
			\item \label{property:cr}
			In CR, \cref{alg:cr}, for all $0 \leq i < j \leq k-1$, we have
			\begin{align*}
				\lr{\f{r}_i, \f{Ar}_j} = 0, \quad \text{and} \quad \lr{\f{Ap}_i, \f{Ap}_j} = 0.
			\end{align*}
		\end{enumerate}
	\end{property}
	
Note that \cref{prop:properties} does not require \cref{assumpt:unlucky} since these properties hold prior to a potential occurrence of an unlucky breakdown. Indeed, \cref{assumpt:unlucky} is only required when we need to ensure that \cref{alg:cg,alg:cr} can successfully generate all iterates up to termination.

\subsection{Conjugate Gradient}\label{sec:cg}
	
Under an inconsistent system, there is a non-trivial, and a prior unknown, lower bound on the norm of the residual. Indeed, $\vnorm{\f{r}}^{2} = \vnorm{\f{b} - \f{Ax}}^{2} = \|\f{b}^\perp\|^2 + \|\ft{b} - \f{Ax}\|^2 \geq \|\f{b}^\perp\|^2 $. Thus, termination criterion based on $ \vnorm{\f{r}} $ is no longer practical. The alternative is to observe $\vnorm{\f{Ar}} \rightarrow 0$, as $\|\f{Ar}\| = 0$ indicates achieving a normal solution. However, we now show that CG is fundamentally unable to recover normal solutions in such settings. While it has been empirically observed that CG exhibits numerical instability under inconsistent systems, \cref{thm:cg:Ar} provides a theoretical analysis of CG's failure to converge in such settings.
\begin{theorem}[CG's Inability to Recover Normal Solutions]\label{thm:cg:Ar}
	Let $ \f{b} \notin \Range(\f{A}) $ and $ \f{b} \notin \Null(\f{A})$. Under \cref{assumpt:unlucky}, in \cref{alg:cg}, we have $\f{A r}_k \neq \f{0}$, for all $0 \leq k \leq g-1$.
\end{theorem}
\begin{proof}
	We prove the claim by contradiction. Clearly, by the assumption on $ \f{b} $, $ \f{Ar}_{0} = \f{Ab} \neq \zero $. Suppose for some $ 1 \leq k \leq g-1 $, $\f{A}\f{r}_{k} = \zero$. Using $\f{p}_{k} = \f{r}_{k} + \beta_{k-1}\f{p}_{k-1}$, we get $\f{A}\f{p}_{k} =  \beta_{k-1} \f{A}\f{p}_{k-1}$. By the assumption on $ \f{b} $, we have $ \beta_{k-1} \neq 0 $. Subsequently, from $\f{r}_{k} = \f{r}_{k-1} - \alpha_{k-1}\f{Ap}_{k-1}$, we get $\f{r}_{k} = \f{r}_{k-1} - \alpha_{k-1}\f{Ap}_{k}/\beta_{k-1}$, and by the inner-product with $\f{p}_{k-1}$ and using \cref{prop:properties}-\labelcref{property:cg}, it gives $\lr{\f{p}_{k-1},\f{r}_{k-1}} = 0$. However, by \cref{prop:properties}-\labelcref{property:cg}, we also have $\lr{\f{p}_{k-1},\f{r}_{k-1}} = \vnorm{\f{r}_{k-1}}^{2} \neq 0$, for $0 \leq k \leq g - 1$, which leads to a contradiction.
\end{proof}

\begin{remark}
    CG's inability to recover a solution of the normal equation can be better understood through the framework from which it is derived. Specifically, CG is a particular algorithmic implementation of the Orthogonal-Residual (OR) framework \cite{greenbaum1997iterative,choi2006iterative}, which applies to square, but not necessarily symmetric, matrices. When restricted to PD matrices, the OR framework can be algorithmically implemented  as \cref{alg:cg}. However, in \cref{sec:appendix:or}, we demonstrate that any algorithm derived from the OR framework cannot converge to a normal solution for symmetric but inconsistent systems (see \cref{thm:or}). Since CG is derived from OR, it inherits this ``non-convergent'' property.
\end{remark}
	
While \cref{thm:cg:Ar} for inconsistent systems stands in sharp contrast to \cref{fact:convergence_consistent} for consistent systems, it turns out that the last iterate of CG for inconsistent settings can still be characterized, in an analogous way to \cref{fact:pg=0_consistent} for consistent systems.
	\begin{lemma}
		\label{lem:cg:Ap}
		Let $\f{b} \notin \Range(\f{A})$. Under \cref{assumpt:unlucky}, in \cref{alg:cg} we have $\f{A p}_{k-1} = \f{0}$ if and only if $k = g$. Furthermore, $ \f{p}_{g-1} \neq \zero$.
	\end{lemma}
	\begin{proof}
		Let $ k = g $. By \cref{prop:properties}-\labelcref{property:cgcr} and \cref{def:grades}, we note that $\f{Ap}_{g-1} \in \krylov{A}{b}{g}$ and
		\begin{align*}
			\f{Ap}_{g-1} & = \sum_{j=1}^{g} c_{j-1} \f{r}_{j-1}.
		\end{align*}
		In addition, from \cref{prop:properties}-\labelcref{property:cg}, we have $\lr{\f{p}_{i-1}, \f{Ap}_{g - 1}} = 0$, for all $1 \leq i < g - 1$. For $i=1$, $0 = \lr{\f{p}_0, \f{Ap}_{g - 1}} = c_0 \lr{\f{p}_0, \f{r}_0} = c_0 \|\f{r}_0\|^2$,
		which since $\|\f{r}_0\| \neq 0$ implies $c_0 = 0$. Similarly, for $2 \leq i < g - 1$, 
		\begin{align*}
			0 &= \lr{\f{p}_{i-1}, \f{Ap}_{g-1}} = \lr{\f{r}_{i-1} + \beta_{i-2}\f{p}_{i-2}, \f{Ap}_{g-1}} = \lr{\f{r}_{i-1}, \f{Ap}_{g-1}}\\
			&= \sum_{j=1}^{g} c_{j-1} \lr{\f{r}_{i-1}, \f{r}_{j-1}} = c_{i-1} \|\f{r}_{i-1}\|^2,
		\end{align*}
		where the last equality follows from \cref{prop:properties}-\labelcref{property:cg}. Since $\|\f{r}_{i-1}\| \neq 0$, it follows that $c_{i-1} = 0$. As a result, we must have $\f{Ap}_{g-1} = c_{g-1}\f{r}_{g-1}$. Now, if $c_{g-1} = 0$, then we are done. Otherwise, assuming $\f{Ap}_{g-1} \not= \zero$, since $\f{r}_{g-1} \neq \zero$, we can, in principle, proceed to the next iteration by calculating $\alpha_{g-1} = \|\f{r}_{g-1}\|^2 / \lr{\f{p}_{g-1},\f{Ap}_{g-1}}$. We first note that by \cref{prop:properties}-\labelcref{property:cg}, we have 
		\begin{align*}
			c_{g-1} = \frac{\lr{\f{p}_{g-1},\f{Ap}_{g-1}}}{\lr{\f{p}_{g-1},\f{r}_{g-1}}} = \frac{\lr{\f{p}_{g-1},\f{Ap}_{g-1}}}{\lr{\f{r}_{g-1} + \beta_{g-2}\f{p}_{g-2},\f{r}_{g-1}}} = \frac{\lr{\f{p}_{g-1},\f{Ap}_{g-1}}}{\|\f{r}_{g-1}\|^{2}} = \frac{1}{\alpha_{g-1}}.
		\end{align*}
		It follows that $\f{r}_{g} = \f{r}_{g-1} - \alpha_{g-1}\f{Ap}_{g-1} = \f{r}_{g-1} - \alpha_{g-1}c_{g-1}\f{r}_{g-1} = \zero$,
		which is impossible since $\f{b} \notin \Range(\f{A})$. By reductio ad absurdum, this implies $\f{Ap}_{g-1} = \zero$.
		
		\bigskip 
		\noindent
		Conversely, let $\f{Ap}_{k-1} = \f{0}$ for some $ 1 \leq k \leq g $. Firstly, we show that $\f{p}_{k-1} \not= \zero$. Clearly, if $ k=1 $, we have $ \f{p}_{0} = \f{b} \neq \zero $. If $ k \geq 2 $, from $ \f{p}_{k-1} = \f{r}_{k-1} + \beta_{k-2}\f{p}_{k-2} $ and \cref{prop:properties}-\labelcref{property:cg}, it follows that 
		\begin{align*}
			\vnorm{\f{p}_{k-1}}^{2} & = \vnorm{\f{r}_{k-1}}^{2} + 2\beta_{k-2}\lr{\f{r}_{k-1}, \f{p}_{k-2}} + \beta_{k-2}^{2}\vnorm{\f{p}_{k-2}}^{2}\\
            &= \vnorm{\f{r}_{k-1}}^{2} + \beta_{k-2}^{2}\vnorm{\f{p}_{k-2}}^{2} \geq \vnorm{\f{r}_{k-1}}^{2} \neq 0.
		\end{align*}
		As a result, since $\f{p}_{k-1} \in \krylov{A}{b}{k}$ and $\f{p}_{k-1} \neq \f{0}$, we can write $ \f{p}_{k-1}  = \sum_{i=0}^{k-1} c_{i}\f{A}^{i}\f{b} $ and at least one $c_{j}$ must be non-zero. Hence, $\f{0} = \f{Ap}_{k-1}  = \sum_{i=0}^{k-1} c_{i}\f{A}^{i+1}\f{b}$, which implies $\dim(\krylov{A}{b}{k+1}) = \dim(\krylov{A}{b}{k})$, and so we must have $k = g$.     
	\end{proof}
	
	\begin{remark}
		\label{rem:grade}
		There is a subtlety in how the grade of $ \f{b} $ with respect to $ \f{A} $ is reflected in the total number of iterations before termination. If $ \f{b} \in \Range(\f{A}) $, CG terminates when $ k = g $ (\cref{fact:convergence_consistent,fact:pg=0_consistent}), while for $ \f{b} \notin \Range(\f{A}) $, the termination occurs at $ k = g-1 $ (\cref{lem:cg:Ap}). However, in both cases, the total number of iterations is determined by the number of non-zero $\f{b}$-relevant eigenvalues\footnote{These are distinct eigenvalues whose corresponding eigenspace is not entirely orthogonal to $ \f{b} $; see the work \cite{liu2022newton} for more details} of $\f{A}$. Indeed, suppose $ \f{b} \notin \Range(\f{A}) $ and consider its decomposition as in \cref{eq:orth_decomp}. Let $g$ and $\tilde{g}$ denote the grade of $\f{b}$ and $\tilde{\f{b}}$, respectively, with respect to $\f{A}$. CG$(\f{A}, \ft{b})$ and CG$(\f{A}, \f{b})$ terminate after $\tilde{g}$ and $g - 1$ iterations, respectively. On the other hand, noting that $ \ft{b} $ is orthogonal to the eigenspace corresponding to the zero eigenvalues of $ \f{A} $, for any eigenvector $ \f{v} $ of $ \f{A} $, we have 
		\begin{align*}
			\lr{\f{b},\f{v}} = \langle\ft{b},\f{v}\rangle + \langle\f{b}^{\perp},\f{v}\rangle = \left\{\begin{array}{ll}
				\langle\ft{b},\f{v}\rangle, & \text{if } \f{v}\in \Range(\f{A}), \\ \\
				\langle \f{b}^{\perp},\f{v}\rangle, & \text{if } \f{v}\in \Null(\f{A}).
			\end{array}\right.
		\end{align*}
		That is, the number of $\f{b}$-relevant eigenvalues is exactly one more than the number of $\ft{b}$-relevant ones. As a result, from Lemma 1 in the work of Liu and Roosta\cite{liu2022newton}, it follows that $\tilde{g} = g - 1$, which implies that both situations result in precisely the same number of iterations. 
	\end{remark}
	
	\cref{thm:cg:Ar} implies that, for inconsistent systems, CG's output might not necessarily correspond to any solutions. Once $ k=g-1 $, by \cref{lem:cg:Ap}, $\alpha_{g-1}$ is no longer well-defined and the update direction $\f{p}_{g-1}$ is not used. One may na\"{i}vely attempt to construct $\f{x}_g = \f{x}_{g-1} + \alpha\f{p}_{g-1}$ for some suitable $\alpha$ in an effort to recover a normal solution. However, for any such update $\f{x}_g  = \f{x}_{g-1} + \alpha\f{p}_{g-1}$, we have $\f{r}_g  = \f{r}_{g-1}$, and hence by \cref{thm:cg:Ar}, $ \f{x}_{g} $ cannot be a normal solution. Nonetheless, we now show that a modification to CG theoretically allows for the recovery of the pseudo-inverse solution in inconsistent settings.  
	
	To do this, we make two key observations. First, by \cref{lem:solution_in_krylov} in \cref{sec:appendix:results}, the Krylov subspace $\krylov{A}{b}{g}$ contains a normal solution. Denoting this normal solution as $ \f{x}^{\star} $, under \cref{assumpt:unlucky} and using \cref{prop:properties}-\labelcref{property:cgcr}, we can write 
	\begin{align}\label{eq:solution_in_krylov}
		\f{x}^{\star} = \sum_{k=1}^{g}\eta_{k-1}\f{p}_{k-1},
	\end{align}
	and the aim is thus to find the suitable scalars $\eta_{i}$'s. Next, if the vector $\f{b}^\perp$ can be somehow calculated as part of CG iterations, then it might be possible to eliminate its contribution in a suitable way to recover the pseudo-inverse solution. Fortunately, it turns out that $\f{b}^\perp$ is readily available at the last iteration of CG.
	
	\begin{lemma}\label{lem:pg_=_r*}
		Let $\f{b} \notin \Range(\f{A})$. Under \cref{assumpt:unlucky}, in \cref{alg:cg}, we have 
		\begin{align*}
			\f{p}_{g-1} = \frac{\|\f{p}_{g-1}\|^2}{\|\f{r}_{g-1}\|^2}\f{b}^\perp.
		\end{align*}
	\end{lemma}
	\begin{proof}
		By \cref{lem:p_in_null_b} in \cref{sec:appendix:results}, we know that $(\f{I} - \f{A}^\dagger\f{A})\f{p}_{g-1} = a_{g-1}\f{b}^\perp$, which by \cref{lem:cg:Ap} gives $\f{p}_{g-1} =  a_{g-1}\f{b}^\perp$. Furthermore, \cref{lem:cg:Ap} gives $\lr{\f{p}_{g-1},\f{b}} = \lr{\f{p}_{g-1},\ft{b}} + \lr{\f{p}_{g-1},\f{b}^{\perp}} = \lr{\f{p}_{g-1},\f{b}^{\perp}} $. Hence, 
		\begin{align*}
			a_{g-1} = \frac{\vnorm{\f{p}_{g-1}}^{2}}{\lr{\f{p}_{g-1},\f{b}^\perp}} = \frac{\vnorm{\f{p}_{g-1}}^{2}}{\lr{\f{p}_{g-1},\f{b}}}.
		\end{align*}
		Finally, using \cref{prop:properties}-\labelcref{property:cg} and the definition of $\beta_{k-1}$, we get $\lr{\f{p}_{g-1}, \f{b}} = \|\f{r}_{g-1}\|^2$.
	\end{proof}
	We now express \cref{eq:solution_in_krylov} as 
	\begin{align}
		\label{eq:solution_in_krylov_gamma}
		\f{x}^{\star} = \sum_{k = 1}^{g} \eta_{k-1} \f{p}_{k-1} = \sum_{k = 1}^{g}(\alpha_{k-1} + \gamma_{k-1})\f{p}_{k-1},
	\end{align}
	where $\alpha_k$'s are from \cref{alg:cg} and $\gamma_k$'s are unknown constant that we aim to find. It follows that 
	\begin{align*}
		\frac{\|\f{r}_{g-1}\|^2}{\|\f{p}_{g-1}\|^2}\f{p}_{g-1} &= \f{b}^\perp = \f{b} - \f{A}\f{x}^{\star} = \f{b} - \sum_{k = 1}^{g}\alpha_{k-1}\f{Ap}_{k-1} - \sum_{k=1}^{g}\gamma_{k-1}\f{Ap}_{k-1} = \f{r}_{g-1} - \sum_{k = 1}^{g-1}\gamma_{k-1}\f{Ap}_{k-1},
	\end{align*}
	where the first equality uses \cref{lem:pg_=_r*}, and the fourth equality employs \cref{lem:cg:Ap} and the fact that in \cref{alg:cg}, $\f{r}_{k} = \f{r}_{k-1} - \alpha_{k-1}\f{Ap}_{k-1}$. Now, for each $1 \leq k \leq g-1$, we have
	\begin{align*}
		\frac{\|\f{r}_{g-1}\|^2}{\|\f{p}_{g-1}\|^2}\lr{\f{p}_{k-1}, \f{p}_{g-1}} & = \lr{\f{p}_{k-1}, \f{r}_{g-1}} - \gamma_{k-1}\lr{\f{p}_{k-1},\f{Ap}_{k-1}},
	\end{align*}
	which using \cref{prop:properties}-\labelcref{property:cg} gives
	\begin{align*}
		\gamma_{k-1} & = -\frac{\|\f{r}_{g-1}\|^2\lr{\f{p}_{k-1}, \f{p}_{g-1}}}{\|\f{p}_{g-1}\|^2\lr{\f{p}_{k-1},\f{Ap}_{k-1}}}.
	\end{align*}
	Furthermore, 
	\begin{align*}
		\lr{\f{p}_{k-1}, \f{p}_{g-1}} & = \lr{\f{p}_{k-1}, \f{r}_{g-1} + \beta_{g-2}\f{p}_{g-2}} = \beta_{g-2}\lr{\f{p}_{k-1}, \f{p}_{g-2}} = \beta_{g-2}\beta_{g-3}\lr{\f{p}_{k-1}, \f{p}_{g-3}} = \ldots\\
        & = \prod_{j=2}^{g-(k-1)}\beta_{g-j}\|\f{p}_{k-1}\|^2 = \frac{\|\f{r}_{g-1}\|^2}{\|\f{r}_{k-1}\|^2}\|\f{p}_{k-1}\|^2,
	\end{align*}
	which gives,
	\begin{align*}
		\gamma_{k-1} = -\frac{\|\f{r}_{g-1}\|^4\|\f{p}_{k-1}\|^2}{\|\f{p}_{g-1}\|^2\|\f{r}_{k-1}\|^2\lr{\f{p}_{k-1},\f{Ap}_{k-1}}},
	\end{align*}
	for $1 \leq k \leq g-1$. Using these scalars, coupled with the particular choice $\gamma_{g-1} = 0$, $ \f{x}^{\star} $ in \cref{eq:solution_in_krylov_gamma} now amounts to a normal solution. To find the pseudo-inverse solution, we note that, \cref{lem:p_in_null_b} implies 
	\begin{align*}
		(\f{I} - \f{A}^\dagger\f{A})\f{x}^{\star} = (\f{I} - \f{A}^\dagger\f{A})\sum_{k=1}^{g}\eta_{k-1}\f{p}_{k-1} = \left(\sum_{k=1}^{g}\eta_{k-1}a_{k-1}\right) \f{b}^{\perp} = c \f{p}_{g-1},
	\end{align*}
	where $c = \sum_{k=1}^{g}\eta_{k-1}a_{k-1}$. However, using \cref{lem:cg:Ap},
	\begin{align*}
		c = \frac{\lr{\f{p}_{g-1}, \f{x}^{\star}}}{\|\f{p}_{g-1}\|^2}. 
	\end{align*}
	Hence, $ (\f{I} - \f{A}^\dagger\f{A})\f{x}^{\star} = \lr{\f{p}_{g-1}, \f{x}^{\star}}\f{p}_{g-1}/\|\f{p}_{g-1}\|^2$, which finally gives
	\begin{align*}
		\f{x}^{\star} - \frac{\lr{\f{p}_{g-1}, \f{x}^{\star}}}{\|\f{p}_{g-1}\|^2}\f{p}_{g-1} & = \f{A}^\dagger\f{A}\f{x}^{\star} = \f{A}^\dagger\f{b}.
	\end{align*}
	Note that the left-hand side can also be regarded as the orthogonal projection of $\f{x}^{\star}$ onto the orthogonal complement of $\Range(\f{p}_{g-1})$. This leads us to \cref{alg:cg_pis,thm:cg:pseudo}.
	\begin{algorithm}[htbp]
		\caption{Conjugate Gradient with Pseudo-inverse Solution}
		\begin{algorithmic}[1]
			\Require $\f{A}$ and $\f{b}$
			\State Let $k = 1$, $\f{r}_0 = \f{p}_0 = \f{b}$
			\While{$\|\f{Ap}_{k-1}\| \neq 0$} \Comment{Termination Condition}
			\State $\alpha_{k-1} = \langle \f{r}_{k-1}, \f{r}_{k-1}\rangle/\lr{\f{p}_{k-1}, \f{Ap}_{k-1}}$ \Comment{Classical CG: Steps 3 - 7}
			\State $\f{x}_{k} = \f{x}_{k-1} + \alpha_{k-1}\f{p}_{k-1}$
			\State $\f{r}_{k} = \f{r}_{k-1} - \alpha_{k-1}\f{Ap}_{k-1}$
			\State $\beta_{k-1} = \lr{\f{r}_{k}, \f{r}_{k}}/\lr{\f{r}_{k-1}, \f{r}_{k-1}}$
			\State $\f{p}_{k} = \f{r}_{k} + \beta_{k-1}\f{p}_{k-1}$
			\State $\gamma_{k-1} = \lr{\f{p}_{k-1}, \f{p}_{k-1}} / (\lr{\f{r}_{k-1}, \f{r}_{k-1}} \lr{\f{p}_{k-1}, \f{Ap}_{k-1}})$ \Comment{Corrective Procedure: Steps 8 - 9}
			\State $\f{x}'_{k} = \f{x}_{k-1}' + \gamma_{k-1}\f{p}_{k-1}$
			\State $k = k + 1$
			\EndWhile
			\If{$\|\f{p}_{k-1}\| \neq \f{0}$} \Comment{i.e., Inconsistent System}
			\State $\f{x}^{\star} = \f{x}_{k-1} - (\|\f{r}_{k-1}\|^4 / \|\f{p}_{k-1}\|^2) \f{x}_{k-1}'$ \Comment{Normal Solution}
			\State $\f{x}^{+} = \f{x}^{\star} - (\lr{\f{p}_{k-1}, \f{x}^{\star}} / \|\f{p}_{k-1}\|^2) \f{p}_{k-1}$ \Comment{Pseudo-inverse Solution}
			\Else \Comment{i.e., Consistent System}
			\State $\f{x}^{+} = \f{x}_{k-1}$
			\EndIf
			\State \Return $\f{x}^{+}$
		\end{algorithmic}\label{alg:cg_pis}
	\end{algorithm}
	\begin{theorem}[Pseudo-inverse Solution with \cref{alg:cg_pis}]
		\label{thm:cg:pseudo}
		Under \cref{assumpt:unlucky}, in \cref{alg:cg_pis} we have $ \f{x}^{+} = \f{A}^{\dagger}\f{b} $.
	\end{theorem}

    \begin{remark}
        We highlight that while, theoretically, \cref{alg:cg_pis} provides a means to extract the pseudo-inverse solution at termination, the extreme numerical instability of CG for inconsistent systems causes the theory to fall short of practical observations. The instability of CG under inconsistent systems is widely observed in many applications and also documented in the works \cite{kaasschieter1988preconditioned,rebjock2023fast}; see also the experiments in  \cref{sec:num_exp}. 
    \end{remark}
	
	\subsection{Conjugate Residual}\label{sec:cr}
    In this section, we show that under \cref{assumpt:unlucky}, CR is able to converge to a normal solution for inconsistent systems (\cref{thm:cr:Ar}). In doing so, we will also demonstrate the similarities between CR and CG, while highlighting additional advantages of CR. Similar to CG, we will then introduce a few additional steps within CR to recover the pseudo-inverse solution (\cref{thm:cr:pseudo}). Finally, we will show that, as long as an unlucky breakdown has not occurred, CR generates iterates that are identical to those of MINRES for any symmetric matrix $\f{A}$ and any vector $\f{b}$ (\cref{thm:mr=cr}).

    \subsubsection{Convergence Behavior}
    \label{sec:cr:convergence_behaviour}
    
    The convergent behavior of CR under various conditions have been extensively studied (see \cref{table:cr_bahaviours}). Originally developed by Hestenes and Stiefel \cite{hestenes1952methods}, CR was designed for positive definite matrices and can be straightforwardly applied to negative definite matrices. Luenberger \cite{luenberger1970conjugate} later extended CR to indefinite, albeit non-singular, matrices. For inconsistent systems, Hayami and Sugihara \cite{hayami2001behaviour,hayami2011geometric,hayami2004convergence} established that for inconsistent systems, the necessary and sufficient condition for CR to avoid unlucky breakdown is that the matrix $\f{A}$ is positive semi-definite. In particular, they demonstrated that for any indefinite $\f{A}$, there exists a right-hand side vector $\f{b}$ that will trigger an unlucky breakdown. While this is a crucial insight, it remains of interest to investigate the behavior of CR for the many other vectors $\f{b}$ that do not cause this breakdown. We do that here. Specifically, we show that for any symmetric matrix $\f{A}$ and any vector $\f{b}$, if no unlucky breakdown occurs, CR will indeed converge to a normal solution. Furthermore, we devise a few extra steps for CR to obtain the pseudo-inverse solution. 
    
    \begin{table}[htbp]
	\centering
	\begin{tabular}{|c|c|c|}
		\hline
		& \multirow{2}{*}{Consistent: $\f{b} \in \Range(\f{A})$} & \multirow{2}{*}{Inconsistent: $\f{b} \notin \Range(\f{A})$}\\
        [8pt]
        \hline
    	\multirow{2}{*}{(Semi-)Definite} & \multirow{2}{*}{Hestenes \& Stiefel \cite{hestenes1952methods}} & \multirow{2}{*}{Hayami \& Sugihara \cite{hayami2001behaviour,hayami2011geometric,hayami2004convergence}}\\
        [8pt]
        \hline
		\multirow{2}{*}{Indefinite} & \multirow{2}{*}{Luenberger \cite{luenberger1970conjugate} \& Greenbaum \cite{greenbaum1997iterative}} & \multirow{2}{*}{\cref{thm:cr:Ar,thm:cr:pseudo}}\\
        [8pt]
        \hline
	\end{tabular}
    \caption{Relevant studies on the convergent behavior of CR under various conditions.}\label{table:cr_bahaviours}
    \end{table}

    \begin{remark}
    	\label{rm:MRII}
        There are many algorithms considering the alternative Krylov subspaces, like $\krylov{A}{Ab}{}$ or $\krylov{A^\top A}{Ab}{}$, instead of $\krylov{A}{b}{}$, for example MR-II / RRMR \cite{hanke2017conjugate} and variants of CG and CR \cite{calvetti1994conjugate,fischer1996note}. The advantages of these alternative Krylov subspaces are that the vector $\f{Ab}$ is in the range of $\f{A}$ and, hence, avoids unnecessary complications arising from the inconsistent systems. However, it is also empirically shown that these alternative subspaces may not produce better results than $\krylov{A}{b}{}$ and has more computational costs, see the work \cite{calvetti2001choice} discussing MINRES and RRMR. It is also shown that the iterates of Krylov subspace methods defined in $\krylov{A}{Ab}{}$ tend to have a slower convergence than $\krylov{A}{b}{}$, see the work\cite{estrin2019euclidean} comparing SYMMLQ and CG.
    \end{remark}
    
    One of the similarities between CR and CG is that, as in \cref{lem:cg:Ap} for CG, at termination for CR we also have $\f{p}_{g-1} \in \Null(\f{A})$.
	\begin{lemma}\label{lem:cr:Ap}
		Let $\f{b} \notin \Range(\f{A})$. Under \cref{assumpt:unlucky}, in \cref{alg:cr}, we have $\f{A p}_{k-1} = \f{0}$ if and only if $k = g$. Furthermore, $ \f{p}_{g-1} \neq \zero$.
	\end{lemma}
	\begin{proof}
		The proof follows a similar line of reasoning as that of \cref{lem:cg:Ap}. Let $ k = g $. If $g=1$, since $ \f{b} \notin \Range(\f{A}) $, we must have $\f{p}_{0} = \f{b} \in \Null(\f{A})$ and the claim holds. So consider $g > 1$. By \cref{prop:properties}-\labelcref{property:cgcr} and \cref{def:grades}, we note that $\f{Ap}_{g-1} \in \krylov{A}{b}{g}$ and
		\begin{align*}
			\f{Ap}_{g-1} & = \sum_{j=1}^{g} c_{j-1} \f{r}_{j-1}.
		\end{align*}
		In addition, from \cref{prop:properties}-\labelcref{property:cr}, we have $\lr{\f{A}\f{p}_{i-1}, \f{Ap}_{g - 1}} = 0$, for all $1 \leq i < g - 1$. For $i=1$, $0 = \lr{\f{A}\f{p}_0, \f{Ap}_{g - 1}} = c_0 \lr{\f{A}\f{p}_0, \f{r}_0} = c_0 \lr{\f{A}\f{r}_0, \f{r}_0}$.
		We note that $\lr{\f{A}\f{r}_0, \f{r}_0} \neq 0$ by \cref{assumpt:unlucky}. Hence, we must have $c_0 = 0$. Similarly, for $2 \leq i < g - 1$, 
		\begin{align*}
			0 &= \lr{\f{A} \f{p}_{i-1}, \f{Ap}_{g-1}} = \lr{\f{A}\f{r}_{i-1} + \beta_{i-2}\f{A}\f{p}_{i-2}, \f{Ap}_{g-1}} = \lr{\f{A}\f{r}_{i-1}, \f{Ap}_{g-1}}\\
			&= \sum_{j=1}^{g} c_{j-1} \lr{\f{A}\f{r}_{i-1}, \f{r}_{j-1}} = c_{i-1} \lr{\f{A}\f{r}_{i-1}, \f{r}_{i-1}},
		\end{align*}
		where the last equality follows from \cref{prop:properties}-\labelcref{property:cr}. Again, since $\lr{\f{A}\f{r}_{i-1}, \f{r}_{i-1}} \neq 0$, by \cref{assumpt:unlucky}, it follows that $c_{i-1} = 0$. As a result, we must have got $\f{Ap}_{g-1} = c_{g-1}\f{r}_{g-1}$. Now, if $c_{g-1} = 0$, then we are done. Otherwise, assuming $\f{Ap}_{g-1} \not= \zero$, since $\f{r}_{g-1} \neq \zero$, the constant $\alpha_{g-1} = \lr{\f{r}_{g-1},\f{A}\f{r}_{g-1}} / \lr{\f{A}\f{p}_{g-1},\f{Ap}_{g-1}}$ can be, in principle, calculated in the next iterate. We first note that by \cref{prop:properties}-\labelcref{property:cr}, we have 
		\begin{align*}
			c_{g-1} &= \frac{\lr{\f{A}\f{p}_{g-1},\f{Ap}_{g-1}}}{\lr{\f{A}\f{p}_{g-1},\f{r}_{g-1}}} = \frac{\lr{\f{A}\f{p}_{g-1},\f{Ap}_{g-1}}}{\lr{\f{A}\f{r}_{g-1} + \beta_{g-2}\f{A}\f{p}_{g-2},\f{r}_{g-1}}}\\ 
            &= \frac{\lr{\f{A}\f{p}_{g-1},\f{Ap}_{g-1}}}{\lr{\f{r}_{g-1},\f{Ar}_{g-1}}} = \frac{1}{\alpha_{g-1}}.
		\end{align*}
		It follows that $\f{r}_{g} = \f{r}_{g-1} - \alpha_{g-1}\f{Ap}_{g-1} = \f{r}_{g-1} - \alpha_{g-1}c_{g-1}\f{r}_{g-1} = \zero$,
		which is impossible since $\f{b} \notin \Range(\f{A})$. By reductio ad absurdum, this implies $\f{Ap}_{g-1} = \zero$.
		
		\bigskip 
		\noindent
		Conversely, let $ 1 \leq k \leq g $ be the first iteration where $\f{Ap}_{k-1} = \f{0}$. We first show that $\f{p}_{k-1} \neq \f{0}$. If $ k=1 $, we have $ \f{p}_{0} = \f{b} \neq \zero $. Suppose $ k \geq 2 $ and $\f{p}_{k-1} = \f{0}$. Clearly, $ \beta_{k-2} \neq 0 $ as otherwise $ \zero = \f{p}_{k-1} = \f{r}_{k-1} + \beta_{k-2}\f{p}_{k-2} $ also implies $\f{r}_{k-1}= \zero$ which by assumption on $ \f{b} $ is impossible. However, by \cref{prop:properties}-\labelcref{property:cgcr}, $\f{r}_{k-1}$ and $\f{p}_{k-2}$ must be linearly independent and we arrive at a contradiction. As a result, since $\f{p}_{k-1} \in \krylov{A}{b}{k}$ and $\f{p}_{k-1} \neq \f{0}$, we can write $ \f{p}_{k-1}  = \sum_{i=0}^{k-1} c_{i}\f{A}^{i}\f{b} $ and at least one $c_{j}$ must be non-zero. Hence, $\f{0} = \f{Ap}_{k-1}  = \sum_{i=0}^{k-1} c_{i}\f{A}^{i+1}\f{b}$, which implies $\dim(\krylov{A}{b}{k+1}) = \dim(\krylov{A}{b}{k})$, and so we must have $k = g$.
	\end{proof}

	While \cref{lem:cg:Ap,lem:cr:Ap} highlight a similarity between CG and CR, \cref{thm:cg:Ar,thm:cr:Ar} showcase a drastic difference. Specifically, we now show that, in sharp contrast to CG, CR always recovers a normal solution. 
	
	\begin{theorem}[CR's Ability to Recover Normal Solution]
		\label{thm:cr:Ar}
		Let $\f{b} \notin \Range(\f{A})$. Under \cref{assumpt:unlucky}, in \cref{alg:cr} we have $\f{A r}_{k-1} = \f{0}$ if and only if $k = g$.
	\end{theorem}
	\begin{proof}
		Suppose $ k = g $. From $ \f{p}_{g-1}  = \f{r}_{g-1} + \beta_{g-2}\f{p}_{g-2} $ and \cref{lem:cr:Ap}, we have $ \lr{\f{r}_{g-1}, \f{Ar}_{g-1}} + \beta_{g-2}\lr{\f{r}_{g-1}, \f{Ap}_{g-2}} = 0$, which using \cref{prop:properties}-\labelcref{property:cr} gives $ \lr{\f{r}_{g-1}, \f{Ar}_{g-1}}  = 0$, implying $ \beta_{g-2} = 0 $ by its definition in \cref{alg:cr}. Hence, $ \f{p}_{g-1}  = \f{r}_{g-1}  $ and the result follows from \cref{lem:cr:Ap}. 
		
		\bigskip 
		\noindent
		Conversely, suppose $\f{Ar}_{k-1} = \f{0}$ for some iteration $1 \leq k \leq g$. From $ \f{p}_{k-1}  = \f{r}_{k-1} + \beta_{k-2}\f{p}_{k-2} $, we get $\f{Ap}_{k-1}  = \beta_{k-2}\f{Ap}_{k-2}$, which using  \cref{prop:properties}-\labelcref{property:cr} gives $ \|\f{Ap}_{k-1}\|^2  = \beta_{k-2}\lr{\f{Ap}_{k-1}, \f{Ap}_{k-2}} = 0 $. Hence, $\f{Ap}_{k-1} = \f{0}$, which by \cref{lem:cr:Ap} can only happen when $k = g$.
	\end{proof}
	
	\begin{remark}
		For both \cref{alg:cg,alg:cr}, we always have $ \f{p}_{0} = \f{r}_{0} $. However, from the proof of \cref{thm:cr:Ar}, it can be seen that under the assumptions of \cref{thm:cr:Ar}, in \cref{alg:cr}, we also have $ \f{p}_{g-1} = \f{r}_{g-1} $. It is important to note that this property is specific to CR, and CG does not exhibit the same characteristics.
	\end{remark}
	
	While, by \cref{thm:cr:Ar}, CR is guaranteed to converge to a normal solution, we next show that we can in fact obtain the pseudo-inverse solution from the last iterate of CR through a straightforward procedure. This stands in sharp contrast to CG, where a similar procedure requires non-trivial modifications to the original algorithm. To achieve this, we first establish several technical lemmas.
	
	\begin{lemma}\label{lem:cr_tilde}
		Consider the iterates generated by the procedures CR$(\f{A}, \ft{b})$ and CR$(\f{A}, \f{b})$, and let where $g$ be the grade of $\f{b}$ with respect to $\f{A}$. For $1 \leq k \leq g$, we have  $\tilde{\alpha}_{k-1} = \alpha_{k-1}$, $\tilde{\beta}_{k-1} = \beta_{k-1}$, and $\ft{x}_{k-1} = \f{A}^\dagger\f{Ax}_{k-1}$.
	\end{lemma}
	\begin{proof}
		We first note that, for any $ \f{x} $, we have $(\f{I} - \f{A}^\dagger\f{A})\f{r} = \f{b}^{\perp}$. To show the claim, we also need to show that $ \ft{p}_{k-1} = \f{A}^\dagger\f{A}\f{p}_{k-1} $, and $\ft{r}_{k-1} = \f{A}^\dagger\f{A}\f{r}_{k-1}$.  The proof is done by induction. For $ k = 1 $, we have
		\[\f{A}^\dagger\f{A}\f{r}_0 = \f{A}^\dagger\f{A}(\ft{b} + \f{b}^\perp) = \ft{b} = \ft{r}_0,\]
		\[\f{A}^\dagger\f{A}\f{p}_0 = \f{A}^\dagger\f{A}(\ft{b} + \f{b}^\perp) = \ft{b} = \ft{p}_0,\]
		and
		\[\alpha_0 = \frac{\lr{\f{r}_0, \f{A}\f{r}_0}}{\|\f{A}\f{p}_0\|^2} = \frac{\lr{(\ft{b} + \f{b}^{\perp}), \f{A}(\ft{b} + \f{b}^{\perp})}}{\|\f{A}(\ft{b} + \f{b}^{\perp})\|^2} = \frac{\lr{\ft{b}, \f{A}\ft{b}}}{\|\f{A}\ft{b}\|^2} = \tilde{\alpha}_0.\]
		Furthermore, $\f{A}^\dagger\f{A}\f{r}_1 = \f{A}^\dagger\f{A}(\f{r}_0 - \alpha_0\f{A}\f{p}_0) = \ft{r}_0 - \tilde{\alpha}_0\f{A}\ft{p}_0 = \ft{r}_1$,
		where we have used the fact $\f{A}^\dagger\f{A} = \f{AA}^\dagger$. So, 
		\[\beta_0 = \frac{\lr{\f{r}_1, \f{A}\f{r}_1}}{\lr{\f{r}_0, \f{A}\f{r}_0}} = \frac{\lr{(\ft{r}_1 + \f{b}^\perp), \f{A}(\ft{r}_1 + \f{b}^\perp)}}{\lr{(\ft{b} + \f{b}^\perp), \f{A}(\ft{b} + \f{b}^\perp)}} = \frac{\lr{\ft{r}_1 \f{A}\ft{r}_1}}{\lr{\ft{b}, \f{A}\ft{b}}} = \tilde{\beta}_0.\]
		Next, assume the claim holds for all $k-1$. It follows that
		\begin{align*}
			\f{A}^\dagger\f{A}\f{r}_{k} & = \f{A}^\dagger\f{A}(\f{r}_{k-1} - \alpha_{k-1}\f{A}\f{p}_{k-1}) = \ft{r}_{k-1} - \tilde{\alpha}_{k-1}\f{A}\ft{p}_{k-1} = \ft{r}_{k},\\
			\f{A}^\dagger\f{A}\f{p}_{k} & = \f{A}^\dagger\f{A}(\f{r}_{k} + \beta_{k-1}\f{p}_{k-1}) = \ft{r}_{k} + \tilde{\beta}_{k-1}\ft{p}_{k-1} = \ft{p}_{k}, \\
			\alpha_{k} & = \frac{\lr{\f{r}_{k},\f{A}\f{r}_{k}}}{\|\f{A}\f{p}_{k}\|^2} = \frac{\lr{\ft{r}_{k},\f{A}\ft{r}_{k}}}{\|\f{A}\ft{p}_{k}\|^2} = \tilde{\alpha}_{k}.
		\end{align*}
		Also, since $\f{A}^\dagger\f{A}\f{r}_{k+1} = \f{A}^\dagger\f{A}(\f{r}_{k} - \alpha_{k}\f{A}\f{p}_{k}) = \ft{r}_{k} - \tilde{\alpha}_{k}\f{A}\ft{p}_{k} = \ft{r}_{k+1}$,
		we obtain
		\begin{align*}
			\beta_{k} & = \frac{\lr{\f{r}_{k+1}, \f{A}\f{r}_{k+1}}}{\lr{\f{r}_k, \f{A}\f{r}_k}} = \frac{\lr{\ft{r}_{k+1}, \f{A}\ft{r}_{k+1}}}{\lr{\ft{r}_k, \f{A}\ft{r}_k}} = \tilde{\beta}_{k}.
		\end{align*}
		Finally, for any $1 \leq k \leq g$, we have
		\begin{align*}
			\f{x}_{k} = \sum_{i=1}^{k} \alpha_{i-1}\f{p}_{i-1},
		\end{align*}
		which implies
		\begin{align*}
			\f{A}^\dagger\f{A}\f{x}_{k} = \f{A}^\dagger\f{A}\sum_{i=1}^{k} \alpha_{i-1}\f{p}_{i-1} = \sum_{i=1}^{k} \tilde{\alpha}_{i-1}\ft{p}_{i-1} = \ft{x}_{k}.
		\end{align*}
	\end{proof}
	We note that, as alluded to in \cref{rem:grade}, both procedures CR$(\f{A}, \ft{b})$ and CR$(\f{A}, \f{b})$ undergo precisely the same number of iterations. Hence, the assertion of \cref{lem:cr_tilde} holds true for all iterations in both procedures. 
	
	Since $ \ft{b} \in \Range(\f{A}) $, by \cref{fact:convergence_consistent}, for the procedure CR($\f{A}, \ft{b}$), we have $\f{A}^\dagger\f{A}\f{x}_{g - 1} =  \ft{x}_{g - 1} = \ft{x}_{\tilde{g}} = \f{A}^{\dagger} \ft{b} = \f{A}^{\dagger} \f{b} = \f{x}^{+}$.
	Now, using a similar line of reasoning as that leading up to \cref{thm:cg:pseudo} for CG, by utilizing \cref{lem:p_in_null_b}, for some constant $ c $, we have
	\begin{align*}
		(\f{I} - \f{A}^\dagger\f{A})\f{x}_{g-1} & = (\f{I} - \f{A}^\dagger\f{A})\sum_{k=1}^{g}\alpha_{k-1}\f{p}_{k-1} = c\f{p}_{g-1},
	\end{align*}
	which using \cref{lem:cr:Ap} gives $c = \lr{\f{p}_{g-1}, \f{x}_{g-1}}/\|\f{p}_{g-1}\|^2$.
    Hence, 
	\begin{align*}
		\f{x}_{g-1} - \frac{\lr{\f{p}_{g-1}, \f{x}_{g-1}}}{\|\f{p}_{g-1}\|^2}\f{p}_{g-1} & = \f{A}^\dagger\f{A}\f{x}_{g-1} = \f{x}^{+},
	\end{align*}
	which leads us to \cref{thm:cr:pseudo,alg:cr_pis}.
	\begin{algorithm}[htbp]
		\caption{Conjugate Residual with Pseudo-inverse Solution}\label{alg:cr_pis}
		\begin{algorithmic}[1]
			\Require a symmetric matrix $\f{A}$ and a vector $\f{b}$
			\State Let $k = 1$, $\f{r}_0 = \f{p}_0 = \f{b}$
			\While{$\|\f{Ar}_{k-1}\| \not = 0$} \Comment{Termination Condition}
			\State $\alpha_{k-1} = \lr{\f{r}_{k-1}, \f{A}\f{r}_{k-1}}/\|\f{A}\f{p}_{k-1}\|^2$ \Comment{Classical CR: Steps 3 - 7}
			\State $\f{x}_{k} = \f{x}_{k-1} + \alpha_{k-1}\f{p}_{k-1}$
			\State $\f{r}_{k} = \f{r}_{k-1} - \alpha_{k-1}\f{Ap}_{k-1}$
			\State $\beta_{k-1} = \lr{\f{r}_{k}, \f{A}\f{r}_{k}}/\lr{\f{r}_{k-1}, \f{A}\f{r}_{k-1}}$
			\State $\f{p}_{k} = \f{r}_{k} + \beta_{k-1}\f{p}_{k-1}$
			\State $k = k + 1$
			\EndWhile
			\State $\f{x}^{\star} = \f{x}_{k-1}$ \Comment{Normal Solution}
			\State $\f{x}^{+} = \f{x}_{k-1} - (\lr{\f{p}_{k-1}, \f{x}_{k-1}}/\|\f{p}_{k-1}\|^2)\f{p}_{k-1}$ \Comment{Pseudo-inverse Solution}
			\State \Return $\f{x}^{+}$
		\end{algorithmic}
	\end{algorithm}
	\begin{theorem}[Pseudo-inverse Solution with \cref{alg:cr_pis}]
		\label{thm:cr:pseudo}
		Under \cref{assumpt:unlucky}, in \cref{alg:cr_pis}, we have $ \f{x}^{+} = \f{A}^{\dagger}\f{b} $.
	\end{theorem}
	Beyond the fact that CR naturally recovers a normal solution, the procedure for obtaining the pseudo-inverse solution in CR is far more straightforward than that for CG as it merely involves the addition of one simple one-dimensional projection at termination.
	
	Under \cref{assumpt:unlucky}, we simply have 
	\begin{align*}
		\|\f{r}_t\|^2 & = \|\f{r}_{t-1} - \alpha_{t-1}\f{Ap}_{t-1}\|^2 = \|\f{r}_{t-1}\|^2 + \alpha_{t-1}^2\|\f{Ap}_{t-1}\|^2 - 2\alpha_{t-1}\lr{\f{r}_{t-1},\f{Ap}_{t-1}}\\
		& = \|\f{r}_{t-1}\|^2 + \frac{\lr{\f{r}_{t-1},\f{Ar}_{t-1}}^2}{\|\f{Ap}_{t-1}\|^2} - 2\frac{\lr{\f{r}_{t-1},\f{Ar}_{t-1}}^2}{\|\f{Ap}_{t-1}\|^2} = \|\f{r}_{t-1}\|^2 - \frac{\lr{\f{r}_{t-1},\f{Ar}_{t-1}}^2}{\|\f{Ap}_{t-1}\|^2}\\ 
        &< \|\f{r}_{t-1}\|^2,
	\end{align*}
	which implies that the norm of the residuals in \cref{alg:cr} (or \cref{alg:cr_pis}) monotonically decreasing. On the other hand, MINRES also guarantees a monotonic decrease of the residual norms. The equivalence between CR and MINRES for PD systems is well-studied \cite{fong2012cg}. Next, we show that, in the absence of an unlucky breakdown, this equivalence indeed extends beyond PD settings to all systems.

	\subsubsection{Equivalence to MINRES}
	\label{sec:CR=MR}
	The minimum residual (MINRES) method of Paige and Saunders \cite{paige1975solution} is a versatile approach for solving any linear systems/least-squares problems involving symmetric matrices. The iterates are constructed in a way that prevents unlucky breakdowns, guaranteeing the algorithm's convergence to a normal solution. In fact, \cref{fact:convergence_consistent,thm:cr:Ar} hold true for MINRES without the necessity to exclude any unfortunate situations, such as those specified in \cref{assumpt:unlucky}.
	
	Recall that the $ k\th $ iteration of MINRES can be formulated as a solution to \cref{eq:minres_least_squares} for any symmetric matrix $\f{A}$. Unlike CR, MINRES explicitly builds an orthonormal basis of the Krylov subspace $ \krylov{A}{b}{k} $ through a procedure called the Lanczos process. With $ \f{v}_{1} = \f{b}/\vnorm{\f{b}} $, after $k$ iterations of the Lanczos process, and in the absence of round-off errors, the Lanczos vectors form an orthonormal matrix $ \f{V}_{k+1} = \begin{bmatrix} \f{v}_1  & \f{v}_{2} & \dots &  \f{v}_{k+1} \end{bmatrix} \in \mathbb{R}^{d \times (k+1)} $, whose columns span $\krylov{A}{b}{k+1}$, and satisfy $\f{A} \f{V}_{k} = \f{V}_{k+1} \tilde{\f{T}}_{k}$ where $ \tilde{\f{T}}_{k} \in \mathbb{R}^{(k+1) \times k}$ is an upper Hessenberg matrix of the form $\tilde{\f{T}}_{k}^{\T} = \begin{bmatrix}
		\f{T}_{k}^{\T}&
		\beta\f{e}_{k}
	\end{bmatrix}$, and $\f{T}_k \in \mathbb{R}^{k \times k}$, $\beta$ and $\f{e}_k$ are the Lanczos tridiagonal matrix, some real number constant and a $ k\th $ unit vector respectively.
	Defining $\f{x} = \f{V}_{k}\f{y}$ for $ \f{y} \in \mathbb{R}^{k} $, one can show that \cref{eq:minres_least_squares} can be simplified to 
	\begin{align*}
		\f{y}_{k} \triangleq \argmin{\f{y} \in \mathbb{R}^{k}} \frac{1}{2}\left\|\|\f{b}\|\f{e}_1 - \tilde{\f{T}}_{k} \f{y}\right\|^2.
	\end{align*}
	In essence, the above problem is solved using the QR factorization of  $\tilde{\f{T}}_{k}$. However, by exploiting its sparsity structure, we can simplify the computations and obtain an update procedure of the form $\f{x}_k = \f{x}_{k-1} + \tau_{k-1}\f{d}_{k-2}$, for some step size $\tau_{k-1}$ and the update direction $\f{d}_k$; see the works \cite{paige1975solution,liu2022minres} for more details.
	
	For positive definite matrices, it is well-known that CR and MINRES, mathematically speaking, generate identical iterates \cite{fong2012cg}. However, beyond PD settings, such an equivalence, at least to our knowledge, has not been thoroughly studied. In these general settings, the $\f{A}$-weighted product $ \lr{.,.}_{\f{A}} $ in CR is no longer a valid inner product and the equivalence between CR and MINRES is no longer automatic. In this section, we study the existence of such an equivalence in indefinite and/or inconsistent settings, and show that, under \cref{assumpt:unlucky}, both CR and MINRES generate the same sequence of iterates for any symmetric systems.
	
	To that end, we first demonstrate that, just like MINRES, subproblems of CR can also be formulated as \cref{eq:minres_least_squares} for any symmetric matrix $ \f{A} $. Let $\f{P}_{k} = \begin{bmatrix}
		\f{p}_0 & \cdots & \f{p}_{k-1}
	\end{bmatrix}$ and $\f{R}_{k} = \begin{bmatrix}
		\f{r}_0 & \cdots & \f{r}_{k-1}
	\end{bmatrix}$ be the matrices whose columns are, respectively, the update directions and the residual vectors from CR. Using \cref{prop:properties}-\labelcref{property:cgcr} and noting that $\f{r}_{0} = \f{b}$, we can write $\f{x} = \f{P}_{k}\f{z}$ for $ \f{z} \in \mathbb{R}^{k} $ and
	\begin{align*}
		\min_{\f{x} \in \krylov{A}{b}{k}}\frac{1}{2}\|\f{b} - \f{Ax}\|^2 & = \min_{\f{z} \in \mathbb{R}^{k}}\frac{1}{2}\left\|\f{R}_{k}\f{e}_1 - \f{AP}_{k}\f{z}\right\|^2,
	\end{align*}
	which gives
	\begin{align}
		\label{eq:z}
		\f{z}_{k} & \triangleq \argmin{\f{z} \in \mathbb{R}^{k}} \frac{1}{2}\lr{\f{z}, \f{P}_{k}^{\T}\f{A}^2\f{P}_{k}\f{z}} - \lr{\f{z},\f{P}_{k}^{\T}\f{A}\f{R}_{k}\f{e}_1}.
	\end{align}
	By \cref{prop:properties}-\labelcref{property:cr}, the matrix $\f{D}_{k} \triangleq \f{P}_{k}^{\T}\f{A}^2\f{P}_{k}$ is diagonal, i.e.,
	\begin{align*}
		\f{D}_{k} = \begin{bmatrix}
			\|\f{Ap}_0\|^2 & & & \\
			& \|\f{Ap}_1\|^2 & & \\
			& & \ddots & \\
			& & & \|\f{Ap}_{k-1}\|^2
		\end{bmatrix}.
	\end{align*}
	Note that, by \cref{lem:cr:Ap}, the matrix $\f{D}_k$ is always invertible prior to termination. Also, from Properties \labelcref{prop:properties}-\labelcref{property:cgcr} and \labelcref{prop:properties}-\labelcref{property:cr}, it follows that the matrix $\f{K}_{k} \triangleq \f{P}_{k}^{\T}\f{A}\f{R}_{k}$ is lower triangular, i.e.,
	\begin{align*}
		\f{K}_{k} = \begin{bmatrix}
			\lr{\f{r}_0, \f{Ap}_0} & 0 & \cdots & 0 \\
			\lr{\f{r}_0, \f{Ap}_1} & \lr{\f{r}_1, \f{Ap}_1} & \cdots & 0 \\
			\vdots & \vdots & \ddots & \vdots\\
			\lr{\f{r}_0, \f{Ap}_{k-1}} & \lr{\f{r}_1, \f{Ap}_{k-1}} & \cdots & \lr{\f{r}_{k-1}, \f{Ap}_{k-1}}
		\end{bmatrix}.
	\end{align*}
	Optimality condition for \cref{eq:z} implies $\f{z}_{k}  = \f{D}_{k}^{-1}\f{K}_{k}\f{e}_1$. The vector $\f{K}_{k}\f{e}_1$ can be written as,
	\[\f{K}_{k}\f{e}_1 = \begin{bmatrix}
		\lr{\f{r}_0, \f{Ap}_0}\\
		\lr{\f{r}_0, \f{Ap}_1}\\
		\vdots\\
		\lr{\f{r}_0, \f{Ap}_{k-1}}
	\end{bmatrix} = \begin{bmatrix}
		\lr{\f{r}_0, \f{Ar}_0}\\
		\beta_0\lr{\f{r}_0, \f{Ar}_0}\\
		\vdots\\
		\prod_{i=1}^{k}\beta_{i-1}\lr{\f{r}_0, \f{Ar}_0}
	\end{bmatrix} = \begin{bmatrix}
		\lr{\f{r}_0, \f{Ar}_0}\\
		\lr{\f{r}_1, \f{Ar}_1}\\
		\vdots\\
		\lr{\f{r}_{k-1}, \f{Ar}_{k-1}}
	\end{bmatrix}\]
	where the last equality follows from the definition of $\beta$ in \cref{alg:cr}. Together, we get
	\[\f{z}_{k} = \f{D}_{k}^{-1}\f{K}_{k}\f{e}_1 = \begin{bmatrix}
		\lr{\f{r}_0, \f{Ar}_0}/\|\f{Ap}_0\|^2\\
		\lr{\f{r}_1, \f{Ar}_1}/\|\f{Ap}_1\|^2\\
		\vdots\\
		\lr{\f{r}_{k-1}, \f{Ar}_{k-1}}/\|\f{Ap}_{k-1}\|^2
	\end{bmatrix} = \begin{bmatrix}
		\alpha_0\\
		\alpha_1\\
		\vdots\\
		\alpha_{k-1}
	\end{bmatrix},\]
	which are exactly the same step-sizes, $\alpha$, generated by \cref{alg:cr}. This implies that, just like MINRES, CR's $ k\th $ iterate is also a minimizer of \cref{eq:minres_least_squares}.

    A direct implication of the above result is that CR and MINRES will, in exact arithmetic, generate the same sequence of iterates, $\{\f{x}_{k-1}\}_{k=1}^{k=g}$, for consistent but potentially singular systems. 
    \begin{lemma}
		\label{lem:consistent_mr=cr}
    Let $ \f{b} \in \Range(\f{A}) $. Then, under \cref{assumpt:unlucky}, CR and MINRES generate the same sequence of iterates.
    \end{lemma}
    \begin{proof}
        As showed above, both iterates of CR and MINRES correspond to solutions of \cref{eq:minres_least_squares}. It suffices to show that, for consistent systems, \cref{eq:minres_least_squares} admits a unique solution for any $ 1\leq k \leq g $. Let $\f{W}_k$ be any basis of $\krylov{A}{b}{k}$. Defining $\f{x} = \f{W}_{k} \f{y}$ for $ \f{y} \in \Rd $, we can write \cref{eq:minres_least_squares} as
        \begin{align*}
            \argmin{\f{y} \in \mathbb{R}^{k}}\frac{1}{2}\left\|\f{b} - \f{AW}_{k}\f{y}\right\|^2 = \argmin{\f{y} \in \mathbb{R}^k}\frac{1}{2}\lr{\f{y},\f{W}_k^{\T}\f{A}^2\f{W}_k\f{y}} - \lr{\f{y},\f{W}_k^{\T}\f{Ab}}.
        \end{align*}
        By the assumption on $ \f{b} $, $\Span(\f{W}_k) \subseteq \Range(\f{A})$, which implies $ \f{W}_k^{\T}\f{A}^2\f{W}_k \succ \zero $. Hence, the above problem admits a unique solution $\f{y}_{k}$, resulting in a unique $\f{x}_{k}$.
    \end{proof}
    Hence, when $\f{b} \in \Range(\f{A})$, there can be only one unique minimizer of \cref{eq:minres_least_squares} in $\krylov{A}{b}{k}$ for $1 \leq k \leq g$. However, the case for inconsistent systems may appear less straightforward. Specifically, when $\f{b} \not\in \Range(\f{A})$, there may be more than one solution to \cref{eq:minres_least_squares}.
    The crucial step, now, is to show that, despite possibly many solutions to \cref{eq:minres_least_squares}, the minimizer from CR is indeed the same as the one from MINRES, i.e., the sequence of iterates $\{\f{x}_{k-1}\}_{k=1}^{k=g}$ for both algorithms, in exact arithmetic, coincide. 
    To do this, we first establish a counterpart of \cref{lem:cr_tilde} for MINRES (\cref{lem:inconsisent_mr_same}) and then leverage \cref{lem:consistent_mr=cr,lem:cr_tilde,lem:inconsisent_mr_same} to establish our claim (\cref{thm:mr=cr}). A diagram depicting our proof strategy is given in \cref{fig:diagram}.
	\begin{figure}
		\begin{center}
			\begin{tikzpicture}
				\matrix (m) [matrix of math nodes,row sep=5em,column sep=18em, minimum width=2em]
				{
					\text{CR}(\f{A}, \ft{b}) & \text{MINRES}(\f{A},\ft{b}) \\
					\text{CR}(\f{A}, \f{b}) &  \text{MINRES}(\f{A},\f{b}) \\
				};
				\path[-stealth]
				(m-1-1) edge [<->] node [left] {\cref{lem:cr_tilde}} 
				node [right] {$\f{A}^\dagger\f{A}\f{x}_{k-1}^{\text{CR}} = \ft{x}_{k-1}^{\text{CR}}$} (m-2-1)
				edge [<->, dashed] node [above] {\cref{lem:consistent_mr=cr}} 
				node [below] {$\ft{x}_{k-1}^{\text{CR}} = \ft{x}_{k-1}^{\text{MINRES}}$} (m-1-2)
				(m-1-2) edge [<->, dashed] node [right] {\cref{lem:inconsisent_mr_same}} 
				node [left] {$\f{A}^\dagger\f{A}\f{x}_{k-1}^{\text{MINRES}} = \ft{x}_{k-1}^{\text{MINRES}}$} (m-2-2)
				(m-2-1) edge [<->, dashed] node [below] {\cref{thm:mr=cr}}
				node [above] {$\f{x}_{k-1}^{\text{CR}} = \f{x}_{k-1}^{\text{MINRES}}$} (m-2-2);
			\end{tikzpicture}
		\end{center}
		\caption{Our proof strategy for establishing equivalence between MINRES and CR. To prove \cref{thm:mr=cr}, we use \cref{lem:consistent_mr=cr,lem:cr_tilde,lem:inconsisent_mr_same}.}\label{fig:diagram}
	\end{figure}
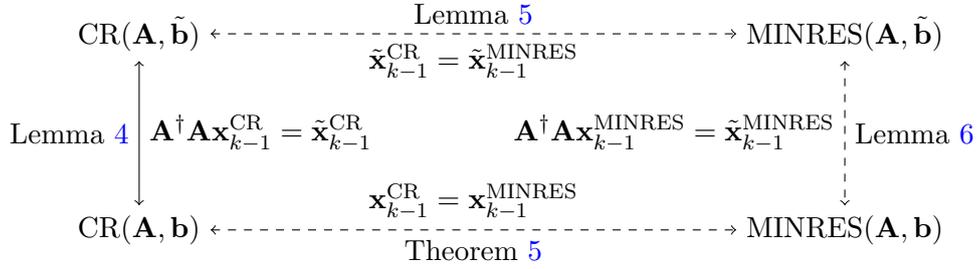
	
	\begin{lemma}\label{lem:inconsisent_mr_same}
		Consider the iterates generated by the procedures MINRES$(\f{A}, \ft{b})$ and MINRES$(\f{A}, \f{b})$, and let where $g$ be the grade of $\f{b}$ with respect to $\f{A}$. For $1 \leq k \leq g$, we have $\ft{x}_{k-1} = \f{A}^\dagger\f{A}\f{x}_{k-1}$.
	\end{lemma}
	\begin{proof}
		We first note that, by Pythagorean theorem, $\|\f{b} - \f{Ax}\|^2 = \|\ft{b} - \f{Ax}\|^2 + \|\f{b}^\perp\|^2$. Hence, the subproblem associated with the $ k\th $ iterate from MINRES$(\f{A}, \f{b})$ can be equivalently written as 
		\begin{align*}
			\min_{\f{x} \in \krylov{A}{b}{k}}\|\ft{b} - \f{Ax}\|^2,
		\end{align*}
		which using $ \f{A} = \f{A}\f{A}^{\dagger}\f{A} $ gives 
		\begin{align*}
			\min_{\f{x} \in \krylov{A}{b}{k}}\|\ft{b} - \f{A}\f{A}^{\dagger}\f{A}\f{x}\|^2.
		\end{align*}
		Denoting $ \ft{x} = \f{A}^{\dagger}\f{A}\f{x}  $, and using the fact that $\f{A}\f{b} = \f{A}\ft{b}$, it follows that $ \ft{x} \in \f{A}^{\dagger}\f{A} \krylov{A}{b}{k} = \mathcal{K}_k(\f{A},\ft{b}) $. Hence, \cref{eq:minres_least_squares} can equivalently be written as 
		\begin{align*}
			\min_{\ft{x} \in \mathcal{K}_k(\f{A},\ft{b})}\|\ft{b} - \f{A}\ft{x}\|^2,
		\end{align*}
		which is precisely the subproblem associated with the $ k\th $ iterate from MINRES$(\f{A}, \ft{b})$. Hence, it naturally follows that $\f{A}^\dagger\f{A}\f{x}_{k-1} = \ft{x}_{k-1}$.
	\end{proof}
	\begin{theorem}\label{thm:mr=cr}
		Under \cref{assumpt:unlucky}, CR and MINRES generate the same sequence of iterates.
	\end{theorem}
	\begin{proof}
		\cref{lem:consistent_mr=cr} establishes this result for the consistent case. So consider the inconsistent setting. Note that, in this case, since $ \tilde{g} = g - 1 $ (see \cref{rem:grade}), the vectors $\{\ft{b}, \f{A}\ft{b}, \cdots, \f{A}^{k-1}\ft{b}\}$ are linearly independent for $1 \leq k \leq g-1$. Let $ \f{x}_{k-1}^{\text{MINRES}} $ and $ \f{x}_{k-1}^{\text{CR}} $ be the $ k\th $ iterates from the procedures MINRES$(\f{A}, \f{b})$ and CR$(\f{A}, \f{b})$, respectively. Clearly, $ \f{x}_{0}^{\text{MINRES}} =  \f{x}_{0}^{\text{CR}} = \zero $. So consider any $ 2 \leq k \leq g $. Since the subproblem of the both procedures can be formulated as \cref{eq:minres_least_squares}, we have
		\begin{align*}
			\f{x}_{k-1}^{\text{MINRES}} &=\sum_{i=1}^{k-1}a_{i-1}\f{A}^{i-1}\f{b}, \quad \text{and} \quad \f{x}_{k-1}^{\text{CR}} =\sum_{i=1}^{k-1}b_{i-1}\f{A}^{i-1}\f{b}.
		\end{align*}
		Using   \cref{lem:consistent_mr=cr,lem:cr_tilde,lem:inconsisent_mr_same}, we have $\f{A}^\dagger\f{A} \f{x}_{k-1}^{\text{MINRES}} = \f{A}^\dagger\f{A} \f{x}_{k-1}^{\text{CR}}$, which implies
		\begin{align*}
			\sum_{i=1}^{k-1}(a_{i-1} - b_{i-1}) \f{A}^{i-1}\ft{b} =\zero.
		\end{align*}        
		Now, from linear independence of $ \{\f{A}^{i-1}\ft{b}\}_{i=1}^{k-1} $, we get $ a_{i-1} = b_{i-1} $ for $1 \leq i \leq k-1$.
	\end{proof}
	\cref{thm:mr=cr} implies that all properties established for MINRES in prior works, such as non-positive curvature detection,          monotonicity, and lifting procedures in \cite{liu2022minres,liu2023obtaining}, are also readily applicable to CR.
    
    After establishing the equivalence between MINRES and CR under \cref{assumpt:unlucky}, it is of theoretical interest to investigate the behavior of MINRES when CR experiences an unlucky breakdown. To address this, we present the following two lemmas that characterize this behavior.
    \begin{lemma}\label{lem:cr_rAr=0_dependent}
        In MINRES/CR, for $ k \leq g - 1 $, the residual vectors $\{\mathbf{r}_0, \mathbf{r}_1, \cdots, \mathbf{r}_{k}\} $ are linearly independent if and only if $ \langle \mathbf{r}_{i-1}, \mathbf{Ar}_{i-1} \rangle \neq 0 $  for  $ 1 \leq i \leq k $.
    \end{lemma}
    \begin{proof}
        First, we note that $\f{r}_i \not= \f{0}$, for $0 \leq i \leq g - 1$. Now, let $\lr{\f{r}_{i-1},\f{Ar}_{i-1}} \not= 0$, for $1 \leq i \leq k$, and assume that $\{\f{r}_0, \f{r}_1, \cdots, \f{r}_{k}\}$ are linearly dependent. Then, there must exist scalars, $c_i$, not all zero, such that $\sum_{i=0}^{k} c_i\f{r}_i = \f{0}$. However, by \cref{prop:properties}-\labelcref{property:cr},  $c_i\lr{\f{r}_i,\f{Ar}_i} = 0, \; 0 \leq i \leq k-1$, which by assumption implies that $c_i = 0$, for all $0 \leq i \leq k - 1$. So, we must have $c_k\f{r}_k = \f{0}$. But since $\f{r}_k\not=0$, we must have $c_k = 0$, which leads to a contradiction. This implies that $\{\f{r}_0, \f{r}_1, \cdots, \f{r}_{k}\}$ are linearly independent vectors. Conversely, suppose  $\lr{\f{r}_i, \f{Ar}_i} = 0$ for some $0 \leq i \leq k-1$. From \cref{alg:cr}, since $\f{r}_{i+1} = \f{r}_i - \alpha_i\f{Ap}_i$ and $\alpha_i = \lr{\f{r}_i,\f{Ar}_i}/\|\f{Ap}_i\|^2 = 0$, it follows that $\f{r}_{i+1} = \f{r}_i$.
    \end{proof}
    Note that if $\lr{\f{r}_k,\f{Ar}_k} = 0$ for the first time, the residual vectors $\{\f{r}_0,\f{r}_1,\cdots,\f{r}_k\}$ are still linearly independent.
    \begin{lemma}
        In MINRES/CR, let $ k \leq g - 1 $. The condition $ \langle \mathbf{r}_{k-1}, \mathbf{A}\mathbf{r}_{k-1} \rangle = 0 $ (indicating an unlucky breakdown for CR) occurs for the first time if and only if the Lanczos tridiagonal matrix $ \mathbf{T}_k $ becomes singular for the first time.
    \end{lemma}
    \begin{proof}
        Let $\f{V}_k \in \mathbb{R}^{d \times k}$ be the Lanczos matrix with orthonormal columns, generated within MINRES, that span $\krylov{A}{b}{k}$. Recall that $\f{V}_k^\top\f{AV}_k = \f{T}_k \in \mathbb{R}^{k \times k}$, where $\f{T}_k$ is tridiagonal. Let $\f{R}_k = \begin{bmatrix}
            \f{r}_0 & \f{r}_1 & \cdots & \f{r}_{k-1}
        \end{bmatrix} = \f{V}_k\f{Z}_{k}$ for some $\f{Z}_{k} \in \mathbb{R}^{k \times k}$. Consider the following expression,
        \begin{align*}
            \begin{bmatrix}
                \lr{\f{r}_0,\f{Ar}_0} & & & \\
                & \lr{\f{r}_1,\f{Ar}_1} & & \\
                & & \ddots & \\
                & & & \lr{\f{r}_{k-1},\f{Ar}_{k-1}}
            \end{bmatrix} = \f{R}_{k}\f{AR}_k = \f{Z}_{k}^\top\f{V}_k^\top\f{AV}_k\f{Z}_{k} = \f{Z}_{k}^\top\f{T}_k\f{Z}_{k}.
        \end{align*}
        where the first equality follows from \cref{prop:properties}-\labelcref{property:cr}. By  \cref{lem:cr_rAr=0_dependent}, as long as  $\lr{\f{r}_{i},\f{Ar}_{i}} \neq 0$ for $0 \leq i \leq k-2$, the columns of $\f{R}_k$ are linearly independent and $\f{Z}_{k}$ must be full rank. Hence, $\lr{\f{r}_{k-1},\f{Ar}_{k-1}} = 0$ if and only if $\f{T}_k$ is singular.
    \end{proof}\label{sec:inconsistent_system}

    \section{Numerical Experiments}\label{sec:num_exp}
    In this section, we conduct a series of experiments to verify aspects of the theory developed in this paper and to explore the empirical performance of the methods discussed in various settings. While numerical inaccuracies due to finite precision arithmetic in iterative methods have long been well known and thoroughly studied, our main focus is to highlight the additional and unexpectedly catastrophic instability of CG in inconsistent settings--behavior not typically observed in consistent settings--and compare it to the CR alternative. In particular, we set out to demonstrate the following:  
    \begin{enumerate}  
        \item CR is generally more stable than CG in most cases, regardless of consistency.  
        \item Unlike CR, CG exhibits extreme instability under inconsistent systems.  
        \item Unlike CR, CG can be unstable even for PD matrices with a few small eigenvalues.  
    \end{enumerate}
    In \cref{sec:synthetic}, using synthetic examples, our aim is to empirically verify \cref{lem:cg:Ap,thm:cg:Ar}. Additionally, we will investigate the stability of these methods and identify instances where the theoretical predictions deviate significantly from practical observations. Moving to \cref{sec:real-life}, our focus remains on the same objectives, but within the context of real-world applications. It is crucial to note that our primary goal here is to observe the behavior of these methods in their original form, without incorporating any additional enhancements such as reorthogonalization strategies (if applicable). 
	
	Algorithms are terminated either upon reaching a maximum number of iterations or when a specific termination condition is met. For CG, this condition is chosen as $\min\{\|\f{r}_k\| / \|\f{b}\|, \|\f{Ap}_k\| / \|\f{Ab}\|\}$ (cf.\ \cref{fact:convergence_consistent,lem:cg:Ap}), and for CR/MINRES, it is $\min\{\|\f{r}_k\|/\|\f{b}\|, \|\f{Ar}_k\|/\|\f{Ab}\|\}$ (cf.\ \cref{fact:convergence_consistent,thm:cr:Ar}).
 
	\subsection{Synthetic Problems}\label{sec:synthetic}
	In this section, we experiment with a series of simple and synthetic diagonal matrices to verify \cref{lem:cg:Ap,thm:cg:Ar} and to explore the stability of these methods in several instances. To validate \cref{lem:cg:Ap,thm:cg:Ar}, we focus on small-size problems in \cref{sec:convergence_cg} to avoid any unwanted numerical issues. Subsequently, we extend our analysis to larger-size matrices in \cref{sec:comparison,sec:instability_cg} to investigate the stability of these methods across a substantial number of iterations. 
	
	For our experiments, consider various matrices that are generated according to the following template:
	\begin{itemize}
		\item PSD Matrix: A diagonal matrix generated by sampling from a uniform distribution, with $m$ zero diagonals.
		\item Indefinite Matrix: A diagonal matrix generated by sampling from a normal distribution, with $m$ zero diagonals.
		\item (Slightly) Indefinite Matrix: A diagonal matrix sampled from a uniform distribution, with $m$ zero diagonals. Additionally, $\lceil 0.1 \times g \rceil$ of these diagonals have their signs flipped to negative, where $g$ represents the grade.
	\end{itemize}
	The vector $\f{b} \in \Rd$ is sampled from a normal distribution. The examples in this section are constructed in such a way that they are always inconsistent, unless otherwise stated.
	
	\subsubsection{CG at Termination: Validating \texorpdfstring{\cref{lem:cg:Ap,thm:cg:Ar}}{}}\label{sec:convergence_cg}
	In this section, we validate \cref{lem:cg:Ap,thm:cg:Ar}. For this, we consider $d = 10$, and generate six matrices by randomly zeroing out $m = 5, 4, 3, 2, 1$ and $0$ diagonal elements of $\f{A}$.
	The results are depicted in \cref{fig:cg_small_matrices,fig:cg_small_matrices_indefinite} where we consider PSD and indefinite settings, respectively. In all these examples, as predicted by \cref{lem:cg:Ap}, $\|\f{Ap}_{k}\|$ indeed vanishes at termination, while $\|\f{Ar}_{k}\|$ remains bounded away from zero in all but consistent case, validating \cref{thm:cg:Ar}. 
	 
    \begin{figure}[htbp]
		\centering
		\begin{subfigure}{0.32\textwidth}
			\includegraphics[width=\linewidth]{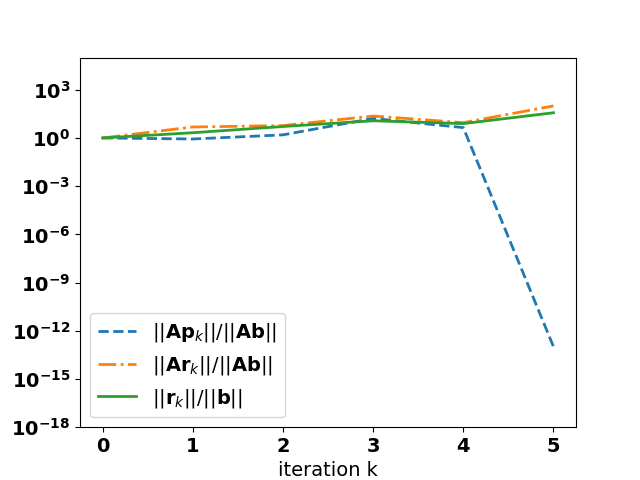}
		\end{subfigure}
		\begin{subfigure}{0.32\textwidth}
			\includegraphics[width=\linewidth]{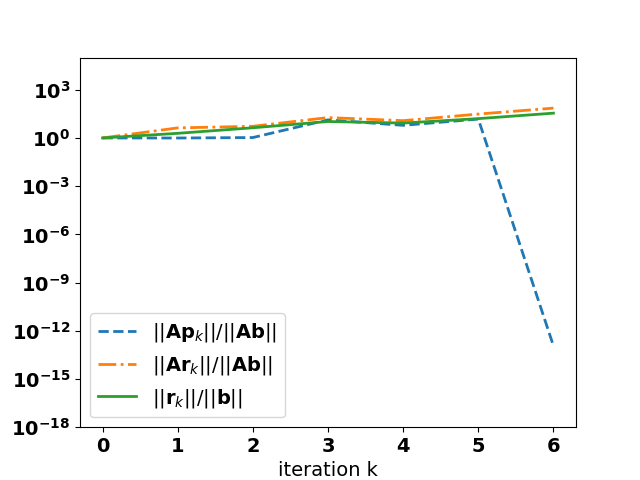}
		\end{subfigure}
		\begin{subfigure}{0.32\textwidth}
			\includegraphics[width=\linewidth]{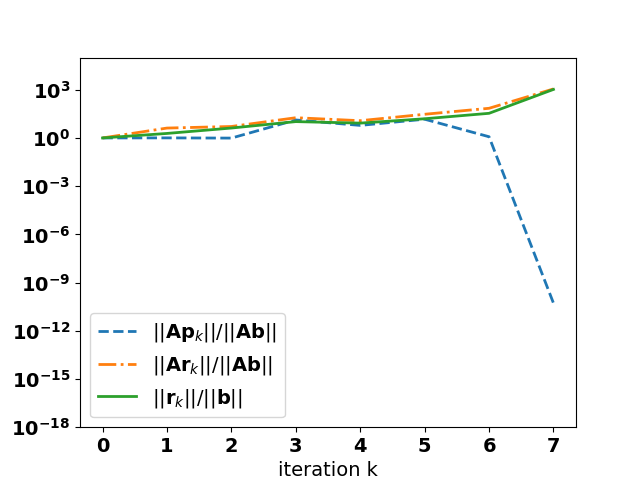}
		\end{subfigure}
		\begin{subfigure}{0.32\textwidth}
			\includegraphics[width=\linewidth]{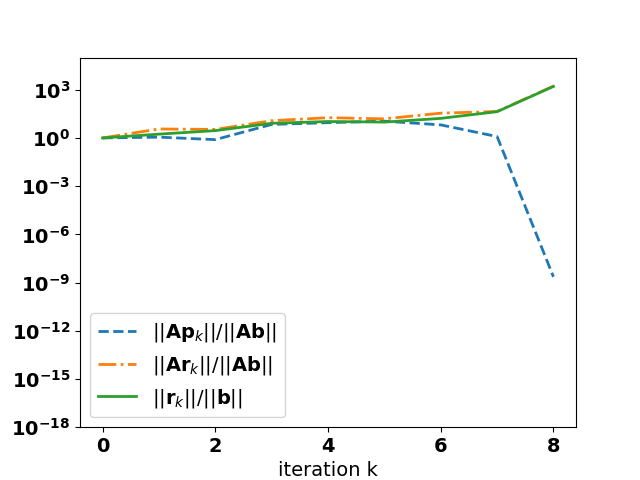}
		\end{subfigure}
		\begin{subfigure}{0.32\textwidth}
			\includegraphics[width=\linewidth]{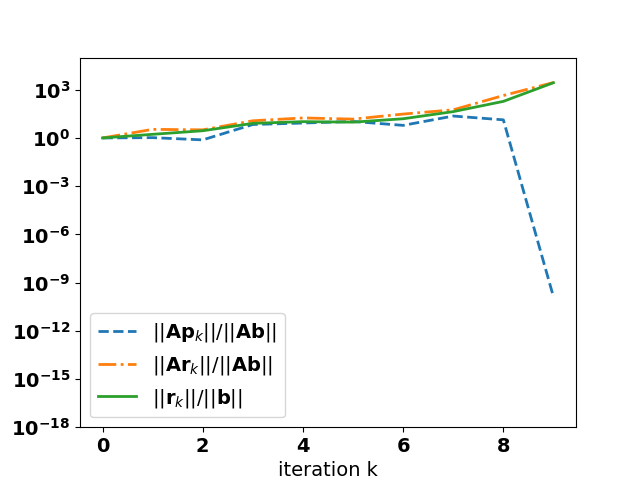}
		\end{subfigure}
		\begin{subfigure}{0.32\textwidth}
			\includegraphics[width=\linewidth]{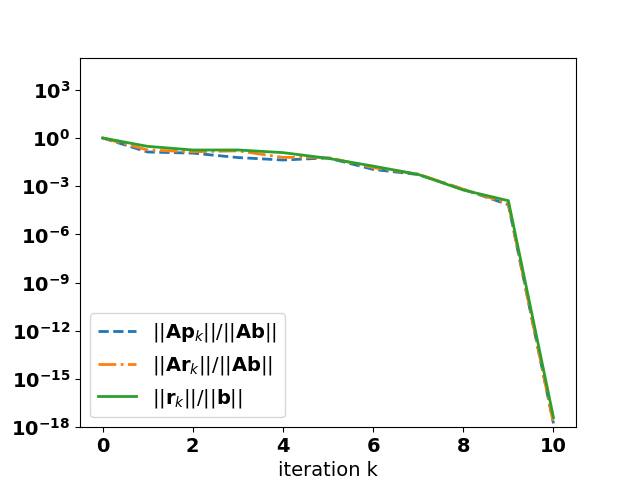}
		\end{subfigure}
		\caption{Experiments for \cref{sec:convergence_cg}  to verify \cref{lem:cg:Ap,thm:cg:Ar}. The matrices are PSD with $d = 10$.  From the top left to the bottom right, the grades are 6, 7, 8, 9, 10 and 10. The bottom right is a positive definite matrix.}
		\label{fig:cg_small_matrices}
	\end{figure}
	
	\begin{figure}[htbp]
		\centering
		\begin{subfigure}{0.32\textwidth}
			\includegraphics[width=\linewidth]{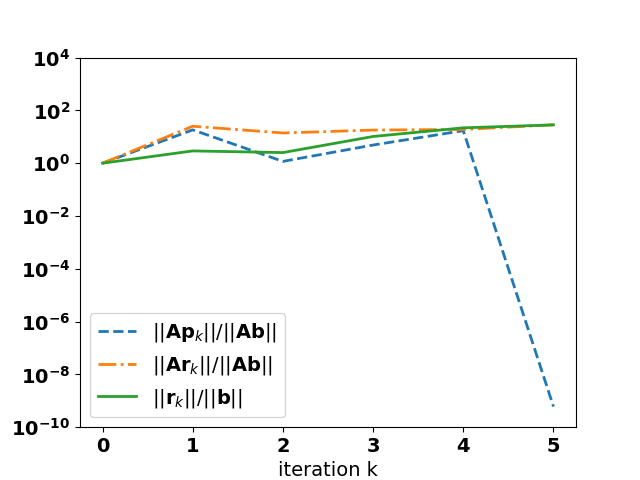}
		\end{subfigure}
		\begin{subfigure}{0.32\textwidth}
			\includegraphics[width=\linewidth]{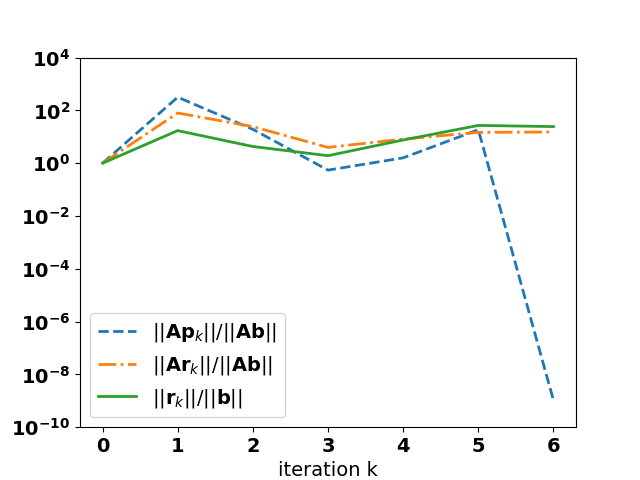}
		\end{subfigure}
		\begin{subfigure}{0.32\textwidth}
			\includegraphics[width=\linewidth]{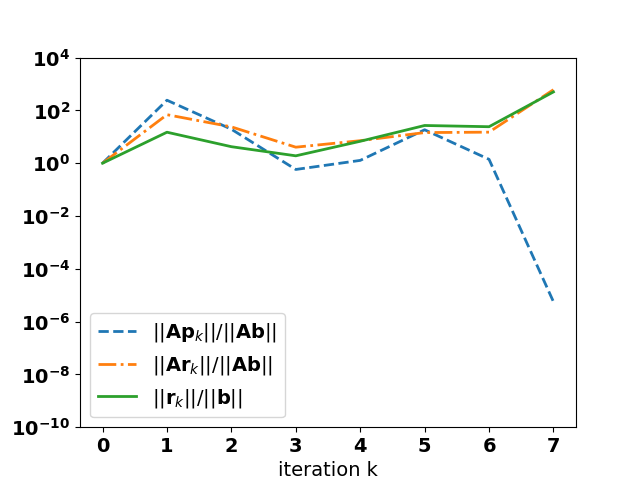}
		\end{subfigure}
		\begin{subfigure}{0.32\textwidth}
			\includegraphics[width=\linewidth]{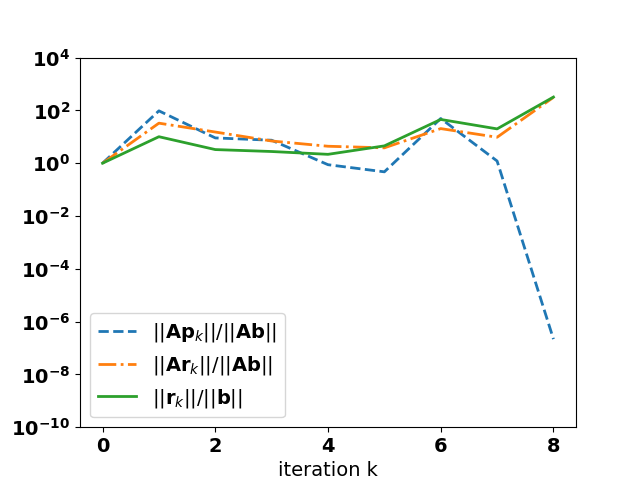}
		\end{subfigure}
		\begin{subfigure}{0.32\textwidth}
			\includegraphics[width=\linewidth]{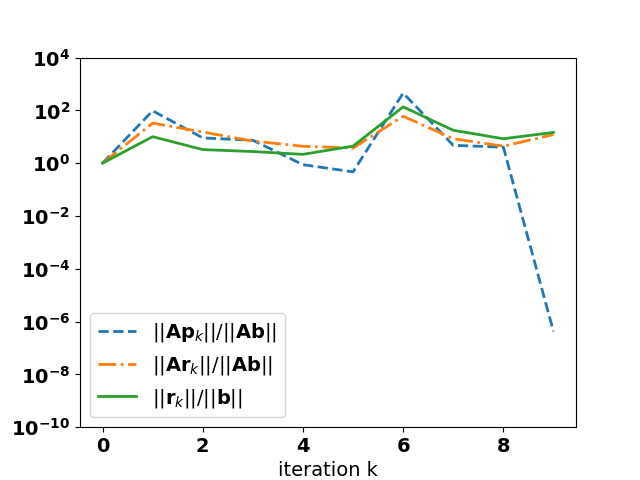}
		\end{subfigure}
		\begin{subfigure}{0.32\textwidth}
			\includegraphics[width=\linewidth]{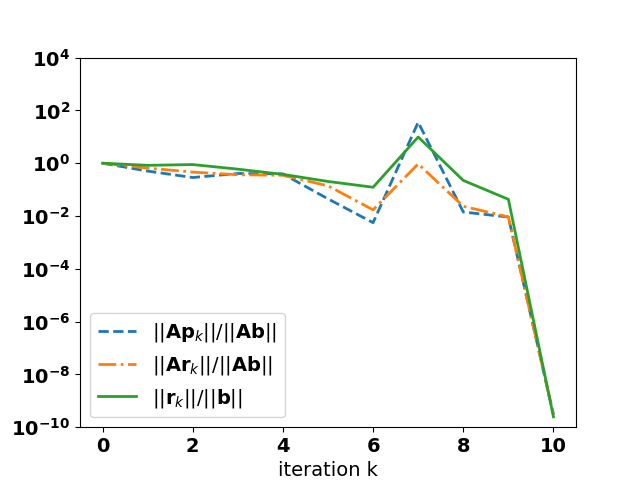}
		\end{subfigure}
		\caption{Experiments for \cref{sec:convergence_cg} to verify \cref{lem:cg:Ap,thm:cg:Ar}. The matrices are indefinite with $d = 10$.  From the top left to the bottom right, the grades are 6, 7, 8, 9, 10 and 10. The bottom right is a full rank indefinite matrix.}
		\label{fig:cg_small_matrices_indefinite}
	\end{figure}

 
	\subsubsection{Numerical Instability of CG}\label{sec:instability_cg}
	In this section, we consider larger size matrices than those in \cref{sec:convergence_cg} to explore the behavior of CG across a larger number of iterations. Specifically, we set $ d = 100 $ and generate a series of matrices by zeroing out $m = 90, 80, 60, 20, 1$ and $0$ diagonal entries. The maximum number of iterations allowed is capped at $g$, i.e., the corresponding grade for each example, and the tolerance for the termination condition $\min\{\|\f{r}_k\| / \|\f{b}\|, \|\f{Ap}_k\| / \|\f{Ab}\|\}$ is $10^{-8}$. 
	The results are gathered in \cref{fig:cgexplode_psd,fig:cgexplode_indefinite,fig:cgexplode_10p}, corresponding to PSD, indefinite, and slightly indefinite settings, respectively. In all these examples, when $g$ is small, the theoretical prediction of \cref{lem:cg:Ap} aligns with the observed behavior. However, as $g$ increases, necessitating more iterations from CG, a clear discrepancy emerges. Unless the system is consistent, the implications outlined in \cref{lem:cg:Ap} do not accurately reflect practical outcomes, highlighting the presence of substantial numerical inaccuracies within CG iterations.  Nonetheless, an interesting observation from \cref{fig:cgexplode_10p,fig:cgexplode_indefinite} is that the presence of negative eigenvalues appears to induce a form of oscillatory behavior in CG. While the depicted quantities remain bounded away from zero, they do not exhibit the pronounced growth observed in \cref{fig:cgexplode_psd}.
	
	\begin{figure}[htbp]
		\centering
		\begin{subfigure}{0.32\textwidth}
			\includegraphics[width=\linewidth]{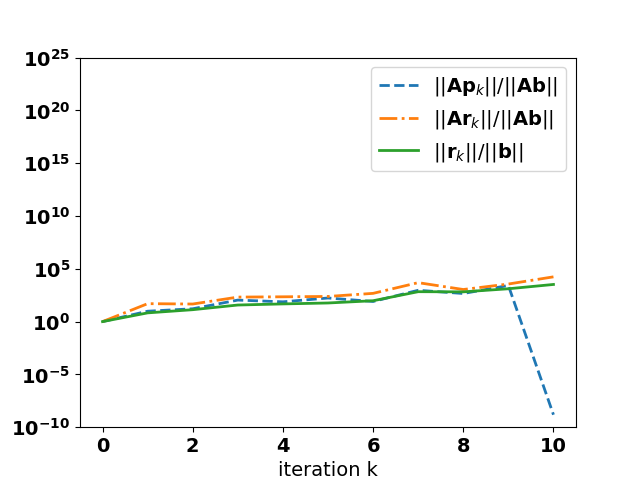}
		\end{subfigure}
		\begin{subfigure}{0.32\textwidth}
			\includegraphics[width=\linewidth]{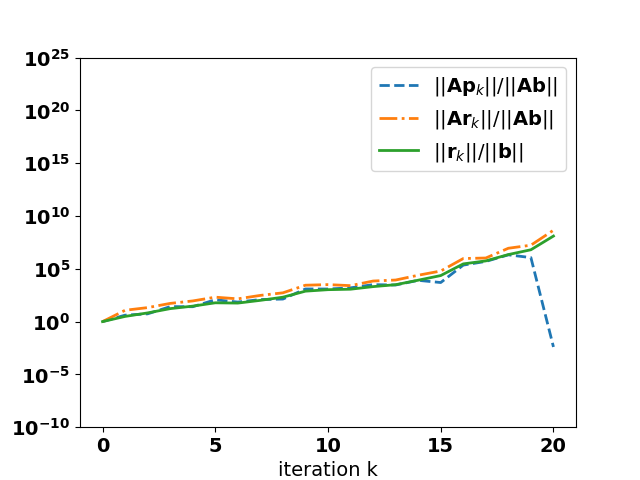}
		\end{subfigure}
		\begin{subfigure}{0.32\textwidth}
			\includegraphics[width=\linewidth]{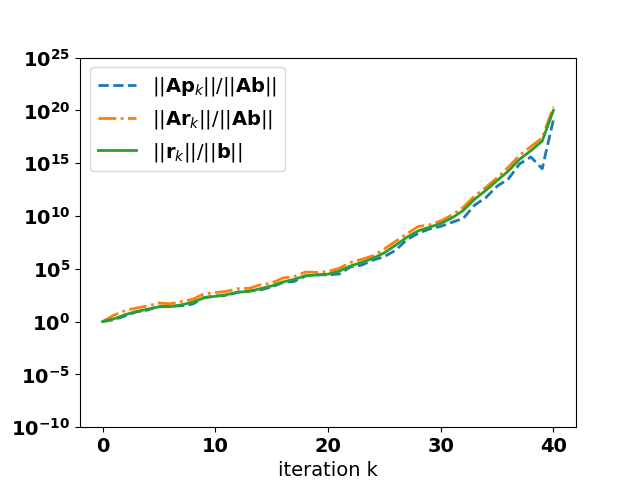}
		\end{subfigure}
		\begin{subfigure}{0.32\textwidth}
			\includegraphics[width=\linewidth]{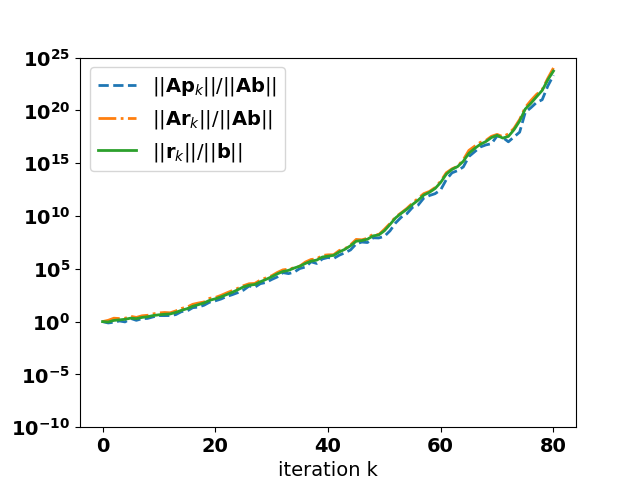}
		\end{subfigure}
		\begin{subfigure}{0.32\textwidth}
			\includegraphics[width=\linewidth]{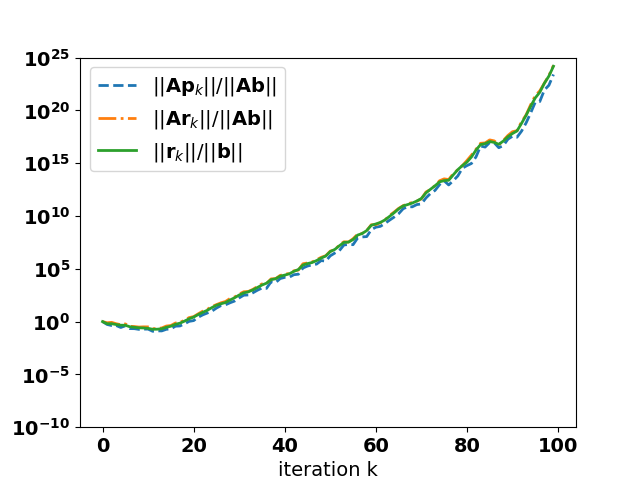}
		\end{subfigure}
		\begin{subfigure}{0.32\textwidth}
			\includegraphics[width=\linewidth]{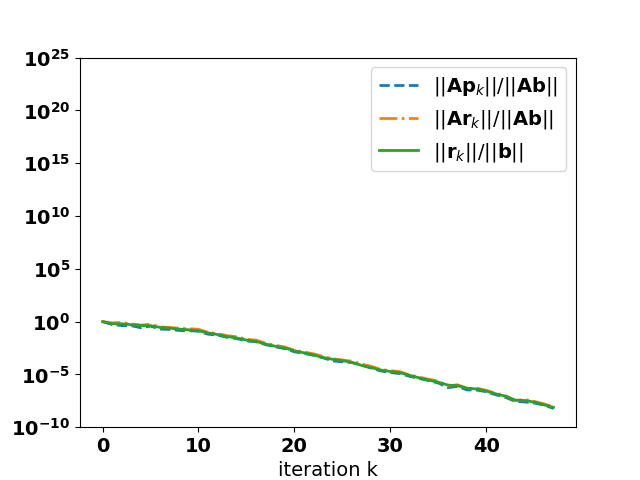}
		\end{subfigure}
		\caption{Experiments for \cref{sec:instability_cg} to explore instabilities within CG. The matrices are PSD with $d = 100$.  From top left to bottom right, the grades are 11, 21, 41, 81, 100 and 100. The bottom right is a positive definite matrix for which CG terminates at iteration 59 with $\|\f{r}_k\|/\|\f{b}\| < 10^{-8}$.}
		\label{fig:cgexplode_psd}
	\end{figure}
	
	\begin{figure}[htbp]
		\centering
		\begin{subfigure}{0.32\textwidth}
			\includegraphics[width=\linewidth]{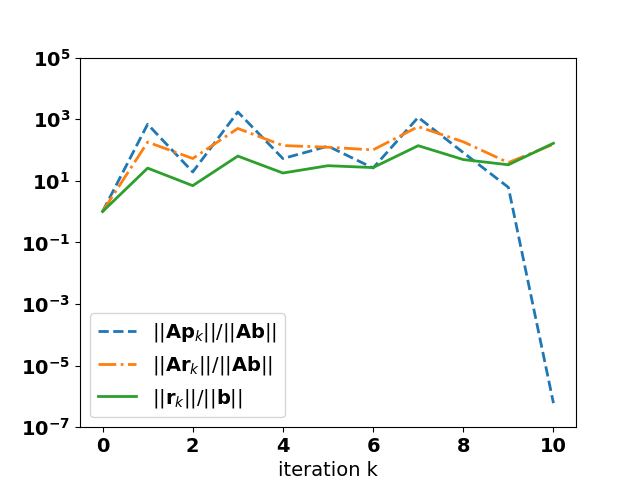}
		\end{subfigure}
		\begin{subfigure}{0.32\textwidth}
			\includegraphics[width=\linewidth]{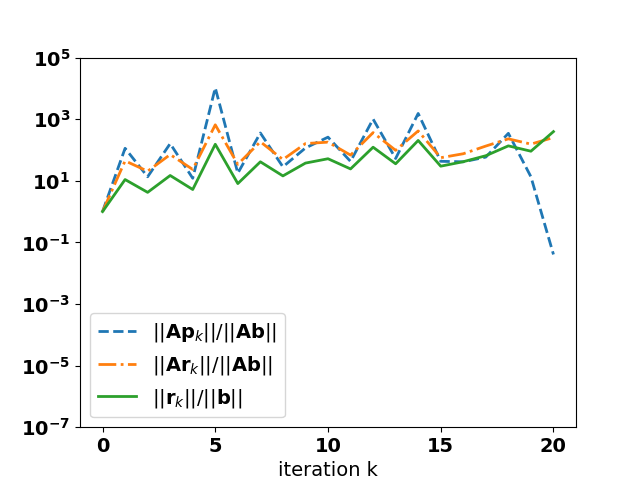}
		\end{subfigure}
		\begin{subfigure}{0.32\textwidth}
			\includegraphics[width=\linewidth]{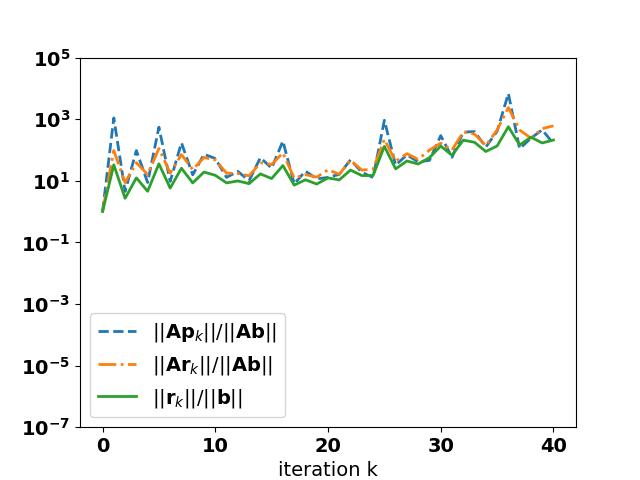}
		\end{subfigure}
		\begin{subfigure}{0.32\textwidth}
			\includegraphics[width=\linewidth]{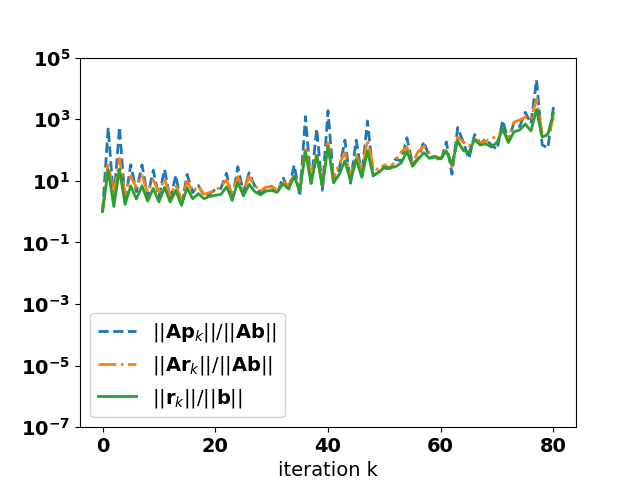}
		\end{subfigure}
		\begin{subfigure}{0.32\textwidth}
			\includegraphics[width=\linewidth]{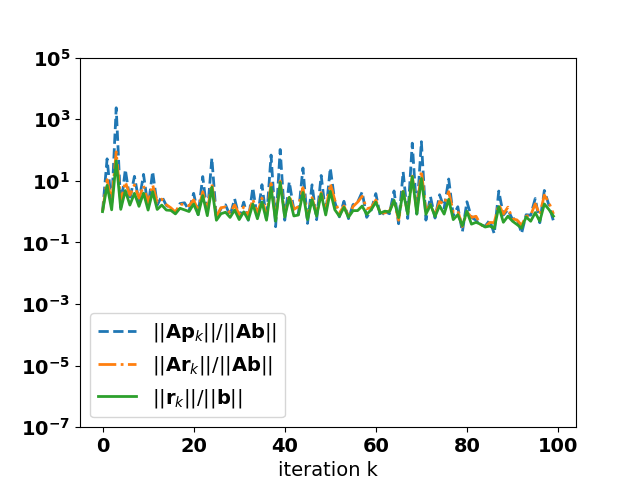}
		\end{subfigure}
		\begin{subfigure}{0.32\textwidth}
			\includegraphics[width=\linewidth]{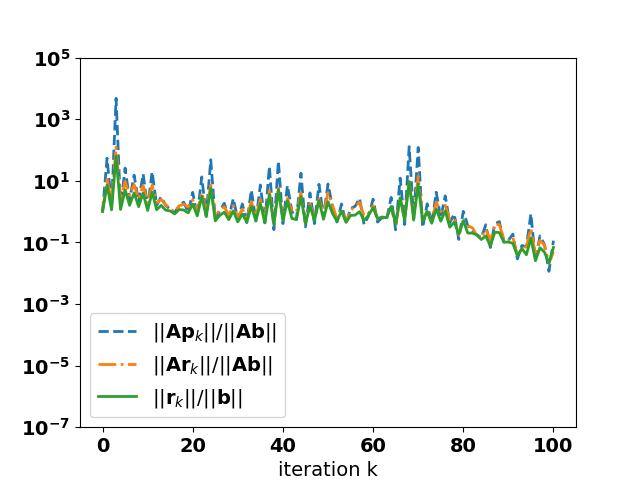}
		\end{subfigure}
		\caption{Experiments for \cref{sec:instability_cg} to explore instabilities within CG. The matrices are indefinite with $d = 100$.  From top left to bottom right, the grades are 11, 21, 41, 81, 100 and 100. The bottom right is an indefinite matrix with full rank for which CG reaches the maximum allowable iteration.}
		\label{fig:cgexplode_indefinite}
	\end{figure}
	
	\begin{figure}[htbp]
		\centering
		\begin{subfigure}{0.32\textwidth}
			\includegraphics[width=\linewidth]{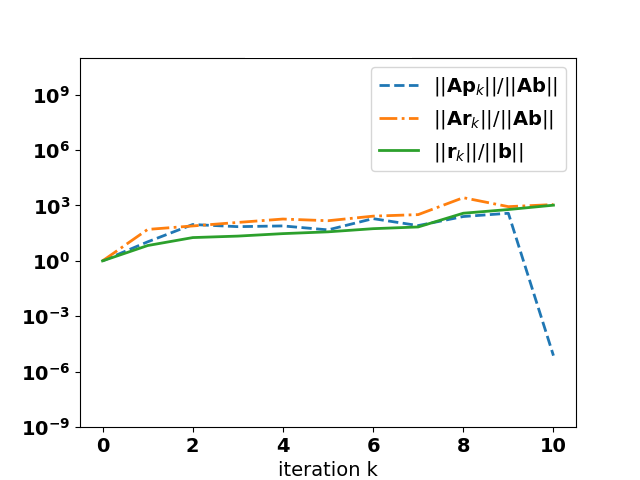}
		\end{subfigure}
		\begin{subfigure}{0.32\textwidth}
			\includegraphics[width=\linewidth]{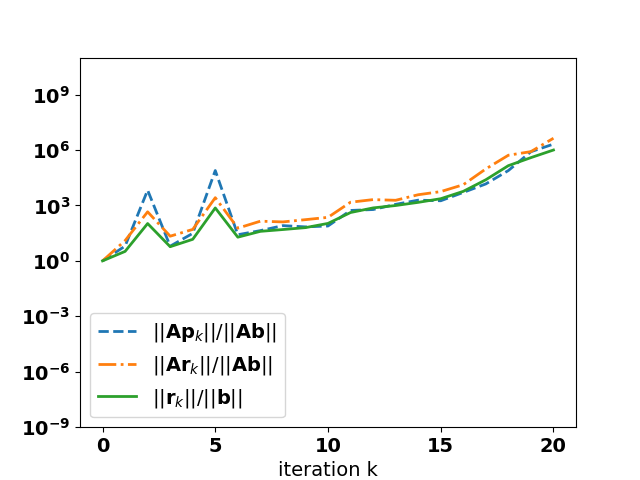}
		\end{subfigure}
		\begin{subfigure}{0.32\textwidth}
			\includegraphics[width=\linewidth]{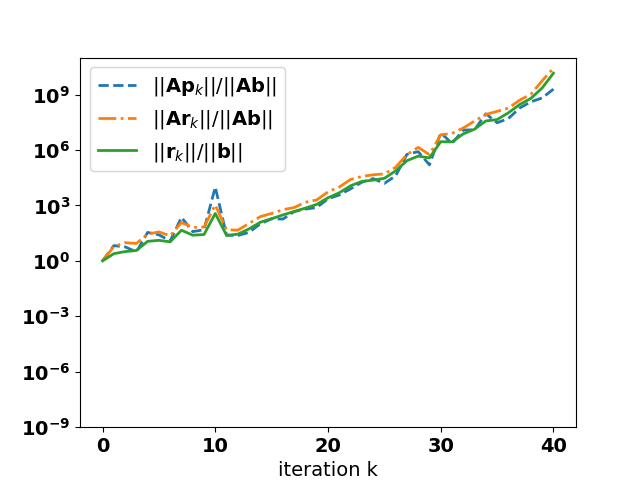}
		\end{subfigure}
		\begin{subfigure}{0.32\textwidth}
			\includegraphics[width=\linewidth]{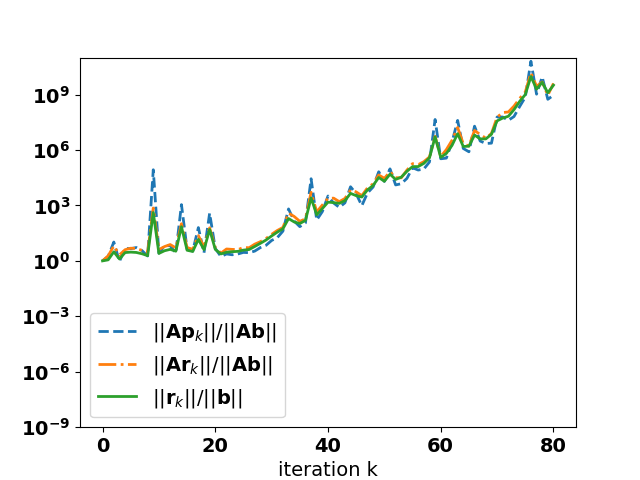}
		\end{subfigure}
		\begin{subfigure}{0.32\textwidth}
			\includegraphics[width=\linewidth]{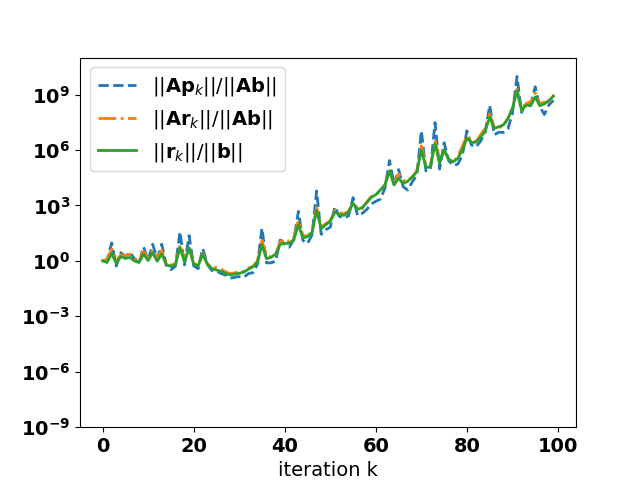}
		\end{subfigure}
		\begin{subfigure}{0.32\textwidth}
			\includegraphics[width=\linewidth]{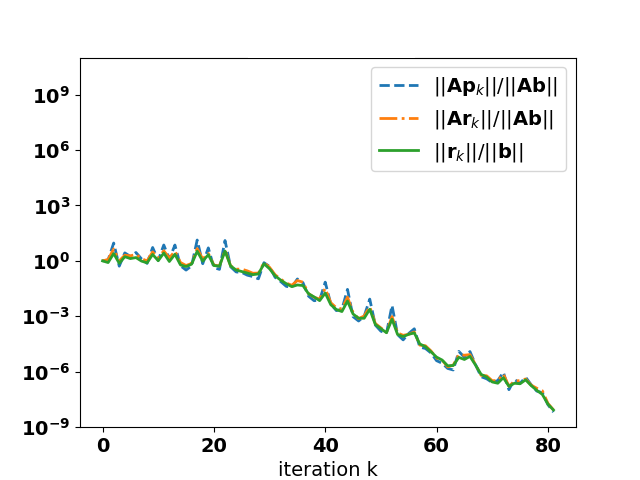}
		\end{subfigure}
		\caption{Experiments for \cref{sec:instability_cg} to explore instabilities within CG. The matrices are slightly indefinite with $d = 100$.  From the top left to the bottom right, the number of negative diagonal entries is $1, 2, 4, 8, 10$ and $10$, while the grades are as in \cref{fig:cgexplode_indefinite}.  The bottom right is a slightly indefinite matrix with full rank for which CG is terminated after the termination condition is triggered.}
		\label{fig:cgexplode_10p}
	\end{figure}
	
	\subsubsection{Numerical (In)Stability of CR/MINRES}\label{sec:comparison}
	In light of the observations in \cref{sec:instability_cg} concerning the apparent numerical instabilities of CG, our focus now shifts to examining the stability, or lack thereof, of CR and MINRES. In the absence of appropriate preconditioning or stabilization strategies like reorthogonalization, iterative procedures, such as CG, CR, or MINRES, inevitably encounter numerical challenges when applied to large and ill-conditioned systems. Nevertheless, in this section, we aim to explore and identify which of these methods, if any, demonstrate a visibly more adverse and unstable behavior.
	
	For the experiments in this section, we set $d = 1,000$ and $m = 800$. The maximum number of iterations is capped at $800$, and the termination condition is triggered at the threshold $10^{-8}$.
	To assess the extent of numerical inaccuracies, we monitor the fundamental theoretical properties of these methods, specifically orthogonality and conjugacy of the relevant vectors, as per \cref{prop:properties}-\labelcref{property:cg} and \cref{prop:properties}-\labelcref{property:cr}. For CG, the deviation from the properties outlined in \cref{prop:properties}-\labelcref{property:cg} is calculated at each iteration. This involves gathering the update directions $\f{p}_{k-1}$ and the normalized residual vectors $\f{r}_{k-1}/\vnorm{\f{r}_{k-1}}$ as columns of matrices $ \f{P}_k $ and $ \f{R}_k $, respectively. The spectral norm of the matrices $\f{I}_k - \f{R}_k^{\T}\f{R}_k$ and $\text{diag}(\{\lr{\f{p}_i,\f{Ap}_i}\}_{i=0}^{i=k}) - \f{P}_k^{\T}\f{AP}_k$ is then calculated, where ``diag'' refers to a diagonal matrix constructed from the given set of values. Following the properties outlined in  \cref{prop:properties}-\labelcref{property:cr},  similar quantities are also calculated for CR and MINRES, namely $\f{I}_k - (\f{AP}_k)^{\T}\f{AP}_k$ (here the vectors $\f{Ap}_{k-1}$ are normalized) and $\text{diag}(\{\lr{\f{r}_i,\f{Ar}_i}\}_{i=0}^{i=k}) - \f{R}_k^{\T}\f{AR}_k$. 
	In the absence of numerical issues, these quantities should all be evaluated to zero.
	
	In addition, we also monitor $\|(\f{b} - \f{Ax}_k) - \f{r}_k\|$ where $ \f{r}_k $ is the residual vector calculated naturally as part of the algorithm, for example, Line 5 in \cref{alg:cg,alg:cr} (a similar quantity also arise as part of MINRES; see \cite{liu2022minres}). Theoretically, the residual calculated within the algorithm must align with that calculated from the iterate as $\f{b} - \f{Ax}_k$. In practice, however, since they are obtained through different means, we expect some discrepancies between them due to numerical errors. Nevertheless, methods that are inherently more stable should exhibit less severe discrepancies.
	
	The results of the experiments for this section are gathered \cref{fig:comparison_psd,fig:comparison_indefinite,fig:comparison_10}. In stark contrast to CR and MINRES, CG exhibits a significant degree of numerical instability in \cref{fig:comparison_psd}, reaching a level that can be categorized as catastrophic. These observations align well with the findings in \cite{kaasschieter1988preconditioned} who suggests that CG under inconsistent systems could lead to divergence, though it does not provide any more insights. Furthermore, in \cref{fig:comparison_indefinite,fig:comparison_10}, although all methods experience some loss of orthogonality and conjugacy in their underlying vectors, CR and MINRES demonstrate slightly more stable numerical behavior. This is evident in the quantity $\|(\f{b} - \f{Ax}_k) - \f{r}_k\|$, and they make relatively satisfactory progress toward obtaining a normal solution with a relatively small residual from the normal equations. 
	
	\begin{figure}[htbp]
		\centering
		\begin{subfigure}{0.32\textwidth}
			\includegraphics[width=\linewidth]{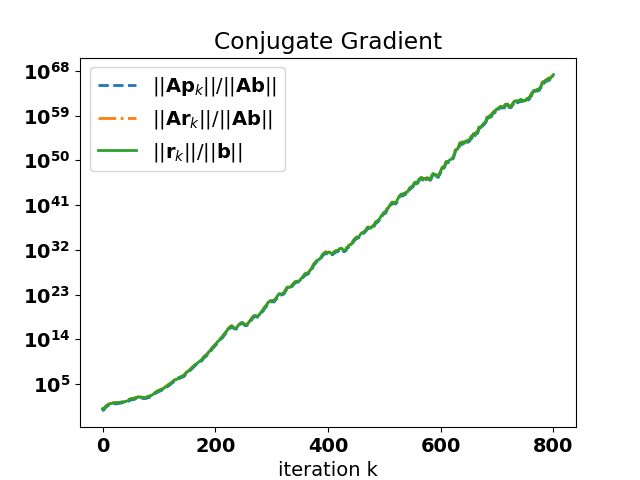}
		\end{subfigure}
		\begin{subfigure}{0.32\textwidth}
			\includegraphics[width=\linewidth]{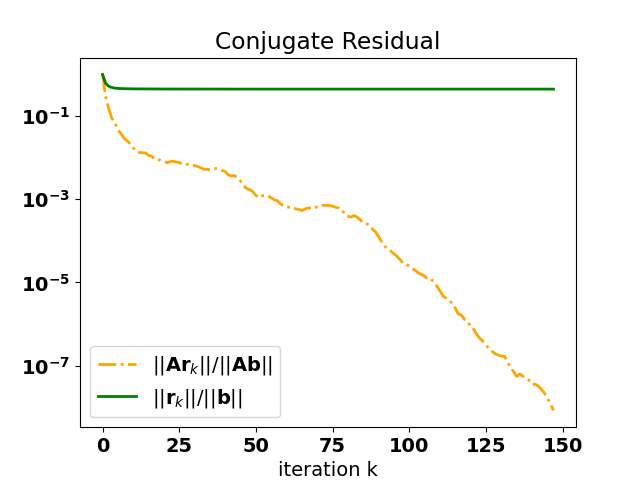}
		\end{subfigure}
		\begin{subfigure}{0.32\textwidth}
			\includegraphics[width=\linewidth]{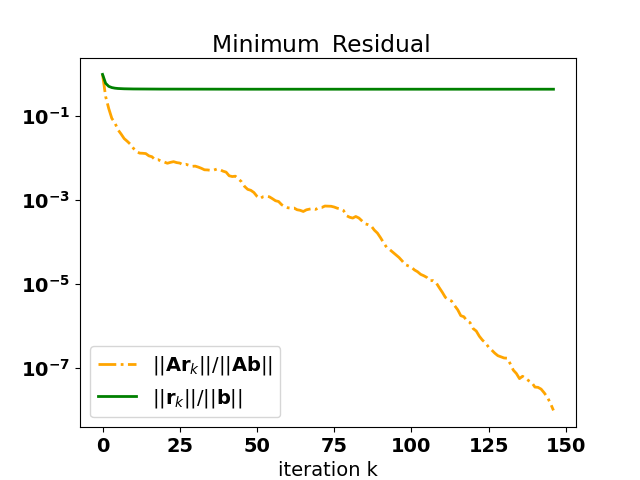}
		\end{subfigure}
		\begin{subfigure}{0.32\textwidth}
			\includegraphics[width=\linewidth]{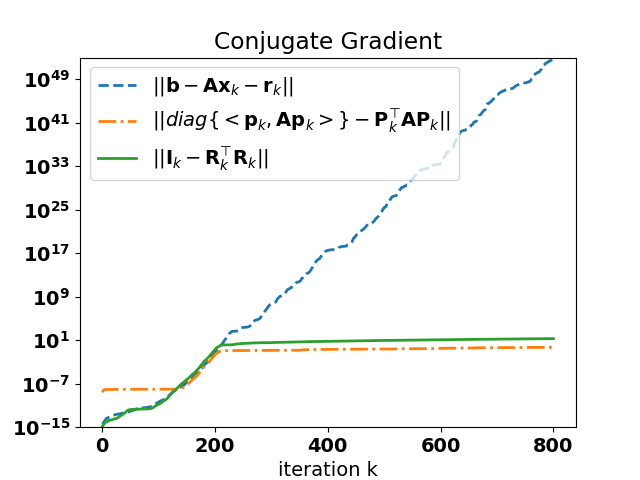}
			\caption*{CG}
		\end{subfigure}
		\begin{subfigure}{0.32\textwidth}
			\includegraphics[width=\linewidth]{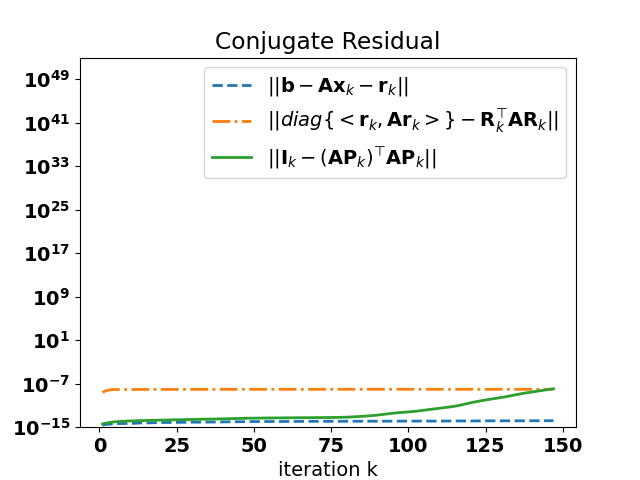}
			\caption*{CR}
		\end{subfigure}
		\begin{subfigure}{0.32\textwidth}
			\includegraphics[width=\linewidth]{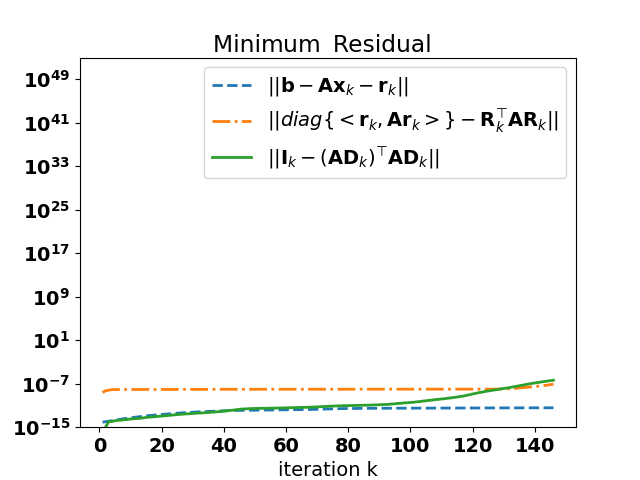}
			\caption*{MINRES}
		\end{subfigure}
		\caption{Experiments for \cref{sec:comparison}  to compare numerical (in)stability among CG, CR, and MINRES. The matrix is PSD with $d = 1000$ and $m = 800$. Surprisingly, the relative residual and $\|(\f{b} - \f{Ax}_k) - \f{r}_k\|$ for CG reach catastrophically large levels. In contrast, not only do both CR and MINRES perform almost identically as predicted, but they also exhibit significantly more stability compared to CG.}
		\label{fig:comparison_psd}
	\end{figure}
	
	\begin{figure}[htbp]
		\centering
		\begin{subfigure}{0.32\textwidth}
			\includegraphics[width=\linewidth]{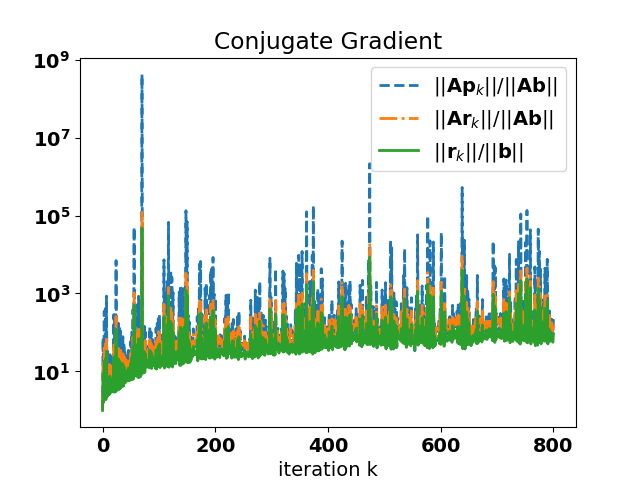}
		\end{subfigure}
		\begin{subfigure}{0.32\textwidth}
			\includegraphics[width=\linewidth]{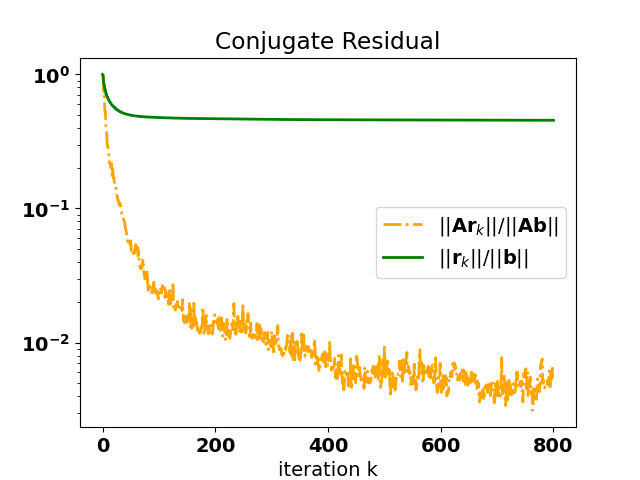}
		\end{subfigure}
		\begin{subfigure}{0.32\textwidth}
			\includegraphics[width=\linewidth]{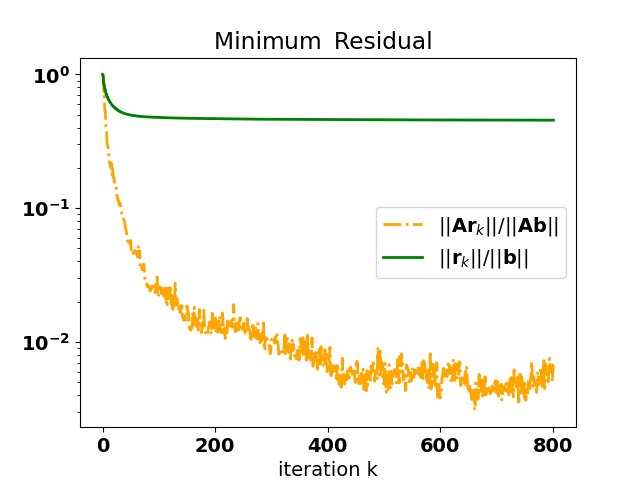}
		\end{subfigure}
		\begin{subfigure}{0.32\textwidth}
			\includegraphics[width=\linewidth]{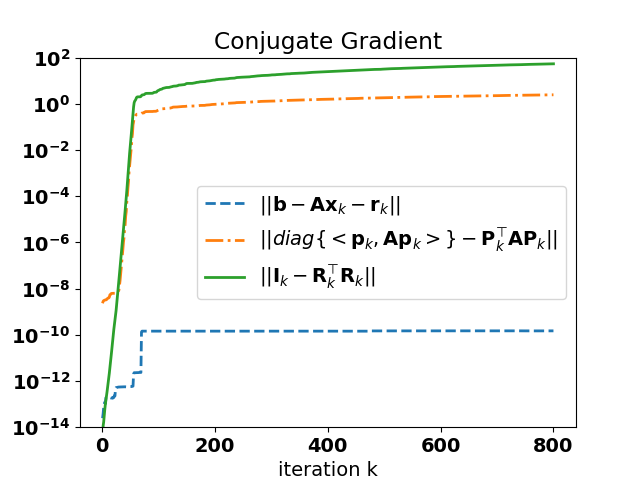}
			\caption*{CG}
		\end{subfigure}
		\begin{subfigure}{0.32\textwidth}
			\includegraphics[width=\linewidth]{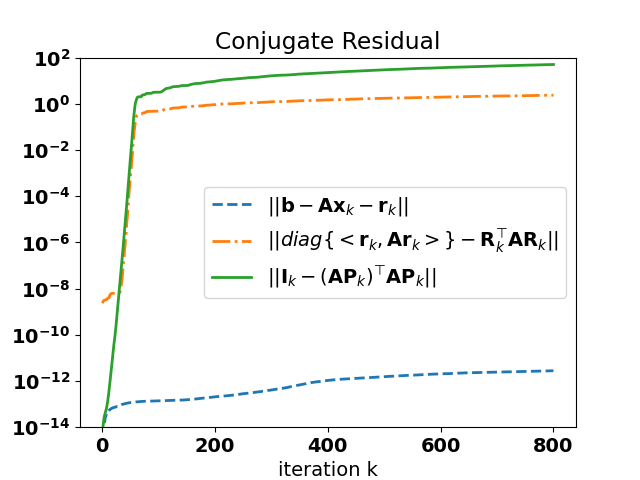}
			\caption*{CR}
		\end{subfigure}
		\begin{subfigure}{0.32\textwidth}
			\includegraphics[width=\linewidth]{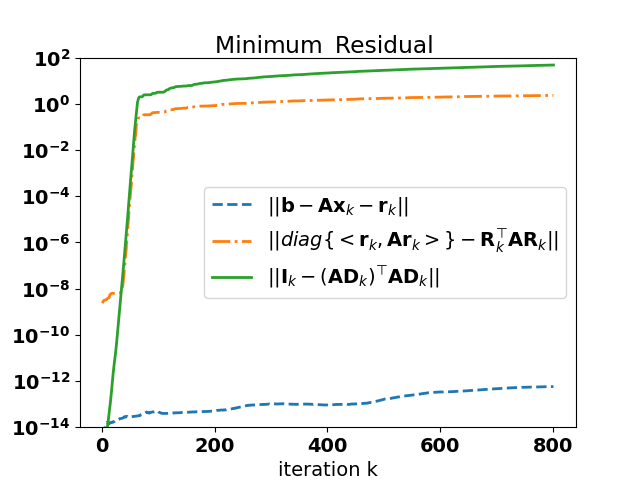}
			\caption*{MINRES}
		\end{subfigure}
		\caption{Experiments for \cref{sec:comparison}  to compare numerical (in)stability among CG, CR, and MINRES. The matrix is indefinite with $d = 1000$ and $m = 800$. While all methods exhibit some loss of orthogonality and conjugacy in the relevant vectors, CR and MINRES show slightly more stability as measured by $\|(\f{b} - \f{Ax}_k) - \f{r}_k\|$. As in \cref{fig:cgexplode_10p,fig:cgexplode_indefinite}, the presence of negative eigenvalues appears to induce a form of oscillatory behavior in CG and the depicted quantities do not exhibit the pronounced growth observed in \cref{fig:comparison_psd}.}
		\label{fig:comparison_indefinite}
	\end{figure}
	
	\begin{figure}[htbp]
		\centering
		\begin{subfigure}{0.32\textwidth}
			\includegraphics[width=\linewidth]{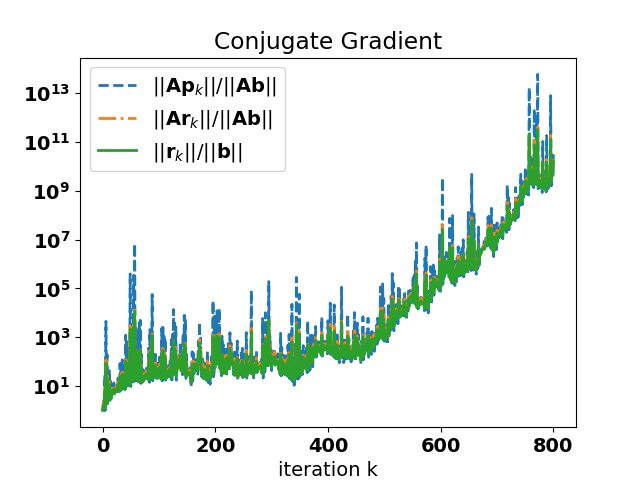}
		\end{subfigure}
		\begin{subfigure}{0.32\textwidth}
			\includegraphics[width=\linewidth]{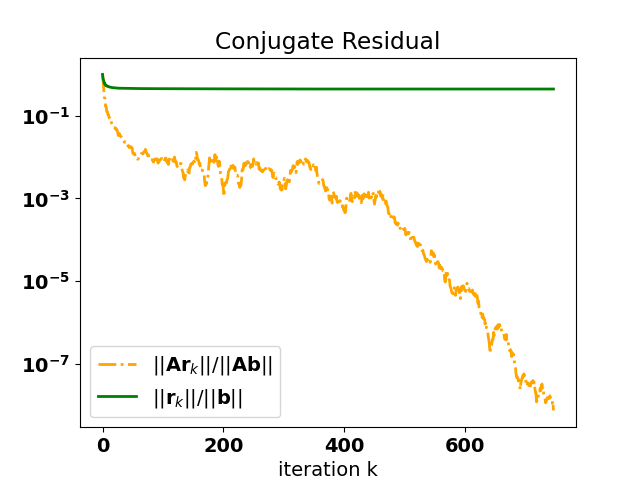}
		\end{subfigure}
		\begin{subfigure}{0.32\textwidth}
			\includegraphics[width=\linewidth]{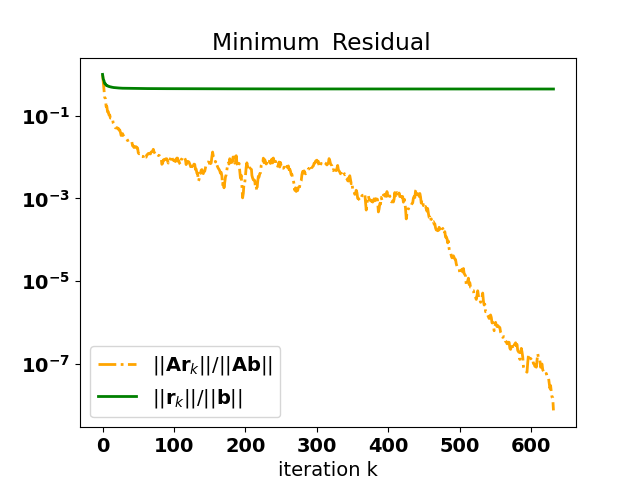}
		\end{subfigure}
		\begin{subfigure}{0.32\textwidth}
			\includegraphics[width=\linewidth]{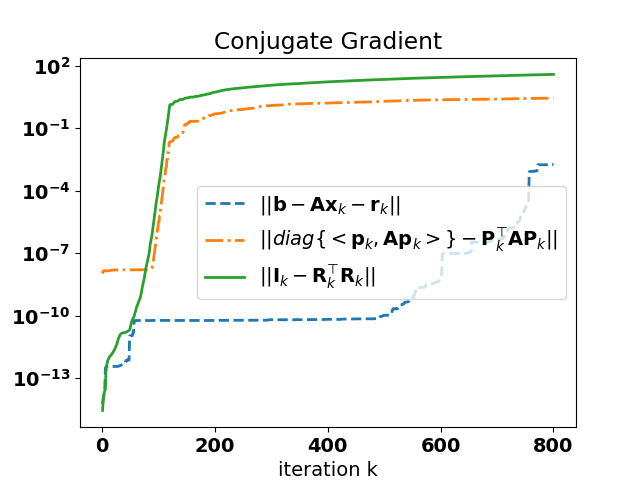}
			\caption*{CG}
		\end{subfigure}
		\begin{subfigure}{0.32\textwidth}
			\includegraphics[width=\linewidth]{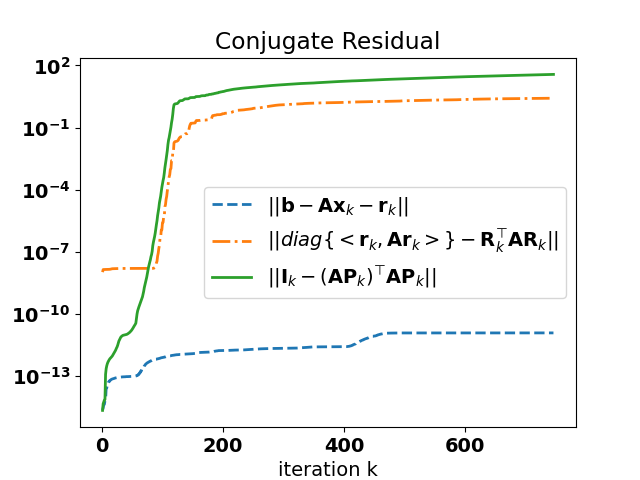}
			\caption*{CR}
		\end{subfigure}
		\begin{subfigure}{0.32\textwidth}
			\includegraphics[width=\linewidth]{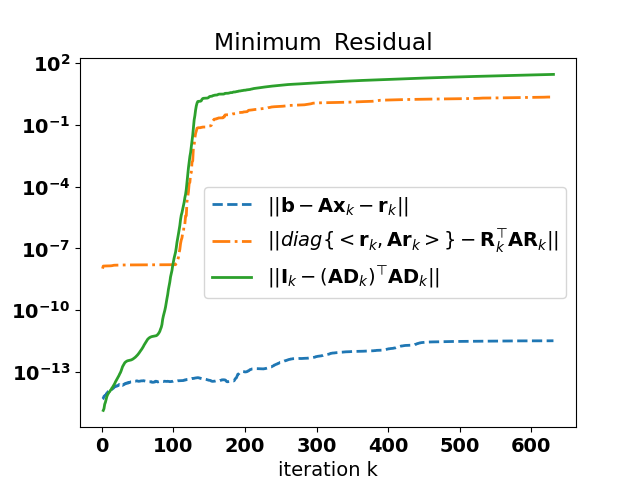}
			\caption*{MINRES}
		\end{subfigure}
		\caption{Experiments for \cref{sec:comparison}  to compare numerical (in)stability among CG, CR, and MINRES. The matrix is slightly indefinite with $d = 1000$, $m = 800$, and only $ 80 $ negative diagonal entries. While all methods exhibit some loss of orthogonality and conjugacy in the relevant vectors, CR and MINRES show slightly more stability as measured by $\|(\f{b} - \f{Ax}_k) - \f{r}_k\|$. In fact, they both converge to a solution with a relatively small residual of the normal equation. Again, the presence of negative eigenvalues appears to induce a form of oscillatory behavior in CG. Yet, compared with \cref{fig:comparison_indefinite} where the matrix contained more negative eigenvalues, the depicted quantities for CG exhibit a more pronounced growth.}
		\label{fig:comparison_10}
	\end{figure}
	
	\subsection{Real-world Applications}\label{sec:real-life}
	In this section, we now evaluate the behavior of CG, CR, and MINRES in some real-world applications. Namely, we consider problems from the numerical solution of partial differential equations (\cref{sec:possion}), computer vision (\cref{sec:image}), and statistical modeling (\cref{sec:kernel}). 
	
	\subsubsection{Poisson Equation with Neumann Boundary Condition}\label{sec:possion}
	Consider numerically solving a 2D Poisson equation with pure Neumann boundary conditions,
	\begin{align*}
		-\Delta u & = f(x, y) \quad \text{where} \quad (x, y) \in \Omega\\
		\lr{\nabla u, \f{n}} & = g(x,y) \quad \text{where} \quad (x, y) \in \partial\Omega,
	\end{align*}
	where $u : \Omega \subseteq \mathbb{R}^2 \to \mathbb{R}$ is the unknown function. Here, $\Delta u$, $\nabla u$ and $\f{n}$ are the Laplace operator of $u$, the gradient of $u$, and the normal vector to the boundary $\partial\Omega$ (the exterior surface), respectively. We use a five-point stencil on uniform staggered grids to formulate the finite difference of the Laplace operator of $u$. The resulting Laplacian matrix of this particular problem is known to be PSD and rank deficient \cite{yoon2016solving,roosta2015randomized}.
	
	In our experiments, we generate $f$ and $g$ by letting $u(x,y) = \sin(\sqrt{x^2 + y^2})$ as the true, and considered unknown, underlying function; see \cref{fig:pde_original}. The domain is centered at $(0.001,0.001)$ with size $(x,y) \in \Omega \triangleq [-10 + 0.001, 10 + 0.001] \times [-10 + 0.001, 10 + 0.001]$. The slight off-center positioning is attributed to the undefined derivatives of $u$ at $(0,0)$. We partition the domain into grids of size $512 \times 512$, resulting in a sparse Laplace matrix with dimensions $(512 + 1)^2 \times (512 + 1)^2$. The maximum number of iterations and the termination condition are set to $2,000$ and $\|\f{Ar}_k\| / \|\f{Ab}\| \leq 10^{-10}$, respectively. 
	
	The results are depicted in \cref{fig:pde_cgcrmr,fig:pde_cg_best}. In \cref{fig:pde_cgcrmr}, unlike CR and MINRES, CG fails to attain the threshold required for the termination condition, prior to reaching the maximum number of iterations. Not only, do both CR and MINRES perform identically, but also they obtain a solution with significantly superior quality compared to that obtained from CG upon termination. \cref{fig:pde_cg_best} illustrates the quality of the recovered solutions from all methods at iteration 359, the point where CG achieves its lowest residual. While the residuals for CR and MINRES are lower than those for CG at this stage, the recovered solution for them exhibits a larger error.

	\begin{figure}[htbp]
		\centering
		\begin{subfigure}{0.4\textwidth}
			\includegraphics[width=\linewidth]{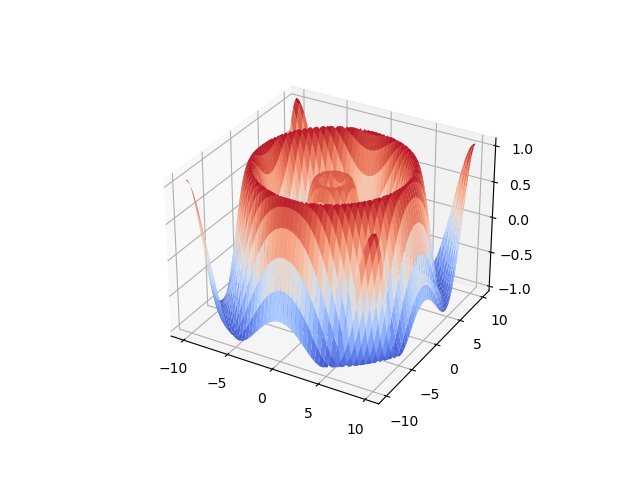}
		\end{subfigure}
		\begin{subfigure}{0.4\textwidth}
			\includegraphics[width=\linewidth]{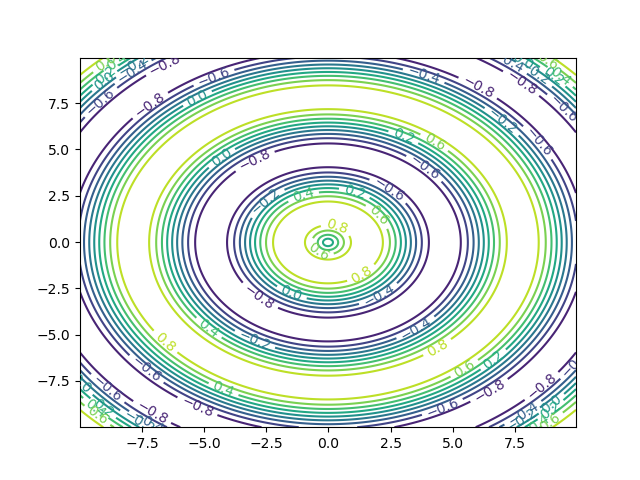}
		\end{subfigure}
		\caption{The true, yet unknown, function $u(x,y) = \sin(\sqrt{x^2 + y^2})$ for generating the experiments of \cref{sec:possion}.}
		\label{fig:pde_original}
	\end{figure}
	
	\begin{figure}[htbp]
		\centering
		\begin{subfigure}{0.32\textwidth}
			\includegraphics[width=\linewidth]{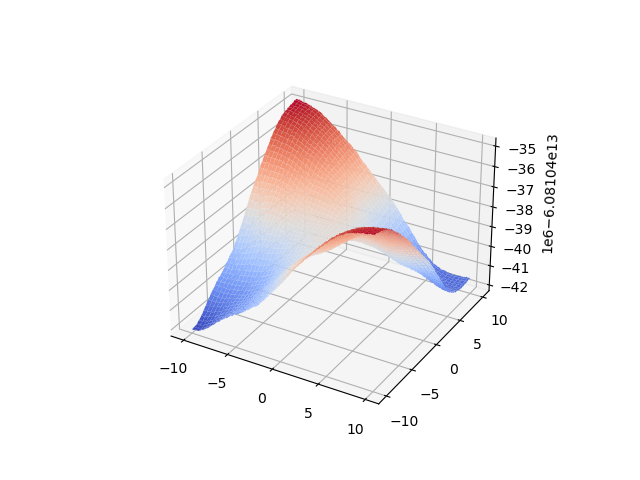}
		\end{subfigure}
		\begin{subfigure}{0.32\textwidth}
			\includegraphics[width=\linewidth]{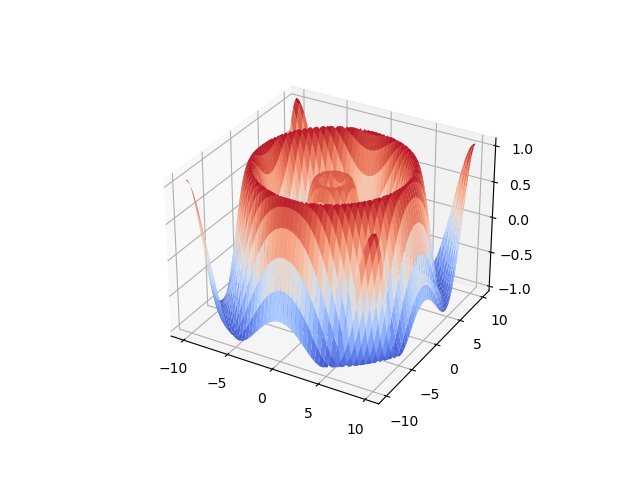}
		\end{subfigure}
		\begin{subfigure}{0.32\textwidth}
			\includegraphics[width=\linewidth]{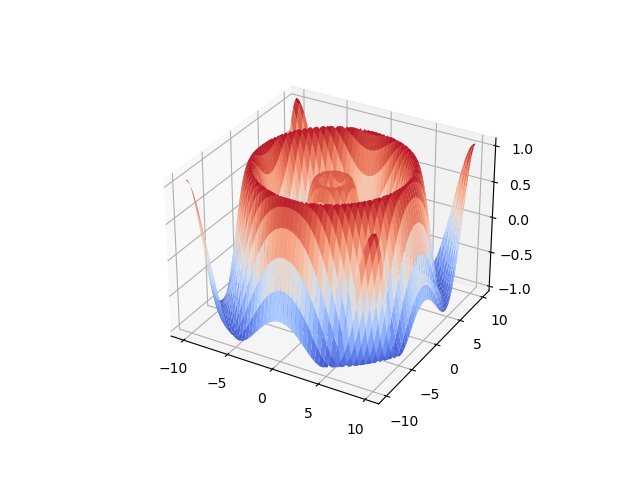}
		\end{subfigure}
		\begin{subfigure}{0.32\textwidth}
			\includegraphics[width=\linewidth]{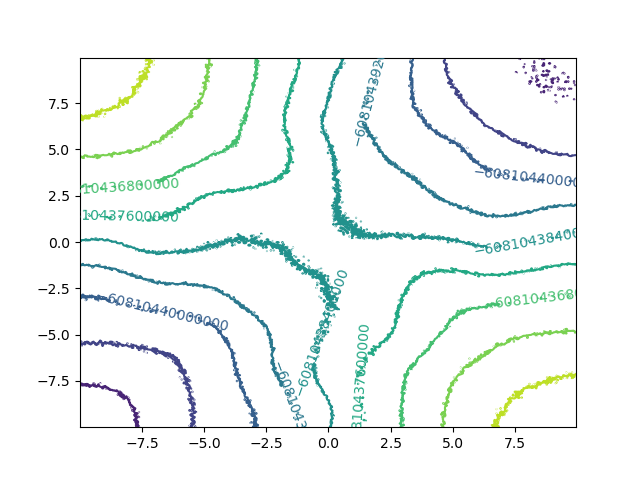}
		\end{subfigure}
		\begin{subfigure}{0.32\textwidth}
			\includegraphics[width=\linewidth]{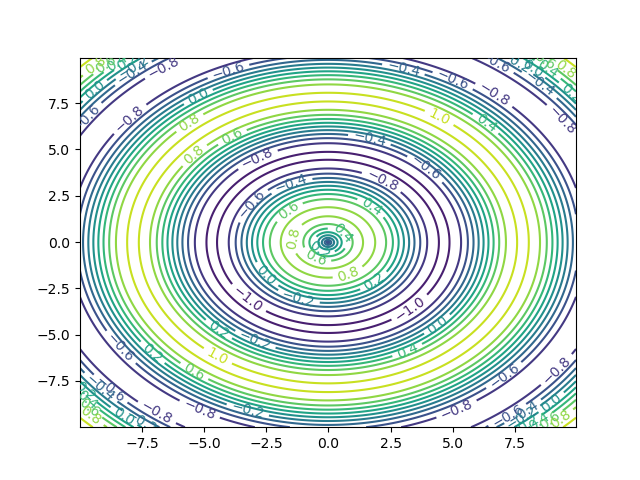}
		\end{subfigure}
		\begin{subfigure}{0.32\textwidth}
			\includegraphics[width=\linewidth]{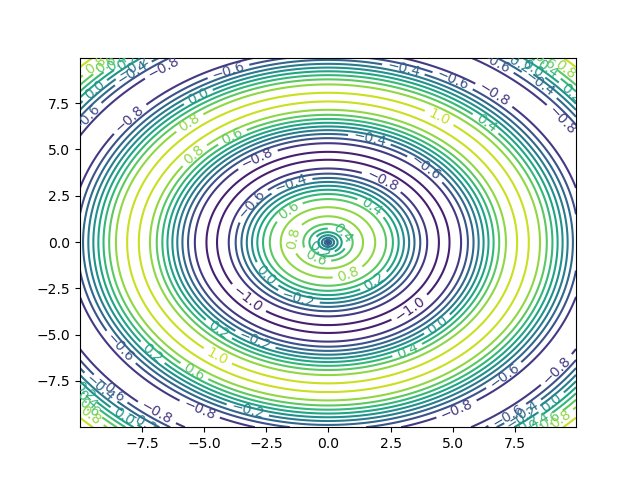}
		\end{subfigure}
		\begin{subfigure}{0.32\textwidth}
			\includegraphics[width=\linewidth]{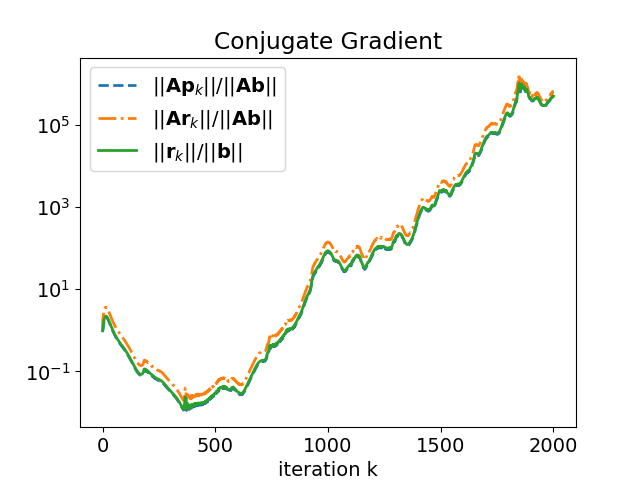}
			\caption*{CG}
		\end{subfigure}
		\begin{subfigure}{0.32\textwidth}
			\includegraphics[width=\linewidth]{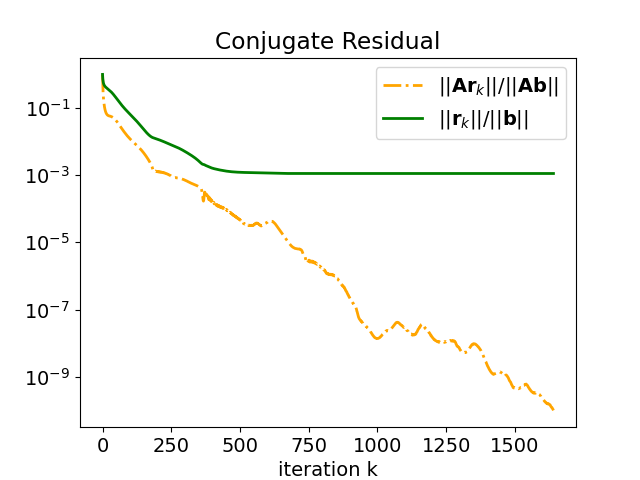}
			\caption*{CR}
		\end{subfigure}
		\begin{subfigure}{0.32\textwidth}
			\includegraphics[width=\linewidth]{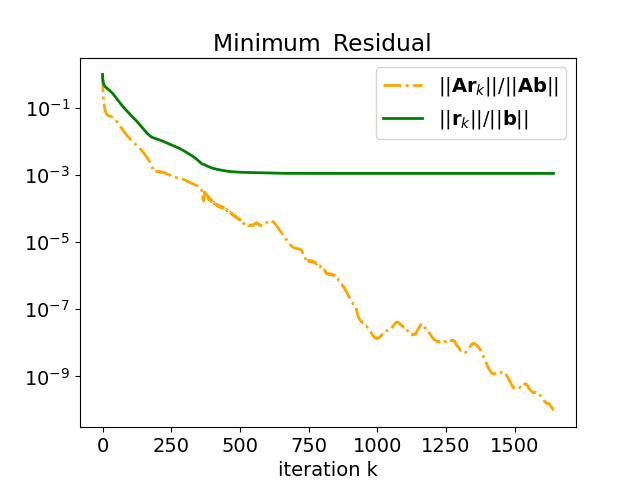}
			\caption*{MINRES}
		\end{subfigure}
		\caption{Experiments for \cref{sec:possion}. Both CR and MINRES behave identically and terminate at iteration 1640 with $\|\f{r}_k\|/\|\f{b}\| \approx 1.12 \times 10^{-3}$, $\|\f{Ar}_k\| / \|\f{Ab}\| \approx 9.91 \times 10^{-11}$ and $\|\f{x}_{\text{true}} - \f{x}_k\| / \|\f{x}_{\text{true}}\| \approx 0.0841$. CG never attains the termination condition threshold prior to reaching the maximum number of iterations. At termination, CG obtains $\|\f{r}_k\| / \|\f{b}\| \approx 5.06 \times 10^{5}$, $\|\f{Ar}_k\| / \|\f{Ab}\| \approx 6.43 \times 10^{5}$, $\|\f{Ap}_k\| / \|\f{Ab}\| \approx 4.9 \times 10^{5}$ and $\|\f{x}_{\text{true}} - \f{x}_k\| / \|\f{x}_{\text{true}}\| \approx 8.7\times 10^{13}$.}
		\label{fig:pde_cgcrmr}
	\end{figure}
	
	\begin{figure}[htbp]
		\centering
		\begin{subfigure}{0.32\textwidth}
			\includegraphics[width=\linewidth]{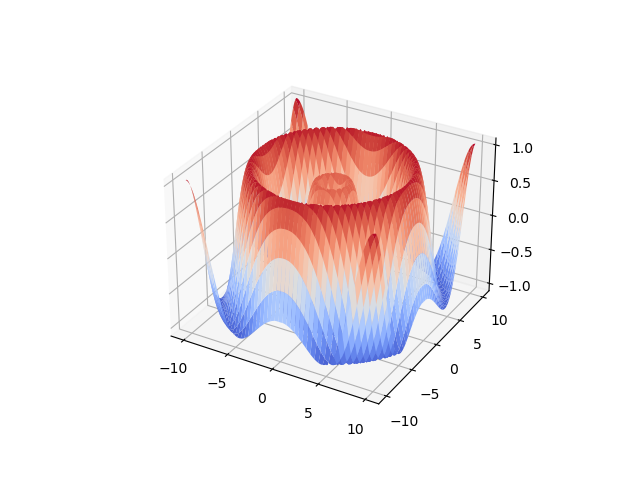}
		\end{subfigure}
		\begin{subfigure}{0.32\textwidth}
			\includegraphics[width=\linewidth]{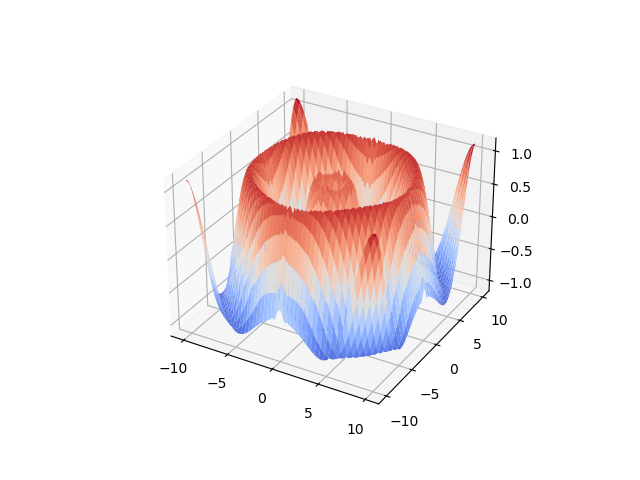}
		\end{subfigure}
		\begin{subfigure}{0.32\textwidth}
			\includegraphics[width=\linewidth]{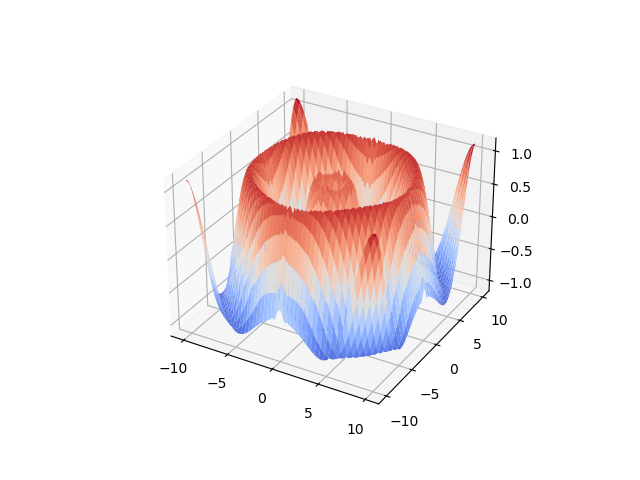}
		\end{subfigure}
		\begin{subfigure}{0.32\textwidth}
			\includegraphics[width=\linewidth]{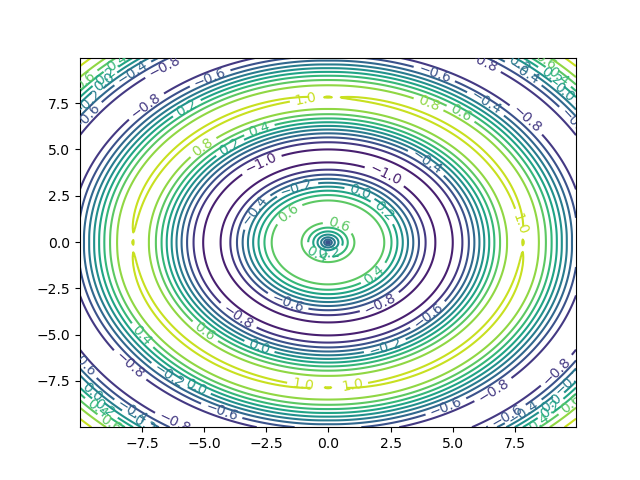}
		\end{subfigure}
		\begin{subfigure}{0.32\textwidth}
			\includegraphics[width=\linewidth]{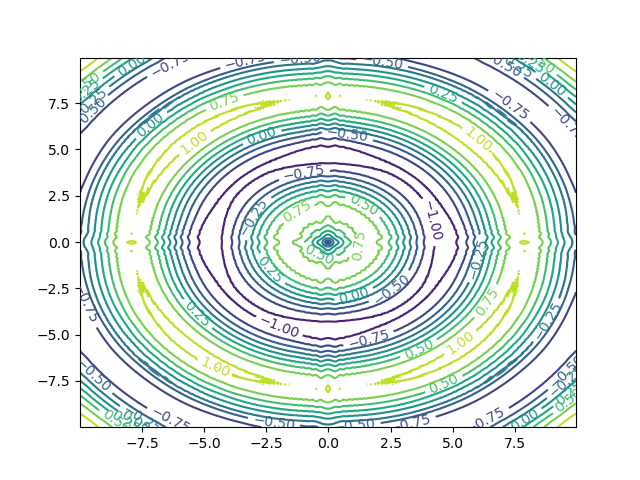}
		\end{subfigure}
		\begin{subfigure}{0.32\textwidth}
			\includegraphics[width=\linewidth]{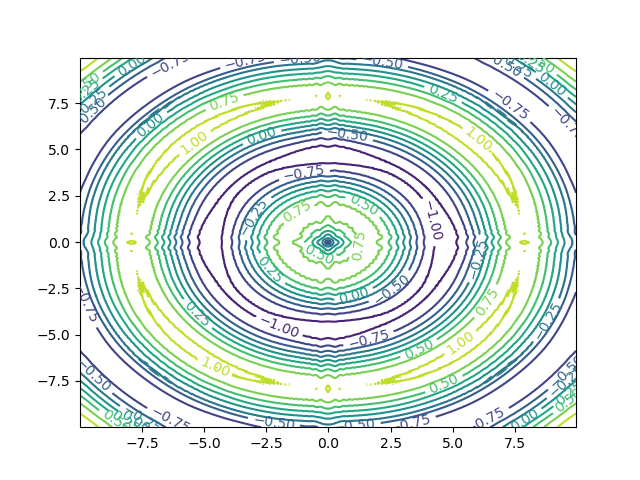}
		\end{subfigure}
		\caption{Experiments for \cref{sec:possion}. CG attains its minimum relative residual at iteration 359 and the recovered solutions for all three methods at this specific iteration are depicted. At this iteration, CG obtains $\|\f{r}_k\| / \|\f{b}\| \approx 1.22 \times 10^{-2}$, $\|\f{Ar}_k\| / \|\f{Ab}\| \approx 1.98 \times 10^{-2}$, $\|\f{Ap}_k\| / \|\f{Ab}\| \approx 1.15 \times 10^{-2}$ and $\|\f{x}_{\text{true}} - \f{x}_k\| / \|\f{x}_{\text{true}}\| \approx 0.0992$, while for CR and MINRES we observe $\|\f{r}_k\| / \|\f{b}\| \approx 2.25 \times 10^{-3}$, $\|\f{Ar}_k\| / \|\f{Ab}\| \approx 3.92 \times 10^{-4}$ and $\|\f{x}_{\text{true}} - \f{x}_k\| / \|\f{x}_{\text{true}}\| \approx 0.1212$.}
		\label{fig:pde_cg_best}
	\end{figure}
	
	\subsubsection{Gaussian Smoothing}\label{sec:image}
	Next, we consider the behavior of CG, CR, and MINRES in the context of an image deblurring problem \cite{linda2001shapiro}. For a given image $\f{x}$, we apply a linear Gaussian smoothing / blurring operator $\f{A}$, as a 2D convolutional operator, on the image and add some noise $\f{e}$ sampled from a uniform distribution, i.e., $\f{b} = \f{A}\f{x} + \f{e}$. We follow a similar experimental setup as the numerical experiment in the work\cite{liu2023obtaining}. The true image and blurry version are given in \cref{fig:gauss_original}. We remind that though $ \f{A} $ typically has full rank, it often contains several small singular values, which make it almost numerically singular. We do not use any particular regularization and simply apply CG, CR and MINRES to attempt to recover the original image. The maximum number of iterations and the termination condition are set to $300$ and $\|\f{Ar}_k\| / \|\f{Ab}\| \leq 10^{-5}$, respectively. 
	
	The results are presented in \cref{fig:gauss_stat,fig:gauss_cg_best}. We stress that even when the matrix $\f{A}$ is PD with a few small eigenvalues, it can be nearly numerically singular. Consequently, CG can exhibit behavior similar to that observed in inconsistent systems and may fail to achieve the required termination threshold before reaching the maximum number of iterations. In contrast, both CR and MINRES, as evident from \cref{fig:gauss_stat,fig:gauss_cg_best}, demonstrate greater stability and terminate once $\|\f{Ar}_k\|/\|\f{Ab}\| \leq 10^{-5}$. Clearly, the image recovered by CG at termination is of significantly lower quality compared to those obtained by CR and MINRES. Similar to the previous section, \cref{fig:gauss_cg_best} showcases the best performance of CG — the image recovered by CG at the iteration with the smallest relative residual — compared with the images from CR and MINRES at the same iteration. At this particular iteration, while CG produces a slightly improved image, it remains visibly blurry with a larger relative error than those obtained by CR/MINRES at termination.
	\begin{figure}[htbp]
		\centering
		\includegraphics[width=0.4\linewidth]{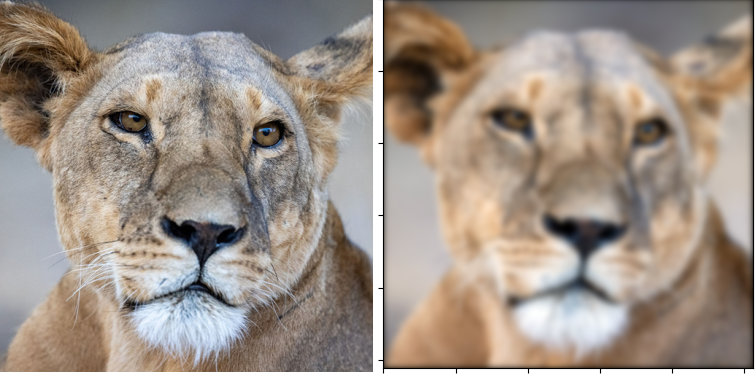}
		\caption{The true and the noisy blurred images used for the experiments of \cref{sec:image}.}
		\label{fig:gauss_original}
	\end{figure}
	\begin{figure}[htbp]
		\centering
		\begin{subfigure}{0.32\textwidth}
			\includegraphics[width=\linewidth]{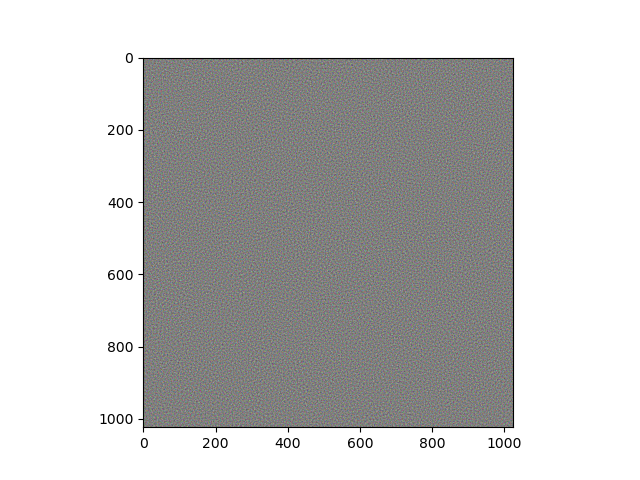}
		\end{subfigure}
		\begin{subfigure}{0.32\textwidth}
			\includegraphics[width=\linewidth]{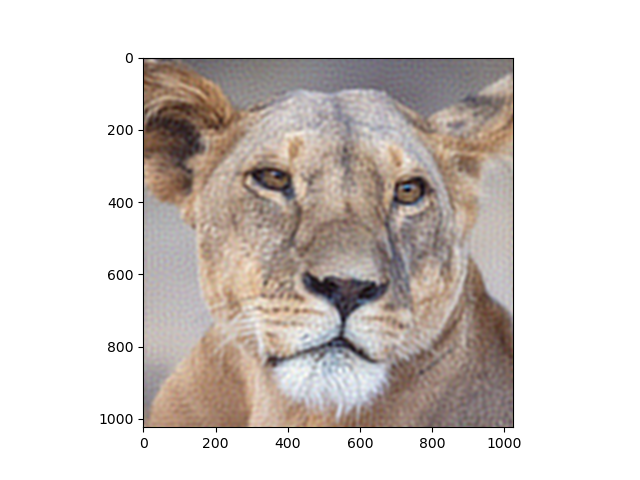}
		\end{subfigure}
		\begin{subfigure}{0.32\textwidth}
			\includegraphics[width=\linewidth]{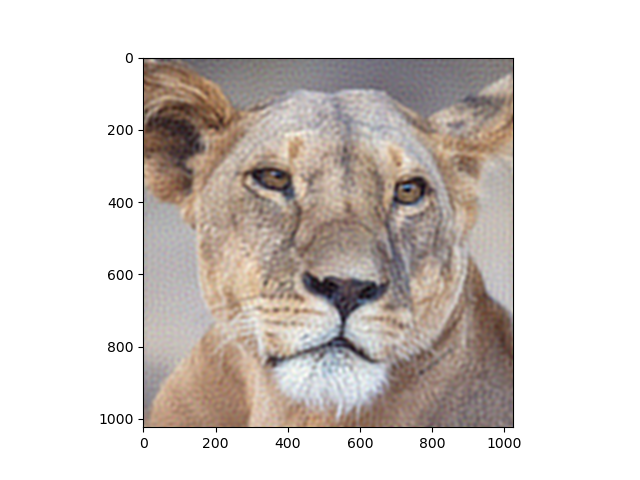}
		\end{subfigure}
		\begin{subfigure}{0.32\textwidth}
			\includegraphics[width=\linewidth]{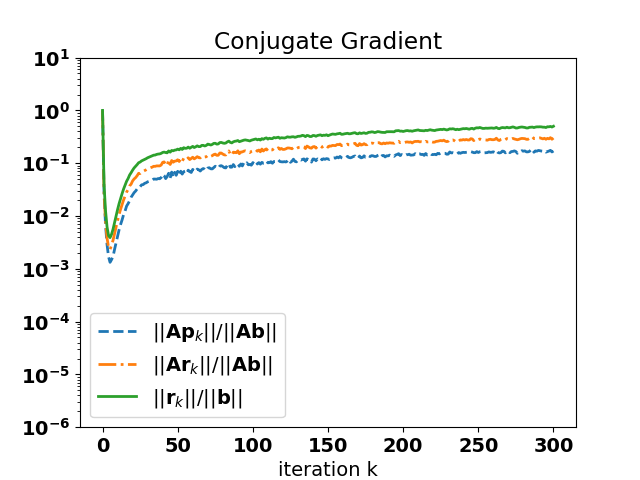}
			\caption*{CG}
		\end{subfigure}
		\begin{subfigure}{0.32\textwidth}
			\includegraphics[width=\linewidth]{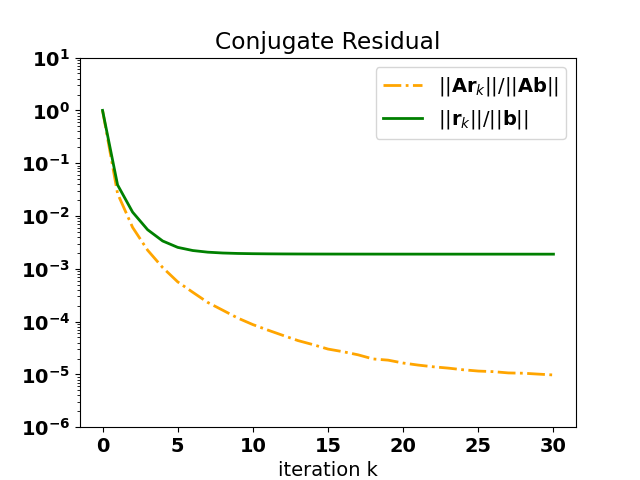}
			\caption*{CR}
		\end{subfigure}
		\begin{subfigure}{0.32\textwidth}
			\includegraphics[width=\linewidth]{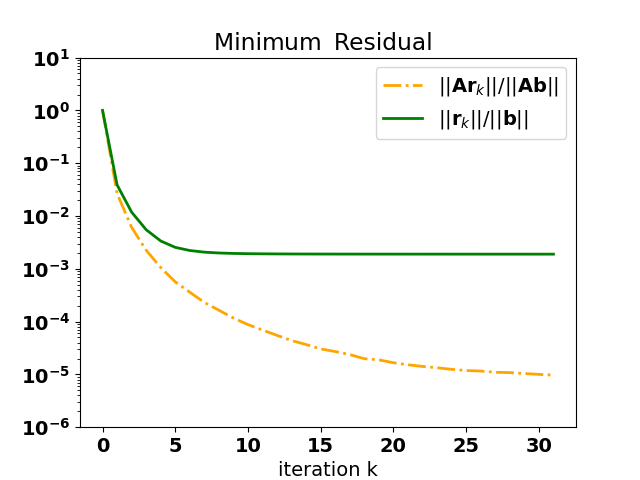}
			\caption*{MR}
		\end{subfigure}
		\caption{Experiments for \cref{sec:image}. The quantities plotted correspond to the red channel of the images. CG terminates at iteration 300 with $\|\f{r}_k\|/\|\f{b}\| \approx 0.511$, $\|\f{Ar}_k\| / \|\f{Ab}\| \approx 0.323$, $\|\f{Ap}_k\| / \|\f{Ab}\| \approx 0.184$ and $\|\f{x}_{\text{true}} - \f{x}_k\| / \|\f{x}_{\text{true}}\| \approx 1.8 \times 10^{5}$. CR and MINRES both satisfy the termination condition at iteration 30 with $\|\f{r}_k\|/\|\f{b}\| \approx 1.89 \times 10^{-3}$, $\|\f{Ar}_k\| / \|\f{Ab}\| \approx 9.76 \times 10^{-6}$ and $\|\f{x}_{\text{true}} - \f{x}_k\| / \|\f{x}_{\text{true}}\| \approx 0.2376$.}
		\label{fig:gauss_stat}
	\end{figure}
	
	\begin{figure}[htbp]
		\centering
		\begin{subfigure}{0.32\textwidth}
			\includegraphics[width=\linewidth]{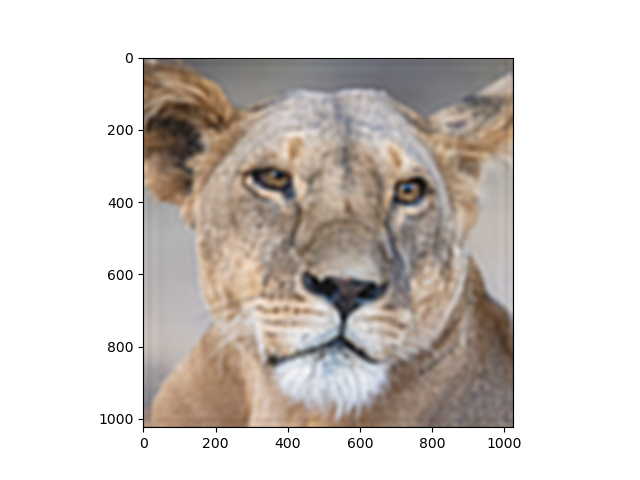}
		\end{subfigure}
		\begin{subfigure}{0.32\textwidth}
			\includegraphics[width=\linewidth]{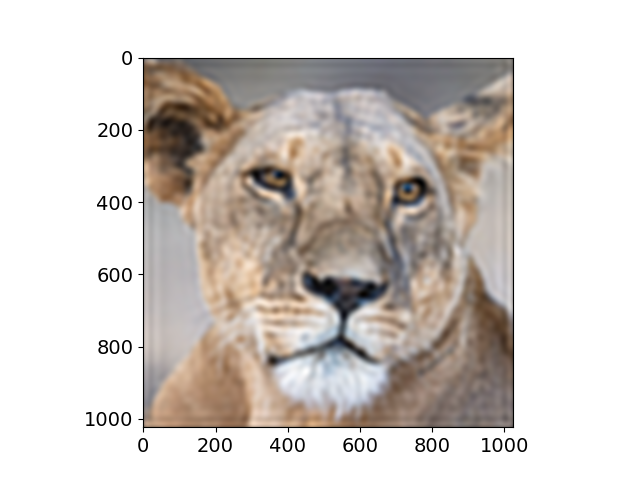}
		\end{subfigure}
		\begin{subfigure}{0.32\textwidth}
			\includegraphics[width=\linewidth]{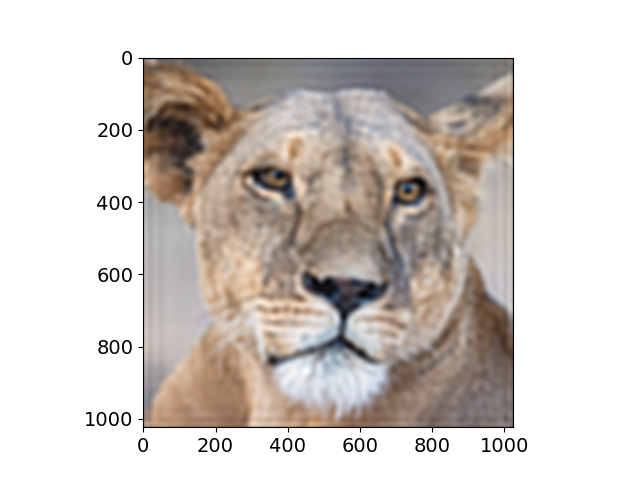}
		\end{subfigure}
		\caption{Experiments for \cref{sec:image}. CG attains its minimum relative residual at iteration 5 and the recovered solutions for all three methods at this specific iteration are depicted. At this iteration, CG obtains $\|\f{r}_k\|/\|\f{b}\| \approx 3.9 \times 10^{-3}$, $\|\f{Ar}_k\| / \|\f{Ab}\| \approx 2.33 \times 10^{-3}$, $\|\f{Ap}_k\| / \|\f{Ab}\| \approx 1.33\times 10^{-3}$ and $\|\f{x}_{\text{true}} - \f{x}_k\| / \|\f{x}_{\text{true}}\| \approx 0.249$, while for CR and MINRES, we observe $\|\f{r}_k\|/\|\f{b}\| \approx 2.55 \times 10^{-3}$, $\|\f{Ar}_k\| / \|\f{Ab}\| \approx 5.66 \times 10^{-4}$ and $\|\f{x}_{\text{true}} - \f{x}_k\| / \|\f{x}_{\text{true}}\| \approx 0.2562$.}
		\label{fig:gauss_cg_best}
	\end{figure}
	
	\subsubsection{Kernel Regression}\label{sec:kernel}
	Finally, we consider kernel regression \cite{hastie2009elements}, which appears frequently in machine learning applications. Kernel regression is an example of non-parametric regression where given training data $ \mathcal{D} = \{\f{a}_{i}, b_{i}\}_{i=1}^{n} \subset \mathbb{R}^{p} \times \mathbb{R} $, we aim to find a prediction function $f: \mathbb{R}^{p} \to \real$ that is a solution to the following minimization problem
	\begin{align*}
		\min_{f \in \mathcal{H}} \left(f(\f{a}_{i})  - b_{i}\right)^{2},
	\end{align*}
	where $ \mathcal{H} $ is some appropriate reproducing kernel Hilbert space (RKHS) with the kernel function $ \kappa:\real^{p} \times \real^{p} \to \real $ \cite{murphy2023probabilistic}. It can be shown that any solution to the above problem can be written as $ f = g + h $ where $ g  = \sum_{j = 1}^{n}x_{j}\kappa(\f{a}_{j},.)$ and $h \perp \Span\{\kappa(\f{a}_{j},.)\}_{j=1}^{n}$, leading to the least-squares regression problem
	\begin{align*}
		\min_{\f{x} \in \mathbb{R}^n} \vnorm{\f{b} - \f{A}\f{x}}^{2},
	\end{align*}
	where $\f{A} \in \mathbb{R}^{n \times n}$ is the positive (semi-)definite kernel matrix such that $ \f{A}_{i,j} = \kappa(\f{a}_{i},\f{a}_{j}) $ and $ \f{b} \in \real^{n}  $ with $ \f{b}_i = b_{i} $. 
	
	We use \texttt{Wine} \texttt{Quality} dataset, comprising $6,497$ training examples with $p = 10$ \cite{cortez2009using}.  The dataset is randomly split into $80\%$ training (thus, $n = 5,197$) and $20\%$ validation sets. Employing the radial basis function kernel $\kappa(\f{a}, \f{a}^{\prime}) = \exp{(-L\|\f{a} - \f{a}^{\prime}\|_2^2)}$ with  $L = 10^{-4}$ results in the positive semi-definite kernel matrix $\f{A}$ of rank $438$. 
	The maximum number of iterations and the termination condition are set to $2,000$ and $\|\f{Ar}_k\| / \|\f{Ab}\| \leq 10^{-7}$, respectively. To assess the predictive performance of the recovered model, we measure the mean squared error (MSE) for the test set at each iteration. The results, depicted in \cref{fig:kernel_stat}, clearly show that both CR and MINRES greatly outperform CG, not only in metrics related to the training set but also in terms of predictive performance as measured by the MSE on the test set.
	
	\begin{figure}[htbp]
		\centering
		\begin{subfigure}{0.32\textwidth}
			\includegraphics[width=\linewidth]{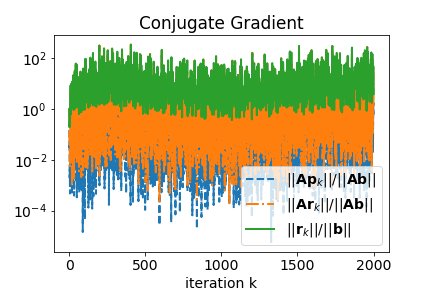}
		\end{subfigure}
		\begin{subfigure}{0.32\textwidth}
			\includegraphics[width=\linewidth]{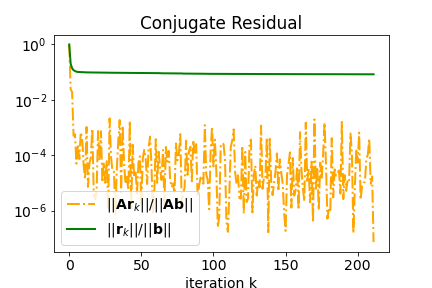}
		\end{subfigure}
		\begin{subfigure}{0.32\textwidth}
			\includegraphics[width=\linewidth]{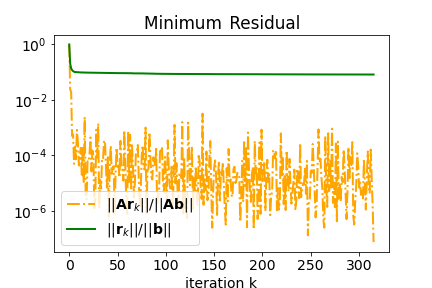}
		\end{subfigure}
		\begin{subfigure}{0.32\textwidth}
			\includegraphics[width=\linewidth]{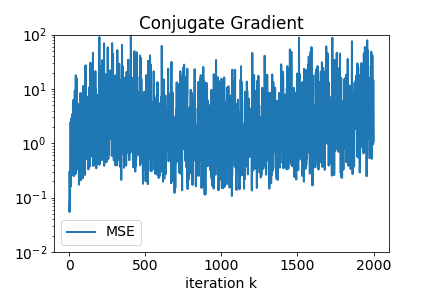}
		\end{subfigure}
		\begin{subfigure}{0.32\textwidth}
			\includegraphics[width=\linewidth]{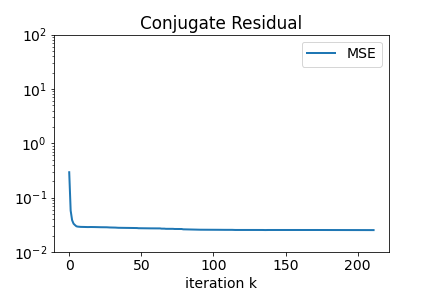}
		\end{subfigure}
		\begin{subfigure}{0.32\textwidth}
			\includegraphics[width=\linewidth]{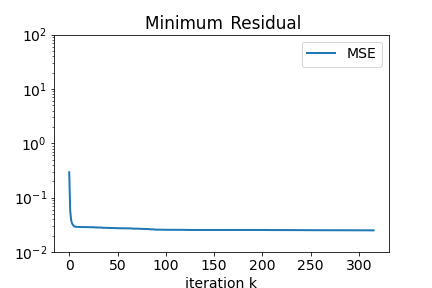}
		\end{subfigure}
		\caption{Experiments for \cref{sec:kernel}. CR and MINRES both attain $\|\f{r}_k\|/\|\f{b}\| \approx 0.0835$, $\|\f{Ar}_k\| / \|\f{Ab}\| \approx 7.27 \times 10^{-8}$ and the validation MSE of 0.025 at 211 and 315 iterations. Whereas, for CG, $\|\f{r}_k\|/\|\f{b}\| \approx 3.81$, $\|\f{Ar}_k\| / \|\f{Ab}\| \approx 0.976$, $\|\f{Ap}_k\| / \|\f{Ab}\| \approx 0.596$ and the validation MSE is 1.1 at iteration 2000.}
		\label{fig:kernel_stat}
	\end{figure}

    \section{Conclusion}
    We theoretically and empirically studied the behavior of CG and CR, as two prominent examples of the conjugate directions methods, for solving symmetric but inconsistent systems. We established a fundamental limitation in classical CG in terms of its inability to recover any solution to the normal equation. To remedy this, we proposed a modification that ensures convergence to the pseudo-inverse solution.
    We then showed that a minor modification to CR at the very last iteration ensures convergence to the pseudo-inverse solution. We further established that CR and MINRES are both essentially equivalent in that they, mathematically speaking, generate the same sequence of iterates. Lastly, we explored the empirical performance of CG, CR, and MINRES in a variety of settings. Contrary to conventional wisdom and in stark contrast to CR/MINRES, we showcased that CG can exhibit numerical instability in a borderline catastrophic manner across scenarios involving inconsistent systems. 

	\bibliographystyle{plain}
	\bibliography{ref}

\begin{thebibliography}{10}

\bibitem{bjorck2015numerical}
{\AA}ke Bj{\"o}rck.
\newblock {\em Numerical methods in matrix computations}.
\newblock Springer, 2015.

\bibitem{calvetti1994conjugate}
D~Calvetti, L~Reichel, and Q~Zhang.
\newblock Conjugate gradient algorithms for symmetric inconsistent linear systems.
\newblock In {\em Proceedings of the Cornelius Lanczos International Centenary Conference, SIAM, Philadelphia}, pages 267--272. SIAM, 1994.

\bibitem{calvetti2001choice}
Daniela Calvetti, Bryan Lewis, and Lothar Reichel.
\newblock On the choice of subspace for iterative methods for linear discrete ill-posed problems.
\newblock {\em International Journal of Applied Mathematics and Computer Science}, 11(5):1069--1092, 2001.

\bibitem{choi2006iterative}
Sou-Cheng Choi.
\newblock Iterative methods for singular linear equations and least-squares problems.
\newblock ICME, Stanford University, 2006.

\bibitem{cortez2009using}
Paulo Cortez, Juliana Teixeira, Ant{\'o}nio Cerdeira, Fernando Almeida, Telmo Matos, and Jos{\'e} Reis.
\newblock Using data mining for wine quality assessment.
\newblock In {\em Discovery Science: 12th International Conference, DS 2009, Porto, Portugal, October 3-5, 2009 12}, pages 66--79. Springer, 2009.

\bibitem{dahito2019conjugate}
Marie-Ange Dahito and Dominique Orban.
\newblock The conjugate residual method in linesearch and trust-region methods.
\newblock {\em SIAM Journal on Optimization}, 29(3):1988--2025, 2019.

\bibitem{estrin2019euclidean}
Ron Estrin, Dominique Orban, and Michael Saunders.
\newblock Euclidean-norm error bounds for symmlq and cg.
\newblock {\em SIAM Journal on Matrix Analysis and Applications}, 40(1):235--253, 2019.

\bibitem{fasano2005planar}
Giovanni Fasano.
\newblock Planar conjugate gradient algorithm for large-scale unconstrained optimization, part 1: theory.
\newblock {\em Journal of Optimization Theory and Applications}, 125(3):523--541, 2005.

\bibitem{fasano2005planar_app}
Giovanni Fasano.
\newblock Planar conjugate gradient algorithm for large-scale unconstrained optimization, part 2: application.
\newblock {\em Journal of Optimization Theory and Applications}, 125(3):543--558, 2005.

\bibitem{fasano2007lanczos}
Giovanni Fasano.
\newblock Lanczos conjugate-gradient method and pseudoinverse computation on indefinite and singular systems.
\newblock {\em Journal of optimization theory and applications}, 132(2):267--285, 2007.

\bibitem{fischer1996note}
Bernd Fischer, Martin Hanke, and Marlis Hochbruck.
\newblock A note on conjugate-gradient type methods for indefinite and/or inconsistent linear systems.
\newblock {\em Numerical Algorithms}, 11:181--187, 1996.

\bibitem{fletcher2006conjugate}
Roger Fletcher.
\newblock Conjugate gradient methods for indefinite systems.
\newblock In {\em Numerical Analysis: Proceedings of the Dundee Conference on Numerical Analysis, 1975}, pages 73--89. Springer, 2006.

\bibitem{fong2011minimum}
David Chin-lung Fong.
\newblock {\em Minimum-Residual Methods for Sparse Least-Squares Using Golub-Kahan Bidiagonalization}.
\newblock Stanford University, 2011.

\bibitem{fong2012cg}
David Chin-Lung Fong and Michael Saunders.
\newblock Cg versus minres: An empirical comparison.
\newblock {\em Sultan Qaboos University Journal for Science [SQUJS]}, 17(1):44--62, 2012.

\bibitem{greenbaum1997iterative}
Anne Greenbaum.
\newblock {\em Iterative methods for solving linear systems}.
\newblock SIAM, 1997.

\bibitem{hanke2017conjugate}
Martin Hanke.
\newblock {\em Conjugate gradient type methods for ill-posed problems}.
\newblock Chapman and Hall/CRC, 2017.

\bibitem{hastie2009elements}
Trevor Hastie, Robert Tibshirani, Jerome~H Friedman, and Jerome~H Friedman.
\newblock {\em The elements of statistical learning: data mining, inference, and prediction}, volume~2.
\newblock Springer, 2009.

\bibitem{hayami2001behaviour}
Ken Hayami.
\newblock On the behaviour of the conjugate residual method for singular systems.
\newblock {\em National Institute of Informatics}, 2001.

\bibitem{hayami2018convergence}
Ken Hayami.
\newblock Convergence of the conjugate gradient method on singular systems.
\newblock {\em arXiv preprint arXiv:1809.00793}, 2018.

\bibitem{hayami2004convergence}
Ken Hayami and Masaaki Sugihara.
\newblock {On the convergence of the GCR(k) method for singular systems}.
\newblock {\em National Institute of Informatics}, 2004.

\bibitem{hayami2011geometric}
Ken Hayami and Masaaki Sugihara.
\newblock A geometric view of krylov subspace methods on singular systems.
\newblock {\em Numerical Linear Algebra with Applications}, 18(3):449--469, 2011.

\bibitem{hestenes1952methods}
Magnus~R Hestenes and Eduard Stiefel.
\newblock Methods of conjugate gradients for solving linear systems.
\newblock {\em Journal of Research of the National Bureau of Standards}, 49(6):409--436, 1952.

\bibitem{kaasschieter1988preconditioned}
Erik~F Kaasschieter.
\newblock Preconditioned conjugate gradients for solving singular systems.
\newblock {\em Journal of Computational and Applied mathematics}, 24(1-2):265--275, 1988.

\bibitem{lim2023complexity}
Alexander Lim and Fred Roosta.
\newblock {Complexity Guarantees for Nonconvex Newton-MR Under Inexact Hessian Information}.
\newblock {\em arXiv preprint arXiv:2308.09912}, 2023.

\bibitem{linda2001shapiro}
G~Linda and C~George.
\newblock {\em Computer Vision}.
\newblock Prentice-Hall, New Jersey, 2001.

\bibitem{liu2023obtaining}
Yang Liu, Andre Milzarek, and Fred Roosta.
\newblock {Obtaining Pseudo-inverse Solutions With MINRES}.
\newblock {\em arXiv preprint arXiv:2309.17096}, 2023.

\bibitem{liu2022newton}
Yang Liu and Fred Roosta.
\newblock {A Newton-MR algorithm with complexity guarantees for nonconvex smooth unconstrained optimization}.
\newblock {\em arXiv preprint arXiv:2208.07095}, 2022.

\bibitem{liu2022minres}
Yang Liu and Fred Roosta.
\newblock Minres: From negative curvature detection to monotonicity properties.
\newblock {\em SIAM Journal on Optimization}, 32(4):2636--2661, 2022.

\bibitem{luenberger1969hyperbolic}
David~G Luenberger.
\newblock Hyperbolic pairs in the method of conjugate gradients.
\newblock {\em SIAM Journal on Applied Mathematics}, 17(6):1263--1267, 1969.

\bibitem{luenberger1970conjugate}
David~G Luenberger.
\newblock The conjugate residual method for constrained minimization problems.
\newblock {\em SIAM Journal on Numerical Analysis}, 7(3):390--398, 1970.

\bibitem{mardal2011preconditioning}
Kent-Andre Mardal and Ragnar Winther.
\newblock Preconditioning discretizations of systems of partial differential equations.
\newblock {\em Numerical Linear Algebra with Applications}, 18(1):1--40, 2011.

\bibitem{murphy2023probabilistic}
Kevin~P Murphy.
\newblock {\em {Probabilistic machine learning: Advanced topics}}.
\newblock MIT press, 2023.

\bibitem{nocedal1999numerical}
Jorge Nocedal and Stephen~J Wright.
\newblock {\em Numerical optimization}.
\newblock Springer, 1999.

\bibitem{paige1975solution}
Christopher~C Paige and Michael~A Saunders.
\newblock Solution of sparse indefinite systems of linear equations.
\newblock {\em SIAM journal on numerical analysis}, 12(4):617--629, 1975.

\bibitem{rebjock2023fast}
Quentin Rebjock and Nicolas Boumal.
\newblock Fast convergence of trust-regions for non-isolated minima via analysis of cg on indefinite matrices.
\newblock {\em arXiv preprint arXiv:2311.07404}, 2023.

\bibitem{roosta2022newton}
Fred Roosta, Yang Liu, Peng Xu, and Michael~W Mahoney.
\newblock {Newton-MR: Inexact Newton method with minimum residual sub-problem solver}.
\newblock {\em EURO Journal on Computational Optimization}, 10:100035, 2022.

\bibitem{roosta2015randomized}
Farbod Roosta-Khorasani.
\newblock {\em Randomized algorithms for solving large scale nonlinear least squares problems}.
\newblock PhD thesis, University of British Columbia, Vancouver, BC Canada V6T 1Z4, 2015.

\bibitem{royer2020newton}
Cl{\'e}ment~W Royer, Michael O’Neill, and Stephen~J Wright.
\newblock A newton-cg algorithm with complexity guarantees for smooth unconstrained optimization.
\newblock {\em Mathematical Programming}, 180:451--488, 2020.

\bibitem{royer2018complexity}
Cl{\'e}ment~W Royer and Stephen~J Wright.
\newblock Complexity analysis of second-order line-search algorithms for smooth nonconvex optimization.
\newblock {\em SIAM Journal on Optimization}, 28(2):1448--1477, 2018.

\bibitem{saad2003iterative}
Yousef Saad.
\newblock {\em Iterative methods for sparse linear systems}.
\newblock SIAM, 2003.

\bibitem{shewchuk2007conjugate}
J~Shewchuk.
\newblock Conjugate gradient without the agonizing pain.
\newblock {\em Technical report}, 2007.

\bibitem{sogabe2022krylov}
Tomohiro Sogabe.
\newblock Krylov subspace methods for linear systems.
\newblock {\em Springer Series in Computational Mathematics}, 2022.

\bibitem{stiefel1955relaxationsmethoden}
Edvard Stiefel.
\newblock Relaxationsmethoden bester strategie zur l{\"o}sung linearer gleichungssysteme.
\newblock {\em Commentarii Mathematici Helvetici}, 29(1):157--179, 1955.

\bibitem{yoon2016solving}
Myoungho Yoon, Gangjoon Yoon, and Chohong Min.
\newblock On solving the singular system arisen from poisson equation with neumann boundary condition.
\newblock {\em Journal of Scientific Computing}, 69:391--405, 2016.

\end{thebibliography}
	\appendix
	\section{Unifying Framework}\label{sec:appendix:unify}
    In this section, we present some technical results for \cref{alg:cg,alg:cr} that serve as foundational tools in all of our derivations in this paper. To simplify the exposition, we adopt a more general framework depicted in \cref{alg:cgcr}, which encapsulates both \cref{alg:cg,alg:cr} at the same time. Specifically, considering $s = 0$ and $s = 1$ in \cref{alg:cgcr} amount to \cref{alg:cg} and \cref{alg:cr}, respectively. In essence, it is shown that all common properties of CG and CR, well-known for positive definite settings, continue to hold under indefinite matrices. The proofs are straightforward adaptation of the PD, found in many classical resources, e.g., \cite{nocedal1999numerical, shewchuk2007conjugate}.

	\begin{algorithm}[htbp]
		\caption{Conjugate Direction Methods ($s=0$ is CG while $s=1$ is CR)} \label{alg:cgcr}
		\begin{algorithmic}[1]
			\Require $\f{A}$ and $\f{b}$, and $ \f{x}_{0} $
			\State Let $k = 1$, $\f{r}_0 = \f{p}_0 = \f{b} - \f{Ax}_0$
			\While{Not Terminated}
			\State $\alpha_{k-1}^{(s)} = \lr{\f{r}_{k-1}, \f{A}^{s}\f{r}_{k-1}}/\lr{\f{p}_{k-1}, \f{A}^{s+1}\f{p}_{k-1}}$ 
			\State $\f{x}_{k} = \f{x}_{k-1} + \alpha_{k-1}^{(s)}\f{p}_{k-1}$
			\State $\f{r}_{k} = \f{r}_{k-1} - \alpha_{k-1}^{(s)}\f{Ap}_{k-1}$
			\State $\beta_{k-1}^{(s)} = \lr{\f{r}_{k},\f{A}^{s}\f{r}_{k}}/\lr{\f{r}_{k-1},\f{A}^{s}\f{r}_{k-1}}$
            \State $\f{p}_{k} = \f{r}_{k} + \beta_{k-1}^{(s)}\f{p}_{k-1}$
			\State $k = k+1$
			\EndWhile
			\State \Return $\f{x}_k$
		\end{algorithmic}
	\end{algorithm}
	
	\begin{lemma}\label{lem:expanding_subspace}
		Let $\{\f{r}_0, \cdots, \f{r}_{k-1}\}$ and $\{\f{p}_0, \cdots, \f{p}_{k-1}\}$ be the set of vectors generated by \cref{alg:cgcr} up to iteration $k-1$, we have 
		\begin{align*}
			\Span(\f{p}_0, \f{p}_1, \cdots, \f{p}_{k-1}) = \Span(\f{r}_0, \f{r}_1, \cdots, \f{r}_{k-1}) = \Span(\f{r}_0, \f{Ar}_0, \cdots, \f{A}^{k-1}\f{r}_0).
		\end{align*}
	\end{lemma}
	\begin{proof}
		We prove by induction. For $j = 0$, we trivially have $\Span(\f{p}_0) = \Span(\f{r}_0)$. Now assume
		\begin{align*}
			\Span(\f{p}_0, \f{p}_1, \cdots, \f{p}_{i-1}) = \Span(\f{r}_0, \f{r}_1, \cdots, \f{r}_{i-1}) & = \Span(\f{r}_0, \f{Ar}_0, \cdots, \f{A}^{i-1}\f{r}_0)
		\end{align*}
		for some $i$. We show the claim for $i + 1$. Since $\f{p}_{i-1}  \in \Span(\f{r}_0, \f{Ar}_0, \cdots, \f{A}^{i-1}\f{r}_0)$, we have $ \f{Ap}_{i-1} \in \Span(\f{Ar}_0, \cdots, \f{A}^{i}\f{r}_0) $, which from $\f{r}_{i} = \f{r}_{i-1} - \alpha_{i-1}^{(s)}\f{Ap}_{i-1}$ implies $\f{r}_{i}  \in \Span(\f{r}_0, \f{Ar}_0, \cdots, \f{A}^{i}\f{r}_0)$. So, $\Span(\f{r}_0, \f{r}_1, \cdots, \f{r}_{i}) \subset \Span(\f{r}_0, \cdots, \f{A}^{i}\f{r}_0)$.
		However, from $ \f{A}^i\f{r}_0  = \f{A}(\f{A}^{i-1}\f{r}_0) $, we have 
		\begin{align*}
			\f{A}^i\f{r}_0 & \in \f{A}\Span(\f{p}_0, \f{p}_1,\cdots,\f{p}_{i-1})  = \Span(\f{Ap}_0, \f{Ap}_1,\cdots,\f{Ap}_{i-1})\\
			& = \Span((\f{r}_1 - \f{r}_0), (\f{r}_{2} - \f{r}_1),\cdots, (\f{r}_{i-1} - \f{r}_{i})) = \Span(\f{r}_0, \f{r}_1, \cdots, \f{r}_{i}).
		\end{align*}
		Together, we get $\Span(\f{b}, \f{Ab}, \cdots, \f{A}^{j}\f{b}) = \Span(\f{r}_1, \f{r}_2, \cdots, \f{r}_{j+1})$.
		Finally, 
		\begin{align*}
			\Span(\f{p}_0, \f{p}_1, \cdots, \f{p}_{i}) & = \Span(\f{p}_0, \f{p}_1, \cdots, \f{p}_{i-1}, \f{r}_{i}) = \Span(\f{r}_0, \f{Ar}_0, \cdots, \f{A}^{i-1}\f{r}_0, \f{r}_{i})\\
			& = \Span(\f{r}_0, \f{r}_1, \cdots, \f{r}_i, \f{r}_{i-1}),
		\end{align*}
		where the first equality is by $\f{p}_i = \f{r}_i - \beta_i^{(s)}\f{p}_{i-1}$
	\end{proof}
    
	\begin{lemma}\label{lem:properties_CD}
		Let $\{\f{r}_0, \cdots, \f{r}_{k-1}\}$ and $\{\f{p}_0, \cdots, \f{p}_{k-1}\}$ be the set of vectors generated by \cref{alg:cgcr} up to iteration $k-1$, for any $ 0 \leq i < j \leq k-1$, we have 
		\begin{subequations}
			\label{eq:conjugation}
			\begin{align}
			\lr{\f{r}_i, \f{A}^{s}\f{r}_j} &= 0, \label{eq:sconj_r} \\
			\lr{\f{p}_i, \f{A}^{s+1}\f{p}_j} &= 0, \label{eq:Asconj_p}
		\end{align}
		\end{subequations}
	\end{lemma}
	\begin{proof} 
		Without loss of generality, assume $k > 1$, we prove the claims by induction. For $i+1 = j = 1$, $\lr{\f{r}_{0}, \f{A}^{s}\f{r}_{1}} = \lr{\f{r}_{0}, \f{A}^{s}(\f{r}_0 - \alpha_{0}^{(s)} \f{A} \f{p}_0)}  = 0$,
		by definition of $ \alpha_{0}^{(s)} $. Since $\alpha_0^{(s)} \neq 0$, then
		\begin{align*}
			\lr{\f{p}_{0}, \f{A}^{s+1}\f{p}_{1}} & = \lr{\f{A} \f{p}_{0}, \f{A}^{s}(\f{r}_1 + \beta_0^{(s)}\f{p}_0)} \\
			&= \frac{1}{\alpha_{0}^{(s)}}\lr{\f{r}_0 - \f{r}_1, \f{A}^{s}\f{r}_1} + \frac{\beta_0^{(s)}}{\alpha_{0}^{(s)}}\lr{\f{r}_0 - \f{r}_1, \f{A}^{s}\f{p}_0} \\
			&= -\frac{1}{\alpha_{0}^{(s)}}\lr{\f{r}_1, \f{A}^{s}\f{r}_1} + \frac{\beta_0^{(s)}}{\alpha_{0}^{(s)}}\lr{\f{r}_0, \f{A}^{s}\f{p}_0} = 0,
		\end{align*}
		by definition of $ \beta_{0}^{(s)} $. Assume 
		\begin{subequations}
			\label{eq:conjugation_prev}
			\begin{align}
				\lr{\f{r}_i, \f{A}^{s}\f{r}_j} &= 0, \quad 0 \leq i < j \leq k-1, \label{eq:sconj_r_prev} \\
				\lr{\f{p}_i, \f{A}^{s+1}\f{p}_j} &= 0, \quad 0 \leq i < j \leq k-1, \label{eq:Asconj_p_prev}
			\end{align}
		\end{subequations}
		for some $k$. Now we show the claim holds for $0 \leq i < j \leq k$. Let $j = k$,
		for $i = k - 1$, we have
		\begin{align*}
			\lr{\f{r}_{k-1}, \f{A}^{s}\f{r}_{k}} &= \lr{\f{r}_{k-1}, \f{A}^{s}(\f{r}_{k-1} - \alpha_{k-1}^{(s)} \f{A} \f{p}_{k-1})}\\
            &= \lr{\f{r}_{k-1}, \f{A}^{s}\f{r}_{k-1}} - \alpha_{k-1}^{(s)} \lr{(\f{p}_{k-1} - \beta_{k-2}^{(s)}\f{p}_{k-2})\f{A}^{s+1}\f{p}_{k-1}}\\
            &= \lr{\f{r}_{k-1}, \f{A}^{s}\f{r}_{k-1}} - \alpha_{k-1}^{(s)} \lr{\f{p}_{k-1},\f{A}^{s+1}\f{p}_{k-1}} = 0,
		\end{align*}
		where the second last and the last equality follow from the inductive hypothesis and the definition of $ \alpha_{k-1}^{(s)} $. For $i < k - 1$, we have
		\begin{align*}
			\lr{\f{r}_{i}, \f{A}^{s}\f{r}_{k}} &= \lr{\f{r}_{i}, \f{A}^{s}(\f{r}_{k-1} - \alpha_{k-1}^{(s)} \f{A} \f{p}_{k-1})}\\
            &= - \alpha_{k-1}^{(s)} \lr{\f{p}_{i} - \beta_{i-1}^{(s)} \f{p}_{i-1}, \f{A}^{s+1} \f{p}_{k-1}} = 0,
		\end{align*}
		by \cref{eq:conjugation_prev}. This gives \cref{eq:sconj_r}. We now show \cref{eq:Asconj_p}. For $i = k - 1$, we have
		\begin{align*}
			\lr{\f{p}_{k-1}, \f{A}^{s+1}\f{p}_{k}} &= \lr{\f{A} \f{p}_{k-1}, \f{A}^{s}(\f{r}_{k} + \beta_{k-1}^{(s)} \f{p}_{k-1})} \\
			&= \frac{1}{\alpha_{k-1}^{(s)}}\lr{\f{r}_{k-1} - \f{r}_{k}, \f{A}^{s}\f{r}_{k}} + \beta_{k-1}^{(s)} \lr{\f{p}_{k-1}, \f{A}^{s+1} \f{p}_{k-1}} \\
            &= -\frac{1}{\alpha_{k-1}^{(s)}} \lr{\f{r}_{k}, \f{A}^{s}\f{r}_{k}} + \beta_{k-1}^{(s)}\lr{\f{p}_{k-1}, \f{A}^{s+1}\f{r}_{k-1}}\\
			&= -\frac{1}{\alpha_{k-1}^{(s)}} \lr{\f{r}_{k}, \f{A}^{s}\f{r}_{k}} + \frac{\beta_{k-1}^{(s)}}{\alpha_{k-1}^{(s)}} \lr{\f{r}_{k-1}, \f{A}^{s}\f{r}_{k-1}} = 0,
		\end{align*}
		by \cref{eq:sconj_r} and the definition of $ \alpha_{k-1}^{(s)} $ and $ \beta_{k-1}^{(s)} $. For $i < k - 1$, we have
		\begin{align*}
			\lr{\f{p}_{i}, \f{A}^{s+1}\f{p}_{k}} &= \lr{\f{p}_{i}, \f{A}^{s+1}(\f{r}_{k} + \beta_{k-1}^{(s)} \f{p}_{k-1})} \\
            & = \frac{1}{\alpha_{i-1}^{(s)}}\lr{\f{r}_{i} - \f{r}_{i + 1}, \f{A}^{s} \f{r}_{k}} + \beta_{k-1}^{(s)} \lr{\f{p}_{i}, \f{A}^{s+1}\f{p}_{k-1}} = 0,
		\end{align*}
		by \cref{eq:sconj_r,eq:Asconj_p_prev} and our induction assumptions.
	\end{proof}

	\begin{lemma}\label{lem:rp}
		Let $\{\f{r}_0, \cdots, \f{r}_{k-1}\}$ and $\{\f{p}_0, \cdots, \f{p}_{k-1}\}$ be the set of vectors generated by \cref{alg:cgcr} up to iteration $k-1$, then, for $0 \leq j < i \leq k-1$, 
		\begin{subequations}
			\label{eq:rp}
			\begin{align}
				\lr{\f{r}_i, \f{A}^{s}\f{p}_j} &= 0, \label{eq:rp=0}
                \intertext{and, for $0 \leq i \leq j \leq k-1$,}
				\lr{\f{r}_i, \f{A}^{s}\f{p}_j} &= \lr{\f{r}_j, \f{A}^{s}\f{r}_j}. \label{eq:rp=rr}
			\end{align}
		\end{subequations}
	\end{lemma}
	\begin{proof}
		We will prove this by induction. First, we show the base case for \cref{eq:rp=0}, for $j+1 = i = 1$,
		\begin{align*}
			\lr{\f{r}_1, \f{A}^{s}\f{p}_0} & = \lr{\f{r}_0 - \alpha_0^{(s)}\f{A}\f{p}_0, \f{A}^{s}\f{p}_0} = \lr{\f{r}_0, \f{A}^{s}\f{p}_0} - \alpha_0^{(s)}\lr{\f{p}_0, \f{A}^{s+1}\f{p}_0} = 0,
		\end{align*}
		where the last equality is by the definition of $\alpha_0^{(s)}$ and $\f{r}_0$. Next, we show the base case for \cref{eq:rp=rr}, for $j = i = 0$, we obtain the desired result trivially. Now, for $j = i = 1$, we obtain $\lr{\f{r}_1, \f{A}^{s}\f{p}_1} = \lr{\f{r}_1, \f{A}^{s} \left(\f{r}_1 + \beta_{0}^{(s)}\f{p}_{0}\right)} = \lr{\f{r}_1, \f{A}^{s}\f{r}_1}$,
		by the previous result. For $i + 1= j = 1$, we have
		\begin{align*}
			\lr{\f{r}_0, \f{A}^{s}\f{p}_1} = \lr{\f{r}_0, \f{A}^{s} \left(\f{r}_1 + \beta_{0}^{(s)}\f{p}_{0}\right)} = \lr{\f{r}_0, \f{A}^{s} \f{r}_1} + \beta_{0}^{(s)} \lr{\f{r}_0, \f{A}^{s} \f{p}_{0}} = \lr{\f{r}_1, \f{A}^{s}\f{r}_1},
		\end{align*}
		where the last line is by the definition of $\beta_0^{(s)}$ and \cref{eq:sconj_r}. Now, assume
		\begin{align*}
			\lr{\f{r}_i, \f{A}^{s}\f{p}_j} &= 0, \quad 0 \leq j < i \leq k-1, \\
			\lr{\f{r}_i, \f{A}^{s}\f{p}_j} &= \lr{\f{r}_j, \f{A}^{s}\f{r}_j}, \quad 0 \leq i \leq j \leq k-1
		\end{align*}
		for some $k$. We now show the claim for $k+1$. Let $i = k$, for $j = k - 1$, we have
		\begin{align*}
			\lr{\f{r}_{k}, \f{A}^{s}\f{p}_{k-1}} & = \lr{\f{r}_{k-1} - \alpha_{k-1}^{(s)}\f{Ap}_{k-1}, \f{A}^{s}\f{p}_{k-1}}\\
            &= \lr{\f{r}_{k-1}, \f{A}^{s}\f{p}_{k-1}} - \alpha_{k-1}^{(s)} \lr{\f{p}_{k-1}, \f{A}^{s+1}\f{p}_{k-1}} = 0,
		\end{align*}
		by the definition of $\alpha_{k-1}^{(s)}$ and the inductive hypothesis. For $j < k - 1$, we get
		\begin{align*}
			\lr{\f{r}_{k}, \f{A}^{s}\f{p}_j} & = \lr{\f{r}_{k-1} - \alpha_{k-1}^{(s)}\f{Ap}_{k-1}, \f{A}^{s}\f{p}_j} \\
            &= \lr{\f{r}_{k-1}, \f{A}^{s}\f{p}_j} - \alpha_{k-1}^{(s)} \lr{\f{p}_{k-1}, \f{A}^{s+1}\f{p}_j} = 0,
		\end{align*}
		where the first term and the second term each evaluate to zero by the inductive hypothesis and \cref{eq:Asconj_p}, respectively. This gives us \cref{eq:rp=0}. Now, let $j = k$, for $i = k$, we have 
        \begin{align*}
            \lr{\f{r}_{k}, \f{A}^{s}\f{p}_{k}} = \lr{\f{r}_{k}, \f{A}^{s} \left(\f{r}_{k} + \beta_{k-1}^{(s)}\f{p}_{k-1}\right)} = \lr{\f{r}_{k}, \f{A}^{s}\f{r}_{k}},
        \end{align*}
		by \cref{eq:rp=0}. For $i < k$, we have
		\begin{align*}
			\lr{\f{r}_{i}, \f{A}^{s}\f{p}_{k}} &= \lr{\f{r}_{i}, \f{A}^{s} \left(\f{r}_{k} + \beta_{k-1}^{(s)}\f{p}_{k-1}\right)} = \lr{\f{r}_{i}, \f{A}^{s} \f{r}_{k}} + \beta_{k-1}^{(s)} \lr{\f{r}_{i}, \f{A}^{s} \f{p}_{k-1}} \\
			&= \beta_{k-1}^{(s)} \lr{\f{r}_{k-1}, \f{A}^{s} \f{r}_{k-1}} = \lr{\f{r}_{k}, \f{A}^{s}\f{r}_{k}},
		\end{align*}
		where the last line is by the definition of $\beta_0^{(s)}$, and the second equality due to \cref{eq:sconj_r} and the inductive hypothesis, respectively. This gives us \cref{eq:rp=rr}.
	\end{proof}
 
	\section{Further Results}
	\label{sec:appendix:results}
	In this section, we gather a few results that are used in the proofs throughout the paper. We first provide a simple yet useful result that holds in general regardless of consistency. 
	\begin{lemma}\label{lem:p_in_null_b}
		Under \cref{assumpt:unlucky}, in \cref{alg:cgcr} we have $(\f{I} - \f{A}^\dagger\f{A})\f{p}_{k-1} = c^{(s)}_{k-1}\f{b}^\perp$ for $1 \leq k \leq g$, where $c^{(s)}_{k-1} = 1 + \beta^{(s)}_{k-2}c^{(s)}_{k-2}, \; k \geq 1$,  with $c^{(s)}_0 = 1$.
	\end{lemma}
	\begin{proof}
		We first note that, for any $ \f{x} $, $(\f{I} - \f{A}^\dagger\f{A})\f{r} = (\f{I} - \f{A}^\dagger\f{A})(\f{b} - \f{A}\f{x}) = \f{b}^\perp$ where we have used the fact that, for the symmetric matrix $ \f{A} $, we have $ \f{A}^{\dagger}\f{A} = \f{A} \f{A}^{\dagger} $.
		The proof follows an inductive argument. The result trivially holds for $k=1$. Suppose the claim holds for any $ 1 \leq k < g $. For $k+1$, we have 
		\begin{align*}
			(\f{I} - \f{A}^\dagger\f{A})\f{p}_{k} & = (\f{I} - \f{A}^\dagger\f{A})(\f{r}_{k} + \beta_{k-1}^{(s)}\f{p}_{k-1}) = \f{b}^\perp + \beta^{(s)}_{k-1}(\f{I} - \f{A}^\dagger\f{A})\f{p}_{k-1}\\
            &= \left(1 + \beta^{(s)}_{k-1} c^{(s)}_{k-1} \right) \f{b}^\perp.
		\end{align*}
	\end{proof}
	
    The next result, which is crucial for several proofs, particularly for the derivations in \cref{sec:appendix:or}, states that the contribution of the null space of $ \mathbf{A} $ to $ \mathbf{b} $, namely $ \mathbf{b}^{\perp} $, is included in the Krylov subspace only once the grade of $ \mathbf{b} $ with respect to $ \mathbf{A} $ is reached.
    \begin{lemma}\label{lem:b_perp_in_krylov_g}
        Let $\f{b} \notin \Range(\f{A})$. Then, $\f{b}^\perp \in \krylov{A}{b}{k}$ if and only if $k \geq g$.
    \end{lemma}
    \begin{proof}
        First, we note that since $\f{b} \notin \Range(\f{A})$, we must have $\f{b} \neq \f{0}$ and $\f{b}^\perp \neq \f{0}$. Suppose $\f{b}^\perp \in \krylov{A}{b}{k}$ but $k < g$. Recall that in this case $\{\f{b}, \f{A}\f{b}, \cdots, \f{A}^{k-1}\f{b}, \f{A}^{k}\f{b}\}$ are linearly independent. We can write
        \begin{align*}
            \f{b}^\perp & = a_0\f{b} + a_1\f{Ab} + \cdots + a_{k-1}\f{A}^{k-1}\f{b},
        \end{align*}
        which gives
        \begin{align*}
            \f{0} = \f{Ab}^{\perp} & = a_0\f{Ab} + a_1\f{A}^2\f{b} + \cdots + a_{k-1}\f{A}^{k}\f{b} \in \krylov{A}{b}{k+1}.
        \end{align*}
        Since not all the coefficients $a_i$, $0 \leq i \leq k-1$, are zero, the vectors $\f{Ab}, \f{A}^2\f{b}, \cdots, \f{A}^{k}\f{b}$ must be linearly dependent, which leads to a contradiction. 
        Conversely,  suppose $k\geq g$. We can write $\f{0} = a_0\f{b} + a_1\f{Ab} + \cdots + a_{g}\f{A}^g\f{b}$ where not all coefficients $a_i, 0 \leq i \leq g$, are zero. By assumption $\f{b} \notin \Range(\f{A})$, which implies that $a_0 = 0$.  This gives  
        \begin{align*}
            \f{0} & = a_1\f{Ab} + a_2\f{A}^2\f{b} + \cdots + a_g\f{A}^g\f{b} = \f{A}(a_1\f{b} + a_2\f{A}\f{b} + \cdots + a_g\f{A}^{g-1}\f{b}) = \f{Ac},
        \end{align*}
        where $\f{c} = a_1\f{b} + a_2\f{A}\f{b} + \cdots + a_g\f{A}^{g-1}\f{b} \in \krylov{A}{b}{g}$. Hence, $\f{0} \neq \f{c} \in \Null(\f{A})$. Clearly, $a_1\neq 0$, as otherwise $\f{c} \in \Range(\f{A})$, which implies $\f{c} = \f{0}$. Now, from $\f{b} = \ft{b} + \f{b}^\perp$, we get 
        \begin{align*}
            \frac{1}{a_{1}} \f{c} - \f{b}^\perp =   \ft{b} + \frac{1}{a_{1}}\left( a_2\f{A}\f{b} + \cdots + a_g\f{A}^{g-1}\f{b}\right).
        \end{align*}
        The left-hand side is a vector in the null space of $\f{A}$, where as the right-hand side is in the range of $\f{A}$. This equality can only hold if both sides are zero vectors. Consequently, we have  $ \f{b}^\perp = \f{c} /a_{1}\in \krylov{A}{b}{g}$
    \end{proof}
    
    The next result states that if the subspace $\mathcal{K}_k(\mathbf{A}, \mathbf{b})$ contains any normal solution, it must be of a specific format. 
    \begin{lemma}\label{lem:normal_solution_y_b}
        If $\f{A}^\dagger\f{b} + \f{y} \in \krylov{A}{b}{k}$ for some $\f{y} \in \Null(\f{A})$, then $\f{y} = a(\f{I} - \f{A}\f{A}^\dagger)\f{b}$ for some $a \in \real$.
	\end{lemma}
    \begin{proof}
        Suppose $\f{A}^\dagger\f{b} + \f{y} \in \krylov{A}{b}{k}$ for some $\f{y} \in \Null(\f{A})$. We have $\f{A}^\dagger\f{b} + \f{y} = a_0\f{b} + a_1\f{Ab} + \cdots + a_{k-1}\f{A}^{k-1}\f{b}$, which implies $\f{y} = (\f{I} - \f{A}\f{A}^\dagger) \f{y} =  a_0 (\f{I} - \f{A}\f{A}^\dagger)\f{b}$.
    \end{proof}
    
    The next result  indicates that the underlying Krylov subspace can contain a normal solution only if it is large enough. This result is likely already known; however, we could not pinpoint a source to reference it. Therefore, we include it here for completeness. 
    \begin{lemma}\label{lem:solution_in_krylov}
        \hfill
		\begin{itemize}
		    \item Let $\f{b} \notin \Range(\f{A})$. Then $\f{A}^\dagger\f{b} + a(\f{I} - \f{A}\f{A}^\dagger)\f{b} \in \krylov{A}{b}{k}$ for some $a \in \real$ if and only if $k \geq g - 1$.
            \item  Let $\f{b} \in \Range(\f{A})$. Then $\f{A}^\dagger\f{b} \in \krylov{A}{b}{k}$ if and only if $k \geq g$.
		\end{itemize}
	\end{lemma}
	\begin{proof}
        \hfill
		\begin{itemize}
		    \item Suppose $\f{b} \notin \Range(\f{A})$, which in particular implies $\f{b} \neq \f{0}$ and $\f{b}^\perp \neq \f{0}$. Let $k \geq g-1$. By \cref{lem:b_perp_in_krylov_g}, we know that $\f{b}^\perp \in \krylov{A}{b}{g}$. We write, $-\f{b}^\perp = a_0\f{b} + a_1 \f{Ab} + \cdots + a_{g-1}\f{A}^{g-1}\f{b}$, which can only hold if $a_0 = -1$. Hence, 
        \begin{align*}
            \ft{b} & = a_1 \f{Ab} + a_2\f{A}^2\f{b} + \cdots + a_{g-1}\f{A}^{g-1}\f{b} = \f{A}(a_1 \f{b} + a_2\f{A}\f{b} + \cdots + a_{g-1}\f{A}^{g-2}\f{b}) \\
            & = \f{A}(\underbrace{a_1\ft{b} + \cdots + a_{g-1}\f{A}^{g-2}\f{b}}_{\triangleq \,\,\f{c}} + \underbrace{a_1\f{b}^\perp}_{\triangleq \,\, \f{c}'}).
        \end{align*}
        From $\f{c} \in \Range(\f{A})$, it follows that 
        \begin{align*}
            \f{A}^{\dagger}\f{b} = \f{A}^{\dagger}\ft{b} = \f{A}^{\dagger} \f{A} \f{c} = \f{c}.
        \end{align*}
        Now, since $\f{c} + \f{c}' \in \krylov{A}{b}{g-1} \subseteq \krylov{A}{b}{g}$, we get the desired result. Conversely, assume that for some  $k < g - 1$, $\krylov{A}{b}{k}$ contains a  normal solution, i.e., $\f{A}^\dagger\f{b} + (\f{I} - \f{AA}^\dagger)\f{y} \in \krylov{A}{b}{k}$ for some $\f{y}$ and $k < g - 1$. We can write 
        \begin{align*}
            \f{A}^\dagger\f{b} + (\f{I} - \f{AA}^\dagger)\f{y} & = a_0\f{b} + a_1\f{Ab} + \cdots + a_{k-1}\f{A}^{k-1}\f{b},
        \end{align*}
        which gives
        \begin{align*}
            \ft{b} = \f{A} \f{A}^\dagger\f{b} = \f{A}(\f{A}^\dagger\f{b} + (\f{I} - \f{AA}^\dagger)\f{y}) & = a_0\f{Ab} + a_1\f{A}^2\f{b} + \cdots + a_{k-1}\f{A}^{k}\f{b}.
        \end{align*}
        Hence, 
        \begin{align*}
            -\f{b}^\perp & = - \f{b} + a_0\f{Ab} + a_1\f{A}^2\f{b} + \cdots + a_{k-1}\f{A}^{k}\f{b} \in \krylov{A}{b}{k+1},
        \end{align*}
        which contradicts \cref{lem:b_perp_in_krylov_g}. So, $k \geq g-1$. 
        \item Suppose $\f{b} \in \Range(\f{A})$. If $\f{b} = \f{0}$, then the results is trivial. Assume $\f{b} \neq \f{0}$ and  $\f{A}^\dagger\f{b} \in \krylov{A}{b}{k}$ but $k < g$. We can write
        \begin{align*}
            \f{A}^\dagger\f{b} & = a_0\f{b} + a_1\f{Ab} + \cdots + a_{k-1}\f{A}^{k-1}\f{b},
        \end{align*}
        which implies
        \begin{align*}
            \f{b} = \f{AA}^\dagger\f{b} & = a_0\f{Ab} + a_1\f{A}^2\f{b} + \cdots + a_{k-1}\f{A}^{k}\f{b}.
        \end{align*}
        However, this 
        leads to a contradiction, since $\f{b}, \f{Ab}, \cdots, \f{A}^k\f{b}$ are linearly independent vectors for $k \leq g-1$. So, we must have $k \geq g$. Conversely, let $k \geq g$ and consider $\krylov{A}{b}{g+1}$. Since the vectors $\f{b}, \f{Ab}, \cdots, \f{A}^{g}\f{b}$ are linearly dependent, there exist coefficients $a_0, a_1, \ldots, a_{g}$, not all zero, such that $\f{0} = -a_0\f{b} + a_1\f{Ab} + \cdots + a_g\f{A}^g\f{b}$. In particular, we must have $a_0 \not= 0$. Otherwise, applying $\f{A}^{\dagger}$ twice to both sides yields
        \begin{align*}
            - a_{1} \f{A}^{\dagger} \f{b} = a_2\f{b} + a_3\f{Ab} + \cdots + a_g\f{A}^{g-2}\f{b} \in \krylov{A}{b}{g-1}.
        \end{align*}
        However, by the first part of the present proof, this cannot occur. Hence, $a_0 \neq 0$, and we can write 
        \begin{align*}
            a_0 \f{A}^{\dagger} \f{b} = a_1\f{b} + \cdots + a_g\f{A}^{g-1}\f{b},
        \end{align*}
        which implies $\f{A}^\dagger\f{b} \in \krylov{A}{b}{g}$.
		\end{itemize}
		
	\end{proof}
    
	\begin{lemma}\label{lem:residual_goes_to_zero}
		Let $\f{b}\in\Range(\f{A})$. Under \cref{assumpt:unlucky}, in \cref{alg:cgcr}, we have $\f{r}_{g} = \f{0}$.
	\end{lemma}
	\begin{proof}
		We first note that when $\f{b}\in\Range(\f{A})$, for the pseudo-inverse solution $\f{A}^\dagger\f{b}$, we have $\f{r}^{\star} = \f{b} - \f{A}\f{A}^\dagger\f{b} = \zero$.
		Let $\f{x}^{\star} \in \krylov{A}{b}{g}$ be the pseudo-inverse solution from \cref{lem:solution_in_krylov}. Using \cref{lem:expanding_subspace}, we can express $\f{x}^{\star}$ as
		\begin{align*}
			\f{x}^{\star} & = \sum_{i=1}^g\alpha_{i-1}^{(s)}\f{p}_{i-1},
		\end{align*}
		for some coefficients $\alpha_{i-1}^{(s)}$'s. It suffices to show that $\alpha_i^{(s)}$'s correspond to those from \cref{alg:cgcr}. From $ \zero = \f{r}^{\star}  = \f{b} - \f{A}\f{x}^{\star} = \f{b} - \sum_{i=1}^g\alpha_{i-1}^{(s)}\f{Ap}_{i-1} $, we have
		\begin{align*}
			\f{0} = \f{A}^{s}\f{r}^{\star} & = \f{A}^{s}\f{b} - \sum_{i=1}^g\alpha_{i-1}^{(s)}\f{A}^{s+1}\f{p}_{i-1}.
		\end{align*}
		Multiplying $\f{p}_j$, for each $0 \leq j \leq g-1$, from the left, using \cref{lem:properties_CD} as well as  $\f{p}_j = \f{r}_j + \beta_{j-1}\f{p}_{j-1}$, gives $\alpha_{i-1}^{(s)} = \lr{\f{r}_{i-1}, \f{A}^{s}\f{r}_{i-1}}/ \lr{\f{p}_{i-1}, \f{A}^{s+1}\f{p}_{i-1}}$, which are precisely the step-sizes calculated within \cref{alg:cgcr}. Thus, $\f{r}_{g} = \f{r}^{\star} = \zero$.
	\end{proof}

    \section{Orthogonal-Residual and Inconsistent Symmetric Systems}
    \label{sec:appendix:or}
    In this section, we provide further insights into why CG fails to converge for inconsistent systems. Traditionally, CG is viewed as a specific algorithmic implementation of the Orthogonal-Residual (OR) framework\footnote{Also referred to as the Full Orthogonalization Method (FOM) in other works \cite{saad2003iterative}.} when applied to PD matrices. More generally, OR enforces the following projection conditions:  
    \begin{subequations}
    \label{cond:OR}
    \begin{align}
        \f{x}_k & \in \krylov{A}{b}{k}, \label{cond:cg_1}\\ 
        \f{r}_k & \perp \krylov{A}{b}{k}, \label{cond:cg_2}
    \end{align}
    \end{subequations}
    where $\f{A}$ is a square but not necessarily symmetric matrix. In the special case where $\f{A}$ is PD, the Arnoldi process to construct a basis for $\krylov{A}{b}{k}$ reduces to the Lanczos process, and OR can be algorithmically implemented as \cref{alg:cg}, i.e., the CG algorithm (see \cite{saad2003iterative} for further details).
    \begin{remark}
        Under this interpretation, we distinguish between the \textit{framework} and its \textit{algorithmic implementations}. For example, the Minimum-Residual (MR) framework is defined as a projection framework that enforces $\mathbf{x}_k \in \mathcal{K}_k(\mathbf{A}, \mathbf{b})$ and $\mathbf{r}_k \perp \mathbf{A}\mathcal{K}_k(\mathbf{A}, \mathbf{b})$. One particular algorithmic implementation of the MR framework is GMRES. For symmetric matrices, other implementations include the CR and MINRES algorithms; see \cref{table:overview_methods_implementations} for an overview.
    \end{remark}
    \begin{table}[htbp]
	\centering
    \bgroup
    \def\arraystretch{1.2}
	\scalebox{0.95}{\begin{tabular}{|c|c|c|c|c|c|c|}
        \hline
        \multirow{3}{3cm}{\centering Krylov Subspace Projection Framework} & \multicolumn{6}{c|}{Algorithmic implementations of the framework}\\
		\cline{2-7}
		& \multirow{2}{*}{$\f{A}$ is square} & \multirow{2}{*}{$\f{A}$ is PD} & \multicolumn{2}{c|}{$\f{A}$ is PSD} & \multicolumn{2}{c|}{$\f{A}$ is indefinite}\\
        \cline{4-7}
        & & & $\f{b} \in \text{R}(\f{A})$ & $\f{b} \notin \text{R}(\f{A})$ & $\f{b} \in \text{R}(\f{A})$ & $\f{b} \notin \text{R}(\f{A})$ \\
        \hline
    	MR framework & \multirow{2}{*}{GMRES} & CR, & CR, & CR, & CR*, & CR*,\\[-4pt]
        $\f{r}_k \perp \f{A}\krylov{A}{b}{k}$ & & MINRES & MINRES & MINRES & MINRES & MINRES\\
        \hline
		OR framework & \multirow{2}{*}{FOM\cite{saad2003iterative}} & CG, & CG, & \multirow{2}{*}{\xmark} & CG*, & \multirow{2}{*}{\xmark} \\[-4pt]
        $\f{r}_k \perp \krylov{A}{b}{k}$ & & LanczosCG\cite{choi2006iterative} & LanczosCG & & LanczosCG* & \\
        \hline
	\end{tabular}}
    \egroup
    \caption{Range$(\f{A})$ is abbreviated as R$(\f{A})$ in this table. An overview of the MR and OR frameworks, along with their respective algorithmic implementations under various conditions. The symbols $\ast$ and \xmark \, indicate, respectively, the necessity of \cref{assumpt:unlucky} and the nonexistence of such algorithms (see \cref{thm:or}).\label{table:overview_methods_implementations}}
    \end{table}
    As demonstrated in \cref{sec:cg}, CG fails to converge to a normal solution for inconsistent systems. Given that CG is derived from the OR framework, one might ask whether an alternative algorithmic implementation of OR for symmetric matrices could converge to a normal solution under inconsistent systems. We establish that no such implementation exists. To demonstrate this, we restate \cref{lem:singular_T_at_g}, originally presented in \cite[Proposition 2.7]{choi2006iterative}, and provide a proof here for completeness.
    \begin{lemma}\label{lem:singular_T_at_g}
        Let $\f{b} \in \real^{d}$ and $\f{A} \in \real^{d \times d}$ be  symmetric matrix. Consider the Lanczos process, which generates the orthonormal matrix $\f{V}_g \in \mathbb{R}^{d\times g}$ and the tridiagonal matrix $\f{T}_g \in \mathbb{R}^{g \times g}$, such that $\f{AV}_g = \f{V}_g\f{T}_g$. We have $\f{b} \notin \Range(\f{A})$ if and only if $\f{T}_g$ is singular.
    \end{lemma}
    \begin{proof}
        First, we note that since $\f{b} \notin \Range(\f{A})$, we must have $\f{b} \neq \f{0}$ and $\f{b}^\perp \neq \f{0}$. By the Lanczos process and \cref{lem:b_perp_in_krylov_g}, given that $ \f{b}^\perp \in \krylov{A}{b}{g} = \Span(\f{V}_g) $, there exists a nonzero vector $ \f{c} \in \mathbb{R}^g $ such that $ \f{V}_g\f{c} = \f{b}^\perp $. Furthermore, since  
        $$ \f{T}_g\f{c} = \f{V}_g^\top\f{AV}_g\f{c} = \f{V}_g^\top\f{A}\f{b}^\perp = \f{0}, $$  
        it follows that $ \f{T}_g $ is singular.  
        Conversely, let $ \f{c} \neq \f{0} $ be a vector in $ \Null(\f{T}_g) $. Then,  $\f{AV}_g\f{c} = \f{V}_g\f{T}_g\f{c} = \f{0}$.   
        Hence, we must have $ \f{0} \neq \f{V}_g\f{c} \in \Null(\f{A}) $.  As a result, given that  
        $ \Span(\f{V}_g) = \Span([\f{b}, \f{Ab}, \dots, \f{A}^{g-1}\f{b}])$,   
        we conclude that $ \f{V}_g\f{c} = \alpha\f{b}^\perp $ for some nonzero constant $ \alpha $.
    \end{proof}
    We are now ready to state and prove the main result of this section.
    \begin{theorem}\label{thm:or}
            Any algorithm derived from the Orthogonal-Residual (OR) framework for symmetric matrices cannot converge to a solution of the normal equation under inconsistent systems.  
        \end{theorem}
    \begin{proof}
        By \cref{lem:solution_in_krylov}, a normal solution can only exist in either $\krylov{A}{b}{g-1}$ or $\krylov{A}{b}{g}$. The proof involves demonstrating that any algorithm based on the OR framework cannot find a normal solution in either of these two subspaces.
        Assume the contrary that , i.e., there exists an algorithm based on the OR framework \cref{cond:OR} generating $\mathbf{x}_{k} = \mathbf{A}^\dagger\mathbf{b} + a(\mathbf{I} - \mathbf{A}\mathbf{A}^\dagger)\mathbf{b} \in \mathcal{K}_{k}(\mathbf{A}, \mathbf{b})$ for some $a$, where $k = g-1$ or $k=g$ (recall that by \cref{lem:normal_solution_y_b}, this can be the only form of a normal solution in these spaces). 
        
        \begin{enumerate}[label = (\roman*)]
        \item First, consider the former case of $k=g-1$. Let  $\f{V}_k \in \mathbb{R}^{d\times k}$ be  the orthogonal matrix whose columns are the basis vectors for $\krylov{A}{b}{k}$ such that $\f{AV}_k = \f{V}_{k}\f{T}_k + \beta_{k} \f{v}_{k+1} \f{e}_{k}^\top$ where $\f{T}_{k} \in \mathbb{R}^{k \times k}$ is a symmetric tridiagonal matrix. Such construction is guaranteed by the Lanczos process.  From \cref{cond:cg_1}, we can write $\mathbf{x}_{g-1} = \mathbf{V}_{g-1}\mathbf{y}$, which implies,
        \begin{align*}
            \mathbf{0} = \mathbf{V}_{g-1}^\top\mathbf{r}_{g-1} & = \mathbf{V}_{g-1}^\top(\mathbf{b} - \mathbf{A}\mathbf{V}_{g-1}\mathbf{y}) = \|\mathbf{b}\|\mathbf{e}_1 - \mathbf{T}_{g-1}\mathbf{y},
        \end{align*}
        where the first equality follows from \cref{cond:cg_2}. 
        Our proof now consists of showing that no vector $\mathbf{y}$ can satisfy the above equality. 
        
        \begin{itemize}
            \item Suppose $\f{T}_{g-1}$ is nonsingular. We have $\f{y} = \f{T}_{g-1}^{-1}\|\f{b}\|\f{e}_1$. Since $\f{x}_{g-1}$ is a normal solution, i.e., $\f{A}\f{r}_{g-1} = \f{0} $, it follows that
        \begin{align*}
             \f{0} & = \f{A}(\f{b} - \f{Ax}_{g-1}) = \f{A}(\f{b} - \f{AV}_{g-1}\f{y}) = 
             \f{A}\left(\f{b} - \begin{bmatrix}
                \f{V}_{g-1} & \f{v}_g
            \end{bmatrix}\begin{bmatrix}
                \f{T}_{g-1}\\ \beta_{g}\f{e}_{g-1}^\top
            \end{bmatrix}\f{y}\right) \\
            & = \f{A}(\f{b} - (\f{V}_{g-1}\f{T}_{g-1} + \beta_{g}\f{v}_g\f{e}_{g-1}^\top)\f{y}) = \f{A}(\f{b} - (\f{V}_{g-1}\f{T}_{g-1} + \beta_{g}\f{v}_g\f{e}_{g-1}^\top)\f{T}_{g-1}^{-1}\|\f{b}\|\f{e}_1)\\
            & = \f{A}(\f{b} - (\f{b} + \beta_g\f{v}_g\f{e}_{g-1}^\top\f{T}_{g-1}^{-1}\|\f{b}\|\f{e}_1)) = - \beta_g\f{A}\f{v}_g\f{e}_{g-1}^\top\f{T}_{g-1}^{-1}\|\f{b}\|\f{e}_1.
        \end{align*}
        Recall the construction of the vector $\f{v}_g$ from the three-term recurrence of the Lanczos process,
        \begin{align*}
            \f{Av}_{g-1} - \beta_{g-1}\f{v}_{g-2} - \alpha_{g-1}\f{v}_{g-1} = \beta_{g}\f{v}_g.
        \end{align*}
        Since all the vectors on the left-hand side of the above equality are non-zero and linearly independent, we have $\f{v}_g^\top\f{Av}_{g-1} = \beta_{g} \not= 0$ by the orthogonality of the vectors $\f{v}_i$. Together, we obtain,
        \begin{align*}
            0 & = \f{v}_{g-1}^\top\f{0} = - \beta_g\f{v}_{g-1}^\top\f{A}\f{v}_g\f{e}_{g-1}^\top\f{T}_{g-1}^{-1}\|\f{b}\|\f{e}_1 = -\beta_g^2\|\f{b}\|\f{e}_{g-1}^\top\f{T}_{g-1}^{-1}\f{e}_1\\
            & = - \beta_g^2\|\f{b}\| \f{e}_{g-1}^\top\f{c} = - \beta_g^2\|\f{b}\| c_{g-1},
        \end{align*}
        where $\f{c} \triangleq \f{T}_{g-1}^{-1}\f{e}_1$. %
        Since $\beta_g^2\|\f{b}\| \neq 0$, we must have  $c_{g-1} = 0$. In this case, we have
        \begin{align*}
           \f{e}_1 = \f{T}_{g-1}\f{c} & = \begin{bmatrix}
                \alpha_1 & \beta_2 & & & \\
                \beta_2 & \alpha_2 & \beta_3 & & \\
                & & & \ddots & \beta_{g-1}\\
                & & & \beta_{g-1} & \alpha_{g-1}
            \end{bmatrix}\begin{bmatrix}
                c_1 \\ c_2 \\ \vdots \\ c_{g-1}
            \end{bmatrix}\\
            & = \begin{bmatrix}
                c_1\alpha_1 + c_2\beta_2\\
                \vdots\\
                c_{g-3}\beta_{g-2} + c_{g-2}\alpha_{g-2} + c_{g-1}\beta_{g-1}\\
                c_{g-2}\beta_{g-1} + c_{g-1}\alpha_{g-1}
            \end{bmatrix}= \begin{bmatrix}
                c_1\alpha_1 + c_2\beta_2\\
                \vdots\\
                c_{g-3}\beta_{g-2} + c_{g-2}\alpha_{g-2}\\
                c_{g-2}\beta_{g-1}
            \end{bmatrix}.
        \end{align*}
        By the construction of the Lanczos process, the constants $\beta_2, \cdots, \beta_{g-1}$ are non-zero. This implies that  $c_{g-2} = c_{g-3} = \cdots = c_1 = 0$, and hence $\f{c} = \f{0}$. However, since $\f{T}_{g-1}$ is assumed to be nonsingular, $\f{T}_{g-1}^{-1}\f{e}_1$ cannot be a zero vector, which leads to a contradiction.
        
        \item So now suppose $\f{T}_{g-1}$ is singular, which in particular implies that it has a nontrivial null-space. We need to show that $\f{e}_1 \notin \Range(\f{T}_{g-1})$. Suppose to the contrary that $\f{e}_1 \in \Range(\f{T}_{g-1})$, or equivalently, for all non-zero $\f{c} \in \Null(\f{T}_{g-1})$, we have $\langle \f{e}_1, \f{c} \rangle = c_1 = 0$. Take any non-zero $\f{c} \in \Null(\f{T}_{g-1})$. It follows that
        \begin{align*}
           \f{0} = \f{T}_{g-1}\f{c} = \begin{bmatrix}
                \alpha_1 & \beta_2 & & & \\
                \beta_2 & \alpha_2 & \beta_3 & & \\
                & & & \ddots & \beta_{g-1}\\
                & & & \beta_{g-1} & \alpha_{g-1}
            \end{bmatrix}\begin{bmatrix}
                c_1 \\ c_2 \\ \vdots \\ c_{g-1}
            \end{bmatrix} = \begin{bmatrix}
                c_1\alpha_1 + c_2\beta_2\\
                c_1\beta_2 + c_2\alpha_2 + c_3\beta_3\\
                \vdots\\
                c_{g-2}\beta_{g-1} + c_{g-1}\alpha_{g-1}
            \end{bmatrix}. 
        \end{align*}
        Again, by the constriction of the Lanczos process, $\beta_2, \cdots, \beta_{g-1}$ are non-zero constants, and hence if $\langle \f{e}_1, \f{c} \rangle = c_1 = 0$, then $c_2 = c_3 = \cdots = c_{g-1} = 0$. This implies $\f{c} = \f{0}$, contradicting the non-zero assumption on $\f{c}$.  Hence, it follows that there must exists a non-zero $\f{c} \in \Null(\f{T}_{g-1})$ such that $\langle \f{e}_1, \f{c} \rangle \neq 0$, i.e., $\f{e}_1 \notin \Range(\f{T}_{g-1})$.
        
        \end{itemize}
        The above argument shows that no method based on the OR framework for symmetric matrices under inconsistent systems can find a normal solution in the subspace $\krylov{A}{b}{g-1}$. 

        \item Now, we consider the latter case of $k=g$. By \cref{cond:cg_1}, $\f{x}_{g} = \f{V}_g\f{y}$, and by \cref{cond:cg_2},
        \begin{align*}
            \f{0} = \f{V}_g^{\top}\f{r}_g & = \f{V}_g^\top(\f{b} - \f{AV}_g\f{y}) = \|\f{b}\|\f{e}_1 - \f{V}_g^\top\f{AV}_g\f{y} = \|\f{b}\|\f{e}_1 - \f{T}_g\f{y}.
        \end{align*}
        However, by \cref{lem:singular_T_at_g}, $\f{T}_g$ is singular. So, a similar argument as above can establish that no vector $\f{y}$ can satisfy the above equality, i.e., the conditions \cref{cond:cg_1,cond:cg_2} cannot be simultaneously satisfied in the subspace $\krylov{A}{b}{g}$.
        \end{enumerate}
    \end{proof}
\end{document}